\documentclass[1p]{elsarticle}

\usepackage{lineno}
\usepackage[colorlinks,linkcolor={blue}]{hyperref}
\modulolinenumbers[5]

\usepackage{amsfonts}
\usepackage{graphicx}
\usepackage{amsmath}
\usepackage{graphicx}
\usepackage{amssymb, nicefrac}

\usepackage{mathtools}
\DeclarePairedDelimiter{\setof}{\{}{\}}

\renewcommand{\iff}{\text{if and only if}}

\usepackage{xcolor}
\definecolor{lightblue}{rgb}{0.8,0.8,1}  

\newtheorem{theorem}{Theorem}
\newtheorem{maintheorem}{Theorem}

\newtheorem{corollary}[theorem]{Corollary}

\newtheorem{definition}[theorem]{Definition}
\newtheorem{example}[theorem]{Example}

\newtheorem{lemma}[theorem]{Lemma}

\newtheorem{remark}[theorem]{Remark}

\newenvironment{proof}[1][Proof]{\textbf{#1.} }{\ \rule{0.5em}{0.5em}}

\newcommand{\correction}[2]{#2}
\newcommand{\cor}[1]{#1}

\newcommand{\comment}[1]{}

\newcommand{\mymarginpar}[1]{}

\makeatletter
\usepackage{tikz}
\newcommand*\circled[2][1.6]{\tikz[baseline=(char.base)]{
    \node[shape=circle, draw, inner sep=1pt,
        minimum height={\f@size*#1},] (char) {\vphantom{WAH1g}#2};}}

\makeatother

\begin{document}

\begin{frontmatter}
\title{Computer assisted proofs for transverse collision \\
and near collision orbits in the restricted three body problem}
%

\author{Maciej J. Capi\'nski\footnote{M. C. was partially supported by the NCN grants 2019/35/B/ST1/00655 and 2021/41/B/ST1/00407.}}
\ead{maciej.capinski@agh.edu.pl}
\address{AGH University of Science and Technology, al. Mickiewicza 30, 30-059 Krak\'ow, Poland}

\author{Shane Kepley}
\ead{s.kepley@vu.nl}
\address{Vrije Universiteit Amsterdam, 
De Boelelaan 1105, 1081 HV Amsterdam, Netherlands }

\author{J.D. Mireles James\footnote{J.D.M.J. was partially supported by 
NSF Grant DMS 1813501}}
\ead{jmirelesjameds@fau.edu}
\address{Florida Atlantic University, 777 Glades Road, Boca Raton, Florida, 33431}

\begin{abstract}
This paper considers two point boundary value problems for conservative systems 
defined in multiple coordinate systems, and develops a flexible a posteriori framework
for computer assisted existence proofs.  
Our framework is applied to the study collision and near collision orbits
in the circular restricted three body problem.  In this case the coordinate systems are the 
standard rotating coordinates, and the two Levi-Civita coordinate systems
regularizing collisions with each of the massive primaries.  The proposed framework is used 
to prove the existence of a number of orbits which have long been studied numerically 
in the celestial mechanics literature, but for which there are no existing analytical proofs 
at the mass and energy values considered here.  These include 
transverse ejection/collisions from one primary body to the other, 
Str\"{o}mgren's assymptotic periodic orbits
(transverse homoclinics for $L_{4,5}$), families of periodic orbits 
passing through collision, and 
orbits connecting $L_4$ to ejection or collision.  
\end{abstract}

\end{frontmatter}



\section{Introduction} \label{sec:intro}


\correction{comment 1}{
The present work develops computer assisted arguments for proving theorems about
collision and near collision orbits in conservative systems.  Using these arguments, 
we answer several open questions about the dynamics of the
planar circular restricted three body problem
(CRTBP), a simplified model of three body motion
popular since the pioneering work of Poincar\'{e} 
\cite{MR1194622,MR1194623,MR1194624}.
Our approach combines classical Levi-Civita  
regularization with a multiple shooting scheme 
for two point boundary value problems (BVPs)
describing orbits which begin and end on 
parameterized curves/symmetry sets
in an energy level.  After numerically computing an 
approximate solution to the BVP, we use a
Newton-Krawczyk theorem to establish the proof of the existence
of a true solution nearby.  The a posteriori
argument makes extensive use of validated Taylor 
integrators for vector fields and variational equations.  
See also Remark \ref{rem:puttingItAllTogether} for 
several additional remarks about the relationship between
the present work and the existing literature.
}

%

The PCRTBP, defined formally in Section \ref{sec:PCRTBP}, 
describes the motion of 
an infinitesimal particle like a satellite, asteroid, or comet 
moving in the field of two massive bodies which orbit their center 
of mass on Keplerian circles.  The massive bodies are called the primaries,   
and and one assumes that their orbits are not disturbed
by the addition of the massless particle.  Changing to a co-rotating frame
of reference results in autonomous equations of motion, and 
choosing normalized units of distance, mass, and time 
reduces the number of parameters in the problem to one: the mass ratio of 
the primaries.  The system has a single first integral referred to as
the Jacobi constant, usually written as $C$.

 It is important to remember that for systems with a conserved quantity, 
 periodic orbits occur in one parameter families 
 -- or tubes -- parameterized by ``energy'' or conserved quantity.  
 We note also that the CRTBP has an equilibrium solution, also called a 
 Lagrange point or libration point, called $L_4$ in the upper half plane 
 forming an equilateral triangle with the two 
 primaries. (Similarly, $L_5$ forms an equilateral triangle
 in the lower half plane).
 We are interested in the following questions for the CRTBP.

 \begin{itemize}
\item  \textbf{Q1:} \textit{Do there exist orbits of the 
infinitesimal body, which collide with 
one primary in forward time, and 
the other primary in backward time?}  We refer to such orbits as  
primary-to-primary ejection-collisions. 
\item \textbf{Q2:} \textit{Do there exist orbits of 
infinitesimal body
which are assymptotic to the $L_4$ in backward time, 
but which collide with a primary in forward time?}
(Or the reverse - from ejection to $L_4$).  We refer to these as 
$L_4$-to-collision orbits (or ejection-to-$L_4$ orbits). 
\item  \textbf{Q3:} \textit{Do there exist orbits 
of the infinitesimal body which are 
asymptotic in both forward and backward time to $L_4$? }
Such orbits are said to be homoclinic to $L_{4}$.
\item \textbf{Q4:} \textit{Do there exist 
tubes of large amplitude periodic orbits
for the infinitesimal body, which accumulate to an ejection-collision orbit with 
one of the primaries?} Such tubes are said to terminate 
at an ejection-collision orbit.
\item  \textbf{Q5:} \textit{Do there exist tubes of periodic orbits
for the infinitesimal body
which accumulate to a pair of ejection-collision orbits going
from one primary to the other and back?}.  Such 
tubes are said to terminate at a consecutive 
ejection-collision.  
\end{itemize}

The questions 1 and 4 are known to have affirmative answers in various 
perturbative situations, but they are open for many mass ratios and/or 
values of the Jacobi constant.
There is also numerical evidence suggesting the 
existence of $L_4$ homoclinic orbits, and
consecutive ejection-collisions. 
However, questions 2,3, and 5 are not perturbative, and 
until now they remain open.  
 We review the literature in more detail in Section \ref{sec:collisionLit}. 
 The following theorems, for non-perturbative 
mass and energy parameters of the planar CRTBP, 
constitute the main results of the present work. 


\correction{comment 9}{}
\correction{comment 3}{
\begin{theorem}\label{thm:main-thm-1}
Consider the PCRTBP with mass ratio $1/4$ and 
Jacobi constant $C = 3.2$.  There exist at least two 
transverse ejection-collision orbits which transit 
between the primary bodies.  
One of these is ejected from the large primary and collides with 
the smaller, and the other is ejected from the 
smaller primary and collides with the larger.
Both orbits make this transition in finite time as 
measured in the original/synodic coordinates.
(See page \pageref{thm:ejectionCollision} for the precise statement).
\end{theorem}
\begin{theorem}\label{thm:main-thm-2} Consider the PCRTBP with equal masses 
and Jacobi constant $C_{L_4} = 3$.
There exists at least one ejection-to-$L_4$ orbit, and at least one 
$L_4$-to-collision.  These orbits
take infinite (forward or backward) time to reach $L_4$.
 (See page \pageref{thm:CAP-L4-to-collision} for the precise statement.)
Analogous orbits exist for $L_5$ by symmetry. 
\end{theorem}
\begin{theorem}\label{thm:main-thm-3} Consider the PCRTBP with equal masses and 
Jacobi constant $C_{L_4} = 3$.
There exist at least three distinct transverse homoclinic orbits 
to $L_4$. (See page \pageref{thm:CAP-connections} for the precise statement). 
Analogous orbits exist for $L_5$ by symmetry considerations. 
These orbits take infinite time to accumulate to $L_4$.
As a corollary of transversality, (see also Remark \ref{rem:termination}) there exist
chaotic subsystems in some neighborhood of each homoclinic orbit.
\end{theorem}
\begin{theorem}\label{thm:main-thm-4} Consider the PCRTBP with Earth-Moon mass ratio.
There exists a one parameter family of periodic orbits which accumulate to 
an ejection-collision orbit originating from and terminating with the Earth.  The ejection collision
orbit has Jacobi constant $C \approx 1.434$, and   
has ``large amplitude'', in the sense that it passes 
near collision with the Moon.  This ejection-collision occurs in finite time
in synodic/unregularized coordinates.
(See page \pageref{th:CAP-Lyap} for the precise statement.) 
\end{theorem}
\begin{theorem} \label{thm:main-thm-5}Consider the PCRTBP with equal masses. 
There exists a family of periodic orbits which accumulate to a consecutive ejection-collision orbit involving both primaries.
Each of the ejection-collisions occurs in finite time in synodic/unregularized coordinates.
The Jacobi constant of the consecutive ejection-collision orbit is
$C \approx 2.06$.
(See page \pageref{thm:doubleCollision} for the precise statement.)
\end{theorem}
}

\correction{comment 4}{Each of the theorems
is interesting in its own right, as we elaborate on in the remarks below.
Nevertheless, we note that the mass ratios and energies in the theorems have
been chosen primarily to illustrate that our approach
can be applied in many different settings.
Similar theorems  could be proven at other parameter values, 
or in other problems involving collisions, 
using the methodology developed here. We also remark that our results
make no claims about global uniqueness.  There could be many other such 
orbits for the given parameter values.  However, due to the transversality, 
such orbits cannot be arbitrarily close to the orbits whose existence we prove.
}


\begin{remark}[Ballistic transport] \label{rem:ballisticTransport}
{\em
Theorem 1 establishes the existence of ballistic transport, 
or zero energy transfer, from one primary to the other in 
finite time (think of this as a ``free lunch''  trip between the primaries).  
In physical terms, ballistic transport allows debris 
to diffuse between a planet and it's moon, or between a 
star and one of its planets, using only the natural dynamics
of the system.  This phenomena is observed 
for example when Earth rocks, ejected into space
after a meteor strike, are later discovered on the Moon
\cite{earthMoonRock} (or vice versa). 
Martial applications of low energy Moon-to-Earth
transfer are discussed in  \cite{next100Years,theMoonIsHarsh}.
Mathematically rigorous existence proofs for primary-to-primary
ejection-collision orbits
have until now required both small mass ratio and high velocity -- that is,
negative enough Jacobi constant.  See \cite{MR682839}.

In a similar fashion, Theorem 2 establishes the existence of zero 
energy transfers involving $L_4$ and a primary, and could for example 
be used to design space missions which visit the triangular libration 
points.   
}
\end{remark}

\begin{remark}[Termination orbits] \label{rem:termination}
{\em
Theorems \ref{thm:main-thm-3}, \ref{thm:main-thm-4}, \ref{thm:main-thm-5} involve the termination of
tubes of periodic orbits.
\cor{
Indeed, a corollary of Theorem \ref{thm:main-thm-3} is that there are families/tubes of 
periodic orbits accumulating to each of our $L_4$ homoclinics.
This follows from a theorem of Henrard \cite{MR0365628}.
Another corollary is that, near each of the orbits of Theorem \ref{thm:main-thm-3},  
there is an invariant chaotic subsystem in the $L_4$ energy level.
This is due to a theorem of Devaney \cite{MR0442990}.
}

\cor{

Numerical evidence for the existence of $L_4$ homoclinics
in the equal mass CRTBP appears already in the work of Str\"{o}mgren
in the 1920's 
\cite{stromgrenMoulton,stromgrenRef}.
See \cite{szebehelyTriangularPoints,onMoulton_Szebehely,theoryOfOrbits}
for more discussion.
Such orbits were once called \textit{asymptotic periodic orbits}, in light of 
the fact that they are closed loops with infinite period.
Despite of the fact that they appeared in the literature more than a hundred
years ago, the present work provides -- to the the best of our knowlege -- the first
mathematically rigorous existence proof of transverse $L_4$ homoclinics in the CRTBP.

}

In Theorems \ref{thm:main-thm-4} and \ref{thm:main-thm-5}, we first prove the existence of 
the ejection-collision orbits, and then directly
establish the existence of one parameter families of periodic orbits
terminating at these ejection-collision orbits 
by an application of the implicit function theorem. 

Termination orbits have a long history in celestial mechanics, 
and are of fundamental importance in equivariant bifurcation 
theory.  We refer the interested reader to the discussion of ``Str\"{o}mgren's
termination principle'' in Chapter 9 of
\cite{theoryOfOrbits}, and to the works of 
\cite{MR1879221,MR2042173,MR2969866} on equivariant families
in the Hill three body and restricted three body problems.  See also 
the works of \cite{MR3007103,MR2821620} on global continuation families
in the restricted $N$-body problem.}
\end{remark}

\correction{comment 6}{
\begin{remark}[Final fate of the velocity variables] \label{rem:Chazy}
{\em
Chazy's 1922 paper \cite{MR1509241} proved an important classification result
describing the possible asymptotic behavior of the position variables of three 
body orbits defined for all time. In the context of the CRTBP, Chazy's result says that 
orbits are either hyperbolic (massless particle goes to infinity with non-zero final 
velocity), parabolic (massless particle goes to infinity with zero final velocity),
bounded (massless particle remains in a bounded region for all time), or 
oscillatory (the lim sup of the distance from the origin 
is infinite, while the lim inf is finite --
that is, the massless particle makes infinitely many excursions between 
neighbourhoods of the origin and infinity).
An analogous complete classification theorem for the velocity 
variables does not exist, however we note that our Theorems 2 and 3 and establish 
the existence of orbits with interesting asymptotic velocities.  For example 
the $L_4$ to collision orbits of Theorem 2 have zero asymptotic velocity and 
reach infinite velocity in forward
time (or vice versa), while the homoclinics of Theorem 3 have zero forward and 
backward asymptotic velocity. }
\end{remark}
}

\begin{remark}[Moulton's $L_4$ periodic orbits] \label{rem:Moulton}
{\em
The family of periodic orbits whose existence 
is established in Theorem 5 is of Moulton's $L_4$ type, 
in the sense of \cite{moultonBook}.  
That is, these are periodic orbits which when projected into the $(x,y)$
plane (i.e. the configuration space) have non-trivial winding about $L_4$.
See also Chapter 9 of \cite{theoryOfOrbits}, 
or the works of \cite{onMoulton_Szebehely,szebehelyTriangularPoints}
for a more complete discussion of the history 
(and controversy) surrounding Moulton's orbits.
The present work provides, to the best of our knowlege, 
the first proof that Moulton type 
$L_4$ periodic orbits exist.  
}
\end{remark}

\cor{The remainder of the paper is organized as follows.  
The next three subsections briefly discuss some literature
on regularization of collision, numerical computational methods, 
and computer assisted methods of proof in celestial mechanics. We conclude them with Remark \ref{rem:puttingItAllTogether}  which places our work in the context of the discussed works.
These sections can be skimmed by the reader who wishes to dive right 
into the mathematical setup, which is described in 
Section \ref{sec:problem}.  There we describe the problem setup 
in terms of an appropriate multiple shooting problem, and 
establish tools for solving it. In particular, we 
define the unfolding parameters which we use to isolate
transverse solutions in energy level sets and use this notion to formulate 
Theorem \ref{th:single-shooting} and Lemma \ref{lem:multiple-shooting-2} 
which we later use for our computer assisted proofs. 
The role of the unfolding parameter is to add a missing variable, which is needed for solving the problems by means of zero finding with a Newton method. The unfolding parameter is an artificial variable added to the equations. It is added though in a way that ensures that we can recover the solution of the original problem from the appended one.
(See Remark \ref{rem:unfolding} for more detailed comments.)\label{explanation-unfolding} }
 In Section \ref{sec:PCRTBP}
we describe the planar CRTBP and it's Levi-Civita regularization.
Sections \ref{sec:ejectionToCollision}, \ref{sec:L4_to_collision}, and 
\ref{sec:symmetric-orbits}
describe respectively the formulation of the multiple shooting 
problem for primary-to-primary ejection-collision orbits, 
$L_4$ to ejection/collision orbits, $L_4$ homoclinic orbits,
and periodic ejection-collision families.
Section \ref{sec:CAP} describes our computer assisted proof strategy and 
illustrates how this strategy is used to prove our main theorems.  
Some technical details are given in the appendices.  
The codes implementing the computer assisted proofs discussed in 
this paper are available at the homepage of the first author MC.

\section*{Literature review}
\subsection{Geometric approach to collision dynamics} \label{sec:collisionLit}

\correction{comment 5}{
Suppose one were to choose, more or less arbitrarily,
 an initial configuration of the gravitational $N$-body problem.
A fundamental  question is to ask  ``does 
this initial configuration lead to collision between two or more of the bodies
 in finite time?''
The question is delicate, and remains central to the theory
even after generations of serious study.  
 Saari for example has shown that the set of orbits which reach collision
 in finite time (the \textit{collision set})
has measure zero \cite{MR295648,MR321386}, so that 
$N$-body collisions are in some sense physically unlikely.   
On the other hand, results due to Kaloshin, Guardia, and Zhang
prove that the collision set  can be $\gamma$-dense in 
open sets \cite{MR3951693}.  Though this notion of density is 
technical, the result shows that the embedding of the collision
set may be topologically complicated.  
}

\cor{
One of the main tools for studying collisions
is to introduce coordinate transformations which 
regularize the singularities.
The virtue of a regularizing coordinate change,
from a geometric perspective, is that it transforms
the singularity set in the original coordinates into a 
nicer geometric object.   For example, after 
Levi-Civita regularization in the planar CRTBP, the singularity sets  
(restricted to a particular fixed energy level)
are transformed into circles \cite{MR1555161}.
We review the Levi-Civita coordinates 
for the CRTBP in Section  \ref{sec:PCRTBP},
 and refer the interested reader 
to Chapter 3 of \cite{theoryOfOrbits}, to 
the notes of \cite{cellettiCollisions,MR633766}, and to the works of 
\cite{MR562695,MR633766,MR359459,MR3069058,MR638060} 
for much more complete overview of the pre-McGehee literature on 
different regularization techniques.  
}

\cor{Advecting the regularized singularity set under the backward flow 
for a time $T$ leads to a smooth manifold of initial conditions whose orbits collide 
with the primary in forward time $T$ or less.  This is referred to as a
local collision manifold.  Running time backwards leads to 
a local ejection manifold.    Studying intersections between ejection 
and collision manifolds,
and their intersections with other invariant 
objects, provides invaluable insights.  } 


\cor{One of the first works to combine this geometric picture of 
collisions with techniques from the qualitative theory of dynamical systems
is the paper by McGehee \cite{MR359459}.  Here, a general method 
for regularizing singularities is developed and used to 
study triple collisions in an isocoleces three body problem. 
As an illustration of the power of the method, the author proves
the existence of an infinite set of initial conditions whose orbits achieve 
arbitrarily large velocities after near collision.  }

\cor{Building on these 
results, Devaney proved the existence of infinitely many
ejection-collision orbits in the same model, 
when one of the masses is small \cite{Devaney1980TripleCI}.  
Further insights, based on similar techniques, are found in the
works of Simo, ElBialy, 
and Lacomba and Losco, and Moeckel
\cite{MR640127, ElBialy:1989td,MR571374,MR571374,10.2307/24893242}.
Using similar methods, 
Alvarez-Ram\'{i}rez, Barrab\'{e}s, Medina, and Oll\'{e}
obtain numerical and analytical results for a related symmetric
collinear four-body problem in \cite{MR3880194}.
In \cite{MR638060}, Belbruno developed a new regularization technique
for the spatial CRTBP and used it to prove the existence of families of periodic orbits
which terminate at ejection-collision when the mass ratio is small enough.  
This is a perturbative analog of our Theorem \ref{thm:main-thm-4} (but in the spatial problem).}

\cor{The paper \cite{MR682839} by Llibre 
is especially relevant to the present study, as the author
establishes a number of 
theorems about ejection-collision orbits in the planar CRTBP.  
Taylor expansions for the local
ejection/collision manifolds
in Levi-Civita coordinates are given and used to show that
the local collision sets are homeomorphic to cylinders
for all values of energy and any mass ratio.  
Then, for mass ratio sufficiently small, the author proves that
the ejection/collision manifolds intersect
twice near the large primary.  This gives the existence 
of a pair of ejection-collision orbits which depart from and 
return to the large body.  
For Jacobi constant 
sufficiently negative and mass ratio sufficiently small, 
he also proves the existence of an ejection form the large body 
which collides with the small body in finite time.  This is a perturbative 
analogue of our Theorem \ref{thm:main-thm-1}.  We note that the 
large primary to small primary ejection-collisions in  \cite{MR682839} 
are ``fast'', in the sense that the relative velocity between the 
infinitesimal body and the large primary is never zero.  Compare this 
to the orbits of our Theorem \ref{thm:main-thm-1}, which twice attain zero relative velocity 
with respect to the large primary (the orbits make a ``loop''). }

\cor{A follow up paper by Lacomba and Llibre
\cite{MR949626} shows that the ejection-collision orbits of  \cite{MR682839} 
are transverse, and as a corollary the authors prove that the CRTBP has no 
$C^1$ extendable regular integrals. Heuristically speaking, this says that
 the Jacobi integral is the only conserved quantity in the CRTBP and hence the 
 system is not integrable.
Transversality is proven analytically for small values of the mass ratio in the CRTBP, 
and studied numerically for Hill's problem.  
In \cite{MR993819}, Delgado proves the transversality result for the Hill's problem, using 
a perturbative argument for $1/C$ small (large Jacobi constant).
We remark that the techniques developed in the present work
could be applied to the Hill problem
for non-perturbative values of $C$.  
We also mention the work of Pinyol  \cite{MR1342132},
which uses similar techniques to prove the 
existence of ejection-collision orbits for the elliptic CRTBP.}

\cor{A number of results for collision and near collision orbits have been established using 
KAM arguments in Levi-Civita coordinates.  
For example in \cite{MR967629} and for the planar
CRTBP, Chenciner and Llibre prove the existence
of invariant tori which intersect the regularized collision circle transversally.  
The argument works for $1/C$ small enough and for any mass ratio.
The dynamics on the invariant tori are conjugate to irrational rotation, 
so that their existence implies that there are infinitely many orbits which pass arbitrarily close
to collision infinitely many times.  Returning to the original coordinates,
the authors refer to these as punctured invariant tori (punctured by the collision set).
Punctured tori in an averaged four body problem (weakly coupled double Kepler
problem) are studied by F\'{e}joz in \cite{MR1849229}, and this work
is extended by the same author to the CRTBP, for a parameter
 regime where the system can be viewed as a perturbation of two uncoupled 
 Kepler problems \cite{MR1919782}.
See also the work of Zhao \cite{MR3417880} for a proof that there 
exists a positive measure set of punctured toi in the spatial CRTBP.}

\cor{In \cite{MR1805879}, Bolotin and Mackay 
use variational methods to prove the existence of 
chaotic collision and near collision dynamics in the planar CRTBP.
The argument studies some normally hyperbolic invariant 
manifolds whose stable/unstable manifolds, 
in regularized coordinates, intersect 
one another near the local ejection/collision manifolds.
The small parameter in this situation is the mass ratio, and the results hold 
in an explicit closed interval of energies.  
The same authors extend the result to the spatial CRTBP in \cite{MR2245344}, 
and Bolotin obtains the existence of chaotic near collision 
dynamics for the elliptic CRTBP in 
\cite{MR2331205}.}

\cor{A more constructive (non-variational) approach to studying chaotic 
collision and near collision dynamics is found in Font, Nunes, and Sim\'{o}
 \cite{MR1877971}. Here, the authors prove the existence of
 chaotic invariant sets containing orbits which make
 infinitely many near collisions 
 with the smaller body in the planar CRTBP.  Again, $\mu$ is taken 
 as the small parameter and the authors compute perturbative 
 expansions for some Poincar\'e maps in Levi-Civita coordinates.
 Using these expansions they directly prove the existence of horseshoe dynamics.  
Since the argument is constructive, they are also able to 
show, via careful numerical calculations, that the
expansions provide useful predictions for $\mu$ as large as $10^{-3}$.
In a follow up paper \cite{MR2475705}, the same authors 
numerically compute all the near collision periodic orbits
in a fixed energy level satisfying certain bounds on the return time, and 
present numerical evidence for the existence of chaotic dynamics between
these.}
 
\cor{In the paper  \cite{MR3693390}, Oll\'{e}, Rodr\'{i}guez, and Soler
introduce the notion of an $n$-ejection-collision orbit.  This is an orbit
which is ejected from larger primary body, and which makes an excursion 
where it achieves a relative maximum distance from the large primary 
 $n$-times before colliding with it: such orbits look like flowers with $n$ petals
 and the primary body at the center.  Finding $n$-ejection-collision orbits
necessitates studying the local ejection/collision manifolds 
at a greater distance from the regularized singularity set than in previous works.   
The authors also numerically 
study the bifurcation structure of these families for $1 \leq n \leq 10$
over a range of energies.     } 
 
\cor{In a follow-up paper \cite{MR4110029}, the same authors prove the 
 existence of four families of $n$-ejection-collision orbits for any value of $n \geq 1$,
 for $\mu$ small enough, and for values of the Jacobi constant sufficiently large.  
 The argument exploits an analytic solution of the variational equations in 
 a neighborhood of the regularized singularity set in the Levi-Civita coordinates. 
 They also perform large scale numerical calculations which suggest that the $n$-ejection-collision
 orbits persist at all values of the mass ratio, and for large ranges of the Jacobi constant.  
 Another paper by the same authors numerically studies the manifold of ejection 
 orbits, for both the large and small primary bodies, over the whole range of mass
 ratios and for a number of different values of the Jacobi constant \cite{MR4162341}.
The authors also propose a geometric mechanism for finding ejection orbits which 
transit from one primary to the neighborhood of the other.  More precisely, they 
numerically compute intersections between the ejection manifold
 manifold and the stable manifold of a periodic orbit in the $L_1$ Lyapunov 
family at the appropriate energy level, and study the resulting dynamics.  }
  
 \cor{In the paper \cite{tereOlleCollisions}, Seara, Oll\'{e}, Rodr\'{i}guez, and Soler 
 dramatically extend the results of \cite{MR4110029}.  First, they  
 show that the existence of an $n$-ejection-collision is equivalent to  
an orbit with $n$-zeros of the angular momentum: a scalar quantity.  Using this
advance they are able to remove the small mass condition, and prove that 
for either primary, at any value of the mass ratio, and for any $n \geq 1$, there 
exist four $n$-ejection-collision orbits.  The argument is based on an 
application of the implicit function theorem, with the $1/C$ as the small parameter.
The authors also make a detailed numerical study of this enlarged family of 
$n$-ejection-collision orbits, taking $\mu \to 1$.  Using insights obtained from these 
numerical explorations, they propose an analytical hypotheses which allows them to prove 
existence results for $n$-ejection-collision orbits 
 for the Hill problem, again for large enough energies.  
We remark that, since $n$-ejection-collision orbits can be formulated 
as solutions of two point boundary value problems
beginning and ending on the regularized collision circle, 
an interesting project could be to 
prove the existence of such orbits for smaller values 
of $C$ using the techniques developed in the present work. }

 %

\subsection{Numerical calculations, computational mathematics, and celestial mechanics} 
\label{sec:numerics}
\correction{comment 7}{
Computational and observational tools for predicting the 
motions of celestial bodies have roots in antiquity, so that even a  terse overview
is beyond the scope of the present study.
Nevertheless, we remark that the numerical methods like  
integration techniques for solving initial value problems, and bisection/Newton 
schemes for solving nonlinear systems of equations have been applied to 
the study of the CRTBP at least since G.H. Darwin's 1897 treatise on periodic orbits
\cite{MR1554890}. 
The reader interested in the history of pen-and-paper calculations for the  
CRTBP will find the 
work of Moulton's  group in Chicago, as well as Stromgren's group 
in Copenhagen, from the 1910's to the 1930's of great interest. 
Detailed discussion of their accomplishments
are found in  \cite{stromgrenMoulton,moultonBook}, and in Chapter 9 of 
Szebehely's Book \cite{theoryOfOrbits}.
}

\cor{The historic work of the mathematicians/human computers at the NACA, and 
subsequent NASA space agencies, had a profound effect on the shape of 
twentieth century affairs, as chronicled in a number of books and films. 
See for example  \cite{hidenHumanComputers,hiddenFigures,dorthyVaughn,nasaComputers,mcmastersPage}.
The ascension of digital computing and the dawn of the space race 
in the 1950's and 1960's led to an explosion of computational work in celestial mechanics.  
Again, the literature is vast and we refer to the books of 
Szebehely \cite{theoryOfOrbits}, Celletti and Perozzi \cite{alesanderaBook}, and Belbruno
\cite{MR2391999} for more thorough discussion of historical developments and
the surrounding literature.  }

\cor{In the context of the present work, it is important to discuss the 
idea of recasting transport problems into 
two point boundary value problems (BVPs).  The main idea is to 
project the boundary conditions for an orbit segment onto a 
representation of the local stable/unstable manifold of some invariant object
(both linear and higher order expansions of the stable/unstable manifolds
are in frequent use).  Then a homoclinic or heteroclinic connection is
reconceptualized as an orbit beginning on the unstable and terminating 
on the stable manifold, giving a clear example of a BVP.  }

\cor{The papers by Beyn, Friedman, Doedel, Kunin
\cite{MR618636,MR1068199,MR1007358,MR1205453,MR1456497}
lay the foundations for such BVP methods. 
 A BVP approach for computing periodic orbits in conservative systems is developed by 
Mu\~{n}oz-Almaraz, Freire, Gal\'{a}n, Doedel, and Vanderbauwhede in 
\cite{MR2003792}.  In particular, they introduce an unfolding parameter
for the periodic orbit problem, an idea we make extensive use of in Section \ref{sec:problem}.
Connections between periodic orbits are studied by 
Doedel, Kooi, Van Voorn, and Kuznetsov in 
\cite{MR2454068,MR2511084}, and 
Calleja, Doedel, Humphries, Lemus-Rodr\'{i}guez and Oldeman 
apply these techniques to the  CRTBP in
\cite{MR2989589}.}

\cor{BVP methods for computing connections between invariant objects
are central to the geometric approach to space mission design 
described in the four volume set of books by 
G\'{o}mez, Jorba, Sim\'{o}, and Masdemont
\cite{MR1867240,MR1881823,MR1878993,MR1875754}, and also in the 
book of Koon, Lo, Marsden, and Ross
\cite{MR1870302}.  A focused (and shorter) research paper describing the 
role of connecting orbits in the spatial CRTBP is found in the paper 
\cite{MR2086140} by G\'{o}mez, Koon, Lo, Marsden, Masdemont, and Ross,
and explicit discussion of the role of invariant manifolds in 
space missions which visit the moons of Jupiter
is found in the paper \cite{MR1884895}, by Koon, Marsden, Ross, and Lo.
A sophisticated multiple shooting scheme for computing families
of connecting orbits between periodic orbits in Lyapunov families is 
developed in  by Barrab\'{e}s, Mondelo, and Oll\'{e} in 
\cite{Barrabes:2009ve}, and extended by the same 
authors to the general Hamiltonian setting in \cite{Barrabes:2013ws}.
Recent extensions are found in the work of  
Kumar, Anderson, and de la Llave 
\cite{Kumar:2021vc,MR4361879} on connecting orbits
between invariant tori in periodic perturbations of the CRTBP, 
and by Barcelona, Haro, and Mondelo \cite{https://doi.org/10.48550/arxiv.2301.08526}
for studying families of connecting orbits between center manifolds.
A broad overview of numerical techniques for studying transport phenomena 
in $N$-body problems is found in the review by 
Dellnitz, Junge, Koon, Lekien,  Lo, Marsden,  Padberg, Preis, Ross, and Thiere
\cite{MR2136742}. }

\cor{We must insist that the references given in this subsection in no way 
constitute a complete list. 
Our aim is only to stress the importance of BVP techniques, and to
suggest their rich history of application in the celestial mechanics literature,
while possibly directing the reader to more definitive sources.   
}

\subsection{Computer assisted proof in celestial mechanics} \label{sec:capLit}
Constructive, computer assisted proofs are a valuable tool in 
celestial mechanics, as they facilitate the study $N$-body dynamics
far from any perturbative regime, and in the absence of any small parameters
or variational structure.   Computer assisted arguments usually begin with 
careful numerical calculation of some dynamically
interesting orbits.  From this starting point, one tries to construct
a posteriori arguments which show that there are true orbits with the desired 
dynamics nearby.  Invariant
objets like equilibria, periodic orbits, quasi-periodic solutions (invariant tori), 
local stable/unstable manifolds, and connecting orbits between 
invariant sets
can be either reformulated as solutions of appropriate functional equations,
or expressed in terms of topological/geometric conditions in 
certain regions of phase space.  Given a numerical candidate which 
approximately satisfies either the functional equation or the geometric 
conditions, fixed point or degree theoretical arguments 
are used to prove the existence of true solutions nearby.  
As an example, the Newton-Krawczyk Theorem \ref{thm:NK}
from Section \ref{sec:CAP} is a tool for
verifying the existence of a unique non-degenerate
zero of a nonlinear map given a good enough
approximate root.  The theorem is proved using the 
contraction mapping theorem, and the interested 
reader will find many similar theorems 
discussed in the references below.  

To get a sense of the power of computer assisted methods of
proof in celestial mechanics, 
we refer the reader to the  
works of Arioli, Barutello, Terracini, Kapela, Zgliczy\'{n}ski, 
Sim\'{o}, Burgos, 
Calleja, Garc\'{i}a-Azpeitia, Lessard, Mireles James, 
Walawska, and Wilczak
\cite{MR2112702,MR2259202,
MR2012847,MR2185163,MR3622273,
MR2312391,MR3896998,MR4208440,MR3923486}
on periodic orbits, the works of 
 Arioli, Wilczak, Zgliczy\'{n}ski, Capi\'{n}ski, Kepley, Mireles James,Galante, and Kaloshin
\cite{MR1947690,MR1961956,MR3032848,MR3906230,MR2824484}
on transverse connecting orbits and chaos,  
the works of Capi\'{n}ski, Guardia, Mart\'{i}n, Sera, Zgliczy\'{n}ski, Roldan, Wodka, and Gidea
 \cite{oscillations,capinski_roldan,diffusionCRTBP,maciejMarianDiffusion}
on oscillations to infinity, center manifolds, and Arnold diffusion, 
and to the works of Celletti, Chierchia, de la Llave, Rana, Figueras, Haro, Luque, 
Gabern, Jorba, Caracciolo, and Locatelli
\cite{MR1101365,MR1101369,alexCAPKAM,MR2150352,MR4128817}
on quasi-periodic orbits and KAM phenomena.  We remark that while this  
list is by no means complete, consulting these papers and the 
papers cited therein will give the reader a reasonable impression of the
state-of-the-art in this area.  More general references on computer 
assisted proofs in dynamical systems
and differential equations are found in the review articles
by   Lanford, Koch, Schenkel, Wittwer, van den Berg, Lessard, 
and G\'{o}mez-Serrano \cite{MR759197,MR1420838,jpjbReview,MR3990999},
and in the books by Tucker, and Nakao, Plum, and Watanabe
\cite{MR2807595,MR3971222}.

We also mention the recent review article by Kapela, Mrozek, Wilczak, and 
 Zgliczy\'{n}ski \cite{CAPD_paper}, which describes the use of the  
CAPD library for validated numerical integration of ODEs 
and their variational equations.  
The CAPD library is a general purpose toolkit, and can
be applied to any problem where explicit formulas for  
the vector field are known in closed form. We make extensive use
of this library throughout the present work.
Additional details about CAPD algorithms 
are found in  the papers by Zgliczy\'{n}ski, and Wilczak
\cite{MR1930946,cnLohner}, but the reader who is interested in 
the historical development of these ideas should consult the references
of  \cite{CAPD_paper}.  Methods for computing validated enclosures of 
stable/unstable manifolds attached to equilibrium solutions 
for some restricted three and four body problems
are discussed in \cite{MR3906230,MR3792792}. and these methods are 
used freely in the sequel.

\begin{remark}[Relevance of the present work] \label{rem:puttingItAllTogether}
{\em
\correction{comment 2}{
In light of the discussion contained in 
Sections \ref{sec:collisionLit}, \ref{sec:numerics}, and \ref{sec:capLit}
a few, somewhat more refined comments about the novelty of the present work
are in order. First, note that our   
results make essential use of the geometric formulation of 
collision dynamics discussed in Section \ref{sec:collisionLit}.
An important difference of perspective is that, since we work without small parameters,
we formulate multiple shooting problems describing
orbits with boundary conditions on the regularized collision set, 
rather than working with perturbative expansions for the local ejection/collision 
manifolds and studying intersections between them.  Once we have a 
good enough numerical approximation of the solution of a BVP we validate 
the existence of a true solution via a standard a posteriori argument.    
}

\cor{
In this sense our work exploits the BVP approach for studying 
dynamical objects discussed in Section \ref{sec:numerics}.  
Our shooting templates, discussed (see Section \ref{sec:problem}),
are general enough to allow for any number of coordinate swaps in 
any order, and allow us to shoot from stable/unstable manifolds, 
regularized collision sets, or discrete symmetry subspaces.
In this way we have one BVP framework which covers all the theorems
considered in this paper.  We note that the setup 
applies also to higher dimensional problems involving stable/unstable 
manifolds of other geometric objects, as seen for example in 
\cite{jayAndMaxime}. 
Our setup incorporates an unfolding parameter 
approach to BVPs for conservative systems -- as discussed in 
\cite{MR1870260,MR2003792,MR1992054} for periodic orbits.
An important feature of our setup is that 
the existence of a non-degenerate solution of the BVP implies transversality 
relative to the energy submanifold. 
One virtue of the abstract framework presented in Section 
\ref{sec:problem} is that we prove transversality results and properties
of the unfolding parameter only once, and they apply to all problems
considered later in the text -- rather than having to establish 
such results for each new problem considered.  }

\cor{Still, this shooting template framework is just a convenience.  
The main contribution of the present work is a flexible computer
assisted approach to proving theorems about collision dynamics
in celestial mechanics problems.
We remark that, until now, collisions have been viewed 
largely as impediments to the implementation of successful computer assisted proofs.
The present work demonstrates that well known tools from regularization 
theory can be combined with existing validated numerical tools and 
a posteriori analysis to prove interesting theorems about collisions
in non-perturbative settings. As applications, we prove a number of 
new results for the CRTBP.}
}
\end{remark}

\color{black}

\section{Problem setup} 

\label{sec:problem}

Consider an ODE with one or more first integrals or constants of motion.
For such systems, the level sets of the integrals 
give rise to invariant sets.  Indeed, the level sets are invariant 
manifolds except at critical points of the conserved quantities.  
In this section we describe a shooting method for two point
boundary value problems between submanifolds of the level set. To be more
precise, we consider two manifolds, parameterized (locally) by some
functions, which are contained in a level set.  We present a method which
allows us to find points on these manifold which are linked by a solution of
an ODE. This in particular implies that the two manifolds intersect. Our
method will allow us to establish transversality of the intersection within
the level set.

We consider an ODE%
\begin{equation}
x^{\prime}=f\left( x\right) ,  \label{eq:ode-1}
\end{equation}
where $f:\mathbb{R}^{d}\rightarrow\mathbb{R}^{d}$. Assume that the flow $%
\phi\left( x,t\right) $ induced by (\ref{eq:ode-1}) has an integral of
motion expressed as%
\begin{equation*}
E:\mathbb{R}^{d}\rightarrow\mathbb{R}^{k},
\end{equation*}
which means that 
\begin{equation}
E\left( \phi\left( x,t\right) \right) =E\left( x\right) ,
\label{eq:E-integral}
\end{equation}
for every $x\in\mathbb{R}^{d}$ and $t\in \mathbb{R}$. Fix $c\in\mathbb{R%
}^{k}$ and define the level set
\begin{equation}
M :=\left\{x \in \mathbb{R}^d : E(x)=c\right\} ,  \label{eq:M-level-set}
\end{equation}
and assume that $M$ is (except possibly at some degenerate points) a smooth
manifold. Consider two open sets $D_{1}\subset\mathbb{R}^{d_{1}}$ and $%
D_{2}\subset\mathbb{R}^{d_{2}}$ and two chart maps
\begin{equation}
P_{i}: D_{i}\rightarrow M\subset\mathbb{R}^{d}\qquad\text{for }i=1,2,
\label{eq:Pi-intro}
\end{equation}
parameterizing submanifolds of $M$.
\begin{remark}
One can for example think of the $P_{1}$ and $P_{2}$ as parameterizations of
the exit or entrance sets on some 
local unstable and stable manifolds, respectively, of some invariant object.
However in some of the applications to follow $P_{1,2}$ will parameterize
collision sets in regularized coordinates or some surfaces of symmetry for $%
f $.
\end{remark}
We seek points $\bar{x}_{i}\in D_{i}$ for $i=1,2$ and a time $\bar{\tau}\in%
\mathbb{R}$ such that%
\begin{equation}
\phi\left( P_{1}(\bar{x}_{1}),\bar{\tau}\right) =P_{2}\left( \bar{x}%
_{2}\right) .  \label{eq:problem-1}
\end{equation}
Note that if $P_1$ and $P_2$ parameterize some $\phi$-invariant manifolds, then
Equation \eqref{eq:problem-1} implies that these manifolds intersect. The
setup is depicted in Figure \ref{fig:setup}.

\begin{figure}[ptb]
\begin{center}
\includegraphics[height=4cm]{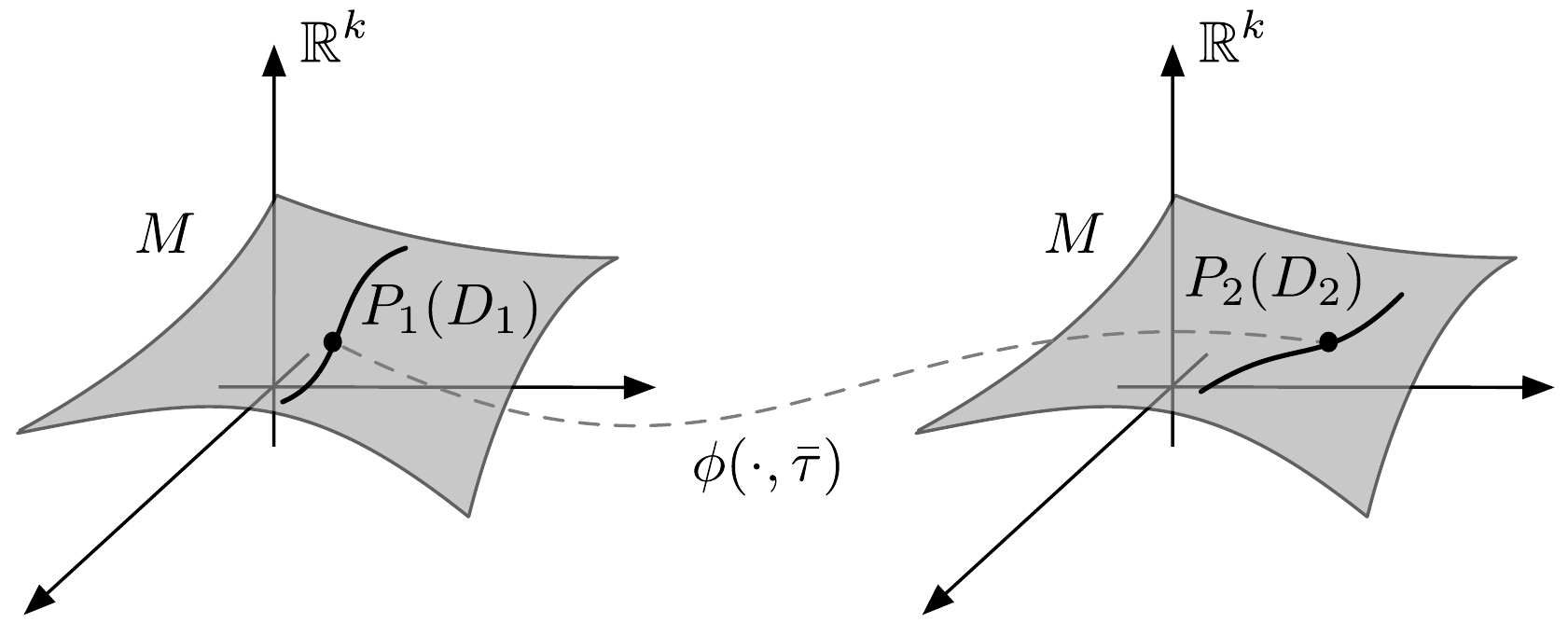}
\end{center}
\caption{The left and right plots are in $\mathbb{R}^{d}$ with a $d-k$
dimensional manifold $M$ depicted in gray. The manifolds $%
P_{i}(D_{i})\subset M$, for $i=1,2$, are represented by curves inside of $M$. We seek $\bar x_{1} \in D_1, \bar x_{2} \in D_2$ and $\bar{\tau} \in \mathbb{R}$ such
that $\protect\phi(P_{1}(\bar x_{1}),\bar{\tau})=P_{2}(\bar x_{2})$.
The two points $P_{i}(\bar x_{i})$, for $i=1,2$, are represented by dots.}
\label{fig:setup}
\end{figure}

\begin{remark}
Denote by $x_{1},x_{2}$ the points $x_{1}\in\mathbb{R}^{d_{1}}$ and by $%
x_{2}\in\mathbb{R}^{d_{2}}$: this avoids confusion with $x\in\mathbb{R}^{d}$. 
\end{remark}

We introduce a general scheme which allows us to:

\begin{enumerate}
\item Establish the intersection of the manifolds parameterized by $P_{1}$
and $P_{2}$ by means of a suitable Newton operator.

\item Establish that the intersection is transverse relative to the level
set $M$.

\item Provide a setup flexible enough for multiple shooting between charts in different coordinates.
\end{enumerate}

Our methodology is applied to establish connections between stable/unstable
and collision manifolds in the PCRTBP.

\subsection{Level set shooting}

We now provide a more detailed formulation of problem (\ref{eq:problem-1}) 
which allows us to describe connections between multiple level sets 
in distinct coordinate systems (instead of just one 
coordinate system as discussed in Section \ref{eq:M-level-set}).
This allows us to study applications to collision dynamics 
as boundary value problems joining points in different coordinate systems.
Let  $U_1, U_2 \subset \mathbb{R}^{d}$  
be open sets and consider smooth
functions $E_{1},E_{2}$%
\begin{equation*}
E_{i}:U_{i}\rightarrow\mathbb{R}^{k}\qquad\text{for }i=1,2,
\end{equation*}
for which $DE_{i}\left( x\right) $ is of rank $k$ for every $x\in U_{i}$,
for $i=1,2.$ We fix $c_{1},c_{2}\in\mathbb{R}^{k}$ and define the following
the level sets 
\begin{equation*}
M_{i}=\left\{ x\in U_{i}:E_{i}\left( x\right) =c_{i}\right\} \qquad\text{for 
}i=1,2,
\end{equation*}
and assume that $M_{i}\neq\emptyset$ for $i=1,2$. Observe that the
$M_i$ are smooth $d-k$ dimensional 
manifolds by the assumption that $DE_{i}$ are of rank $k$, for $i=1,2$.


Consider now a smooth function 
$R:U_{1}\times\mathbb{R\times\mathbb{R}}^{k}\rightarrow\mathbb{R}^{d}$ 
We introduce the following notation for coordinates%
\begin{equation*}
\left( x,\tau,\alpha\right) \in\mathbb{R}^{d}\times\mathbb{R\times \mathbb{R}%
}^{k},\qquad y\in\mathbb{R}^{d},
\end{equation*}
and define a parameter dependent family of maps $R_{\tau
,\alpha}:U_{1}\rightarrow\mathbb{R}^{d}$ by%
\begin{equation*}
R_{\tau,\alpha}\left( x\right) :=R\left( x,\tau,\alpha\right) ,
\end{equation*}
and assume that for each $(x, \tau, \alpha) \in \mathbb{R}^{d + k + 1}$, the
$d \times d$ matrix 
\begin{equation*}
\frac{\partial}{\partial x} R(x, \tau, \alpha),
\end{equation*}
is invertible, so that $R_{\tau, \alpha}(x)$ is a local diffeomorphism
on $\mathbb{R}^d$.

The following definition makes precise our assumptions about when 
$R_{\tau, \alpha}(x)$ takes values in $M_2$.

\begin{definition}
\label{def:unfolding}We say that $\alpha $ is an unfolding parameter for $R$ 
if the following two conditions are satisfied for every $x \in M_1$.
\begin{enumerate}
\item If $R_{\tau, \alpha}(x) \in M_2$, then $\alpha = 0$. 
\item If $R_{\tau, 0}(x) \in U_2$, then $R_{\tau, 0}(x) \in M_2$. 
\end{enumerate}
\end{definition}

\medskip

To emphasize that we are interested in points mapped 
from $M_{1}$ to $M_{2}$, we say that $\alpha$ is an unfolding
parameter for $R$\emph{\ from }$M_{1}$\emph{\ to }$M_{2}$. 

Assume from now on that $\alpha$ is an unfolding parameter for $R$. We
consider two open sets $D_{1}\subset \mathbb{R}^{d_{1}}$ and $D_{2}\subset%
\mathbb{R}^{d_{2}}$ where $d_{1},d_{2}\in\mathbb{N}$ and two smooth
functions 
\begin{equation*}
P_{i}:D_{i}\rightarrow M_{i},\qquad\text{for }i=1,2,
\end{equation*}
each of which is a diffeomorphism onto its image. Define%
\begin{equation*}
F:  D_{1}\times D_{2}\times\mathbb{R}\times%
\mathbb{R}^{k}\rightarrow\mathbb{\mathbb{R}}^{d}
\end{equation*}
by the formula
\begin{equation}
F\left( x_{1},x_{2},\tau,\alpha\right) :=R_{\tau,\alpha}\left( P_{1}\left(
x_{1}\right) \right) -P_{2}\left( x_{2}\right) .  \label{eq:F-def-1-shooting}
\end{equation}
We require that
\begin{equation}
d_{1}+d_{2}+1+k=d,  \label{eq:dimensions}
\end{equation}
and seek $\bar{x}_{1},\bar{x}_{2},\bar{\tau}$ such that%
\begin{equation} 
\label{eq:zero}
F\left( \bar{x}_{1},\bar{x}_{2},\bar{\tau}, 0\right) = R_{\bar{\tau }%
, 0}\left( P_{1}\left( \bar{x}_{1}\right) \right) -P_{2}\left( \bar{x}
_{2}\right) =0, 
\end{equation}
with $DF\left( \bar{x}_{1},\bar{x}_{2},\bar{\tau},0\right)$  an isomorphism. 
In fact, we do more than simply solve (\ref{eq:zero}). For some open interval $I \subset\mathbb{R}$ containing $\bar{\tau}$ we establish a transverse
intersection between the smooth manifolds $R\left( P_{1}\left( D_{1}\right)
,I,0\right) $ and $P_{2}\left( D_{2}\right) $ at $\bar{y} := P_{2}\left( \bar{x}%
_{2}\right) \in M_{2}$.  

\medskip

\begin{remark}[Role of the unfolding parameter] \label{rem:unfolding}
{\em
\correction{comment 8}{
The setup above, and in particular the roles of the
parameters $\alpha$ and $\tau$, might first appear puzzling. In
the applications we have in mind, $\tau$ is the time associated with the flow map
of an ODE. The unfolding parameter $\alpha$ deals with the fact that we
solve a problem restricted to the level sets $M_i$
for $i=1,2$.  Consider for example shooting from a 1D arc to 
a 1D arc in a 4D conservative vector field, where the arcs are 
in the same level set of the conserved quantity
(think of the arc as an outflowing segment on the boundary of a
local 2D stable/unstable manifold, or part of the collision set).
Since we are working in a 3D level set, the 2D surfaces formed
by advecting the arcs can intersect transversally relative to the level 
set.  However, since the arcs are parameterized by one variable functions, and
the time of flight $\tau$ is unknown, taking into account the dimension
of the vector field, we have 4 equations in 3 unknowns.  
Adding an unfolding parameter to the problem balances the equations, 
but it must be done carefully, so that the 
new variable does not change the solution set for 
the problem.  The idea was first exploited for periodic orbits of 
conservative systems in \cite{MR2003792,MR1870260}.
We adapt the idea to the shooting templates developed for the present work, 
and this is the purpose of the variable $\alpha$.
}

\cor{An alternative formulation would be to fix the energy and use 
its formula to eliminate one of the variables in the equations of motion, 
or to  work with coordinates in which we can write $M_i$ 
as graphs of some functions and use these functions and
appropriate projections to enforce the constraints. 
Another possibility is to throw away one of the equations in the 
BVP formulation when applying Newton, and to check a-posteriori that 
this equation is satisfied  \cite{BDLM,MR3919451}.
Yet another approach would be to 
directly apply a Newton scheme for the unbalanced BVP, exploiting  
the Moore-Penrose psudoinverse at each step.  While such approaches
lead to excelent numerical methods, one encounters difficulties 
when translating them into computer assisted arguments. 
We believe that the unfolding parameter is a good solution in 
this setting, as it leads to balanced equations with isolated solutions
suitable for verification using fixed point theorems.}} 
\end{remark}

We now
give an example which informs the intuition.

\begin{example}
\label{ex:motivating}(Canonical unfolding.) Consider the ODE in Equation \eqref{eq:ode-1}%
and $E:\mathbb{R}^{d}\rightarrow\mathbb{R}$ satisfying Equation \eqref{eq:E-integral}. Suppose $c\in\mathbb{R}$ is fixed
and denote its associated level set by $M := \left\{ E=c\right\}$ (In this example we have $k=1$ and $E_{1}=E_{2}=E$%
.) Assume there are smooth functions $P_{1},P_{2}$ as in \eqref{eq:Pi-intro}
and that $d_{1}+d_{2}+2=d$. 
We construct a shooting operator for Equation \eqref{eq:problem-1} by choosing $R$ as follows. 
Consider the $\alpha$-parameterized family of ODEs 
\begin{equation*}
x^{\prime}=f(x)+\alpha\nabla E\left( x\right).
\end{equation*}
Let $\phi_{\alpha}\left( x,t\right)$ denote the induced flow and note that 
$\phi_{0}=\phi$ is the flow induced by Equation \eqref{eq:ode-1}. Defining the shooting 
operator by the formula
\begin{equation}
R\left( x,\tau,\alpha\right) :=\phi_{\alpha}\left( x,\tau\right),
\label{eq:R-alpha}
\end{equation}
we see that solving Equation \eqref{eq:problem-1} 
is equivalent to solving Equation \eqref{eq:zero}. 

Observe that $\alpha$ is unfolding for $R$ because $E$ is an
integral of motion for $\phi$ from which it follows that 
\begin{align*}
	\frac{d}{dt}E\left( R_{\tau,\alpha}\left( x\right) \right) & = \frac{d}{dt}
	E(\phi_{\alpha}(x,t)) \\
	& =\nabla E\left( \phi_{\alpha}\left( x,t\right) \right) \cdot\left( f(\phi
	_{\alpha}\left( x,t\right) )+\alpha\nabla E\left( \phi_{\alpha}\left(
	x,t\right) \right) \right) \\
	& =\alpha\left\Vert \nabla E\left( \phi_{\alpha}\left( x,t\right) \right)
	\right\Vert ^{2},
\end{align*}
where $\cdot$ denotes the standard scalar product. Here we have used the fact that
Equation \eqref{eq:E-integral} implies $\nabla E\left( x\right) \cdot f(x)=0$ but also $\nabla E(\phi_\alpha(x,t)) \neq 0$ since $\nabla E$ is assumed to have rank $1$ everywhere.

\end{example}
Returning to the general setup we have the following theorem.
\begin{theorem}
\label{th:single-shooting}Assume that $\alpha$ is an unfolding parameter for 
$R$ and $F$ is defined as in Equation \eqref {eq:F-def-1-shooting}. If 
\begin{equation}
F\left( \bar{x}_{1},\bar{x}_{2},\bar{\tau},\bar{\alpha}\right) =0,
\label{eq:F-zero-in-thm1}
\end{equation}
then $\bar{\alpha}=0$. Moreover, if $DF\left( \bar{x}_{1},\bar{x}_{2},\bar{%
\tau},0\right) $ is an isomorphism, then there exists an open interval $I\subset%
\mathbb{R}$ of $\bar{\tau}$ such that the manifolds $R\left( P_{1}\left( D_{1}\right)
,I,0\right) $ and $P_{2}\left( D_{2}\right) $ intersect transversally in $%
M_{2}$ at $\bar{y}:=P_{2}\left( \bar{x}_{2}\right) $. Specifically, we have the splitting
\begin{equation}
T_{\bar{y}}R\left( P_{1}\left( D_{1}\right) ,I,0\right) \oplus T_{\bar{y}%
}P_{2}\left( D_{2}\right) =T_{\bar{y}}M_{2},
\label{eq:th1-transversality}
\end{equation}
and moreover, $\bar{y}$ is an isolated transverse point.

\end{theorem}

\begin{proof}
Recalling the definition of $F$ in Equation \eqref{eq:F-def-1-shooting}
and the hypothesis of Equation \eqref{eq:F-zero-in-thm1}, 
we have that $\bar{x}=P_{1}\left( 
\bar {x}_{1}\right) \in M_{1}$ and $\bar{y}=P_{2}\left( \bar{x}_{2}\right)
\in M_{2}$.   The fact that $\alpha$ is an unfolding parameter for $R$,
combined with $R\left( \bar{x},\bar{\tau},\bar{\alpha}\right) =\bar{y}$,
implies that $\bar{\alpha}=0$.
Since $F(\bar x_1,\bar x_2, \bar \tau, 0)=0$, we see that $R(P_1(D_1),I,0)$
and $P_2(D_2)$ intersect at $\bar y$.%

Our hypotheses on $P_{1,2}$ and $R$ imply that  
$R\left( P_{1}\left( D_{1}\right) ,I,0\right) $ and $P_{2}\left(
D_{2}\right) $ are submanifolds 
of $M_{2}$ so evidently
\begin{equation*}
T_{\bar{y}}R\left( P_{1}\left( D_{1}\right) ,I,0\right) \oplus T_{\bar{y}%
}P_{2}\left( D_{2}\right) \subset T_{\bar{y}}M_{2}.
\end{equation*}
However, from the assumption in Equation \eqref{eq:dimensions} we have $%
d-k=d_{1}+d_{2}+1$ and therefore it suffices to prove that $T_{\bar{y}}R\left( P_{1}\left(
D_{1}\right) ,I,0\right) \oplus T_{\bar{y}}P_{2}\left( D_{2}\right) $ is $d-k$ dimensional.

Suppose $\setof*{e_{1},\ldots ,e_{d_{1}}}$ is a basis for $\mathbb{R}^{d_{1}}$ and $\setof*{\tilde{e}_{1},\ldots ,\tilde{e}%
_{d_{2}}}$ is a basis for $\mathbb{R}^{d_{2}}$. Define%
\begin{align*}
v_{i}& :=\frac{\partial R}{\partial x_{1}}\left( \bar{x}_{1},\bar{\tau}%
,0\right) DP_{1}\left( \bar{x}_{1}\right) e_{i}\qquad \text{for }i=1,\ldots
,d_{1} \\
v_{i}& :=DP_{2}\left( \bar{x}_{2}\right) \tilde{e}_{i-d_{1}}\qquad \text{for 
}i=d_{1}+1,\ldots ,d_{1}+d_{2} \\
v_{d_{1}+d_{2}+1}& :=\frac{\partial R}{\partial \tau }\left( \bar{x}_{1},%
\bar{\tau},0\right) .
\end{align*}%
After differentiating Equation \eqref{eq:F-def-1-shooting} we obtain the formula
\begin{equation*}
DF=\left( 
\begin{array}{cccc}
\frac{\partial F}{\partial x_{1}} & \frac{\partial F}{\partial x_{2}} & 
\frac{\partial F}{\partial \tau } & \frac{\partial F}{\partial \alpha }%
\end{array}%
\right) =\left( 
\begin{array}{cccc}
\frac{\partial R}{\partial x_{1}}DP_{1} & -DP_{2} & \frac{\partial R}{%
\partial \tau } & \frac{\partial R}{\partial \alpha }%
\end{array}%
\right),
\end{equation*}%
and since $DF$ is an isomorphism at $\left( \bar{x}_{1},\bar{x}_{1},\bar{\tau%
},0\right) $, it follows that the vectors $v_{1},\ldots ,v_{d_{1}+d_{2}+1}$ span a $%
d_{1}+d_{2}+1=d-k$ dimensional space.
Observe that  
\begin{align*}
T_{\bar{y}}R\left( P(D_1) , I, 0\right) & =\text{span}\left(
v_{1},\ldots,v_{d_{1}},v_{d_{1}+d_{2}+1}\right) , \\
T_{\bar{y}} P_{2} \left(D_2 \right) & =\text{span}\left(
v_{d_{1}+1},\ldots,v_{d_{1}+d_{2}}\right),
\end{align*}
proving the claim in Equation \eqref{eq:th1-transversality}. 
Moreover, since 
\begin{equation*}
\dim R\left( P_{1}\left( D_{1}\right) ,I,0\right) +\dim P_{2}\left(
D_{2}\right) =\left( d_{1}+1\right) +d_{2}=d-k=\dim M_{2},
\end{equation*}
 it follows that $\bar{y}$ is an isolated transverse intersection point which concludes
the proof.
\end{proof}

We finish this section by defining an especially simple ``dissipative'' unfolding 
parameter which works in the setting of the PCRTBP.

\begin{example}
\label{ex:dissipative-unfolding}(Dissipative unfolding.) Let $x,y\in \mathbb{%
R}^{2k}$, let $\Omega:\mathbb{R}^{2k}\rightarrow\mathbb{R}$ and $J\in\mathbb{%
R}^{2k\times2k}$ be of the form%
\begin{equation*}
J=\left( 
\begin{array}{cc}
0 & \operatorname{Id}_{k} \\ 
-\operatorname{Id}_{k} & 0%
\end{array}
\right) ,
\end{equation*}
where $\operatorname{Id}_{k}$ is a $k\times k$ identity matrix. Let us consider an ODE of
the form 
\begin{equation*}
\left(x', y'\right) = f\left( x,y\right) :=\left( y,2Jy+\frac {\partial%
}{\partial x}\Omega\left( x\right) \right) .
\end{equation*}
One can check that $E\left( x,y\right) =-\left\Vert y\right\Vert ^{2}+2\Omega\left(
x\right) $ is an integral of motion. Consider the parameterized family of ODEs%
\begin{equation}
\left(x',y'\right) = f_{\alpha}\left( x,y\right) :=f\left( x,y\right)
+\left( 0,\alpha y\right) ,  \label{eq:dissipative-vect-alpha}
\end{equation}
and let $\phi_{\alpha}\left( \left( x,y\right) ,t\right) $ denote the flow
induced by Equation \eqref{eq:dissipative-vect-alpha}. Define 
the shooting operator defined by
\begin{equation}
R\left( \left( x,y\right) ,\tau,\alpha\right) :=\phi_{\alpha}\left( \left(
x,y\right) ,\tau\right).  \label{eq:R-alpha-dissipative}
\end{equation}
As in Example \ref{ex:motivating}, one can check the 
equivalence between Equations \eqref{eq:problem-1} 
and \eqref{eq:zero}. The fact that $\alpha$ is unfolding for $R$ follows as%
\begin{equation*}
	\frac{d}{dt}E\left( \phi_{\alpha}\left( \left( x,y\right) ,t\right) \right)
	=-2\alpha\left\Vert y\right\Vert ^{2}.
\end{equation*}
\end{example}


\subsection{Level set multiple shooting\label{sec:multiple-shooting}}
Consider a sequence of open sets $U_{1},\ldots,U_{n}\subset\mathbb{R}^{d}$
and a sequence of smooth maps 
\begin{equation*}
E_{i}:U_{i}\rightarrow\mathbb{R}^{k}\qquad\text{for }i=1,\ldots,n
\end{equation*}
for which $DE_{i}\left( x\right) $ is of rank $k$ for every $x\in U_{i}$,
for $i=1,\ldots,n$. Let $c_{1},\ldots,c_{n}\in\mathbb{R}^{k}$ be a fixed
sequence with corresponding level sets
\begin{equation*}
M_{i}:=\left\{ x\in U_{i}:E_{i}\left( x\right) =c_{i}\right\} \qquad\text{%
for }i=1,\ldots,n.
\end{equation*}
Let%
\begin{equation*}
R^{i}:U_{i}\times\mathbb{R}\times\mathbb{R}^{k}\rightarrow\mathbb{R}%
^{d}\qquad\text{for }i=1,\ldots,n-1
\end{equation*}
be a sequence of smooth functions which defines a sequence of parameter
dependent maps 
\begin{align*}
R_{\tau,\alpha}^{i} & :U_{i}\rightarrow\mathbb{R}^{d}, \\
R_{\tau,\alpha}^{i}\left( x\right) & :=R^{i}\left( x,\tau,\alpha\right)
,\qquad\text{for }i=1,\ldots,n-1.
\end{align*}
We assume that for each fixed $\tau$ and $\alpha$, each of the maps is a
local diffeomorphism on $\mathbb{R}^d$.

Let $D_{0} \subset \mathbb{R}^{d_{0}}$ and $D_{n} \subset \mathbb{R}^{d_{n}}$ be
open sets, and let 
\begin{equation*}
P_{0}:D_{0}\rightarrow M_0\subset \mathbb{R}^{d},\qquad\qquad
P_{n}:D_{n}\rightarrow M_n\subset \mathbb{R}^{d},
\end{equation*}
be diffeomorphisms onto their image. Assume that%
\begin{equation}
d_{0}+d_{n}+1+k=d  \label{eq:dimensions-multiple-shooting}
\end{equation}%
and consider the function%
\begin{equation*}
\tilde{F}:\mathbb{R}^{nd}\supset D_{0}\times \underset{n-1}{\underbrace{%
\mathbb{R}^{d}\times \ldots \times \mathbb{R}^{d}}}\times D_{n}\times 
\mathbb{R}\times \mathbb{R}^{k}\rightarrow \underset{n}{\underbrace{\mathbb{R%
}^{d}\times \ldots \times \mathbb{R}^{d}}},
\end{equation*}%
defined by the formula
\begin{equation}
\tilde{F}\left( x_{0},\ldots ,x_{n},\tau ,\alpha \right) =\left( 
\begin{array}{r@{\,\,\,\,}l}
P_{0}\left( x_{0}\right)  & -\,\,\,x_{1} \\ 
R_{\tau ,\alpha }^{1}\left( x_{1}\right)  & -\,\,\,x_{2} \\ 
  & \, \vdots \\ 
R_{\tau ,\alpha }^{n-2}\left( x_{n-2}\right)  & -\,\,\,x_{n-1} \\ 
R_{\tau ,\alpha }^{n-1}\left( x_{n-1}\right)  & -\,\,\,P_{n}\left(
x_{n}\right) 
\end{array}%
\right)   \label{eq:multi-prob}
\end{equation}
We now define the following functions 
\begin{align*}
R& :U_{1}\times \mathbb{R}\times \mathbb{R}%
^{k}\rightarrow \mathbb{R}^{d}, \\
F& :  D_{0}\times
D_{n}\times \mathbb{R}\times \mathbb{R}^{k}\rightarrow \mathbb{R}^{d}
\end{align*}%
by the formulas
\begin{align}
R\left( x_{1},\tau ,\alpha \right) & =R_{\tau ,\alpha }\left( x_{1}\right)
:=R_{\tau ,\alpha }^{n-1}\circ \ldots \circ R_{\tau ,\alpha }^{1}\left(
x_{1}\right) ,  \notag \\
F\left( x_{0},x_{n},\tau ,\alpha \right) & :=R_{\tau ,\alpha }\left(
P_{0}\left( x_{0}\right) \right) -P_{n}\left( x_{n}\right) .
\label{eq:F-parallel}
\end{align}

\begin{definition}
We say that $\alpha $ is an unfolding parameter for the sequence $R_{\tau
,\alpha }^{i}$ if it is unfolding for $R_{\tau ,\alpha }=R_{\tau ,\alpha
}^{n-1}\circ \ldots \circ R_{\tau ,\alpha }^{1}.$
\end{definition}
We now formulate the following lemma.
\begin{lemma}
\label{lem:multiple-shooting-2}If $\tilde{F}\left( \bar{x}_{0},\ldots,\bar {x%
}_{n},\bar{\tau},\bar{\alpha}\right) =0$ and $D\tilde{F}\left( \bar{x}%
_{0},\ldots,\bar{x}_{n},\bar{\tau},\bar{\alpha}\right) $ is an isomorphism,
then $F\left( \bar{x}_{0},\bar{x}_{n},\bar{\tau},\bar{\alpha}\right) =0$ and 
$DF\left( \bar{x}_{0},\bar{x}_{n},\bar{\tau},\bar{\alpha}\right) $ is an
isomorphism.
\end{lemma}

\begin{proof}
The fact that $F\left( \bar{x}_{0},\bar{x}_{n},\bar{\tau},\bar{\alpha }%
\right) =0$ follows directly from the way $\tilde{F}$ and $F$ are defined in Equations
\eqref{eq:multi-prob} and \eqref{eq:F-parallel} respectively. Before
proving that $DF$ is an isomorphism, we set up some notation. We will write%
\begin{equation*}
dR^{i}:=\frac{\partial R^{i}}{\partial x_{i}}\left( \bar{x}_{i},\bar{\tau },%
\bar{\alpha}\right)\qquad\text{for }i=1,\ldots,n-1.
\end{equation*}
It will be convenient for us to swap the order of the coordinates, so we
define%
\begin{equation}
\hat{F}\left( x_{1},\ldots, x_{n},x_{0},\tau,\alpha\right) :=\tilde{F}\left(
x_{0},x_{1},\ldots, x_{n},\tau,\alpha\right),  \label{eq:F-reordered}
\end{equation}
and write%
\begin{equation*}
\hat{F}=\left( \hat{F}_{1},\ldots,\hat{F}_{n}\right) \qquad\text{where\qquad 
}\hat{F}_{i}:\mathbb{R}^{nd}\rightarrow\mathbb{R}^{d},\text{ for }%
i=1,\ldots, n.
\end{equation*}
Finally, the last notation we introduce is $z\in\mathbb{R}^{d}$ to combine
the coordinates from the domain of $F$ together 
\begin{equation*}
z=\left( z_{1},\ldots,z_{d}\right) =\left( x_{n},x_{0},\tau,\alpha\right) \in%
\mathbb{R}^{d_{n}}\times\mathbb{R}^{d_{0}}\times\mathbb{R}\times \mathbb{R}%
^{k}=\mathbb{R}^{d}.
\end{equation*}
Note that $z$ is also the variable corresponding to the last $d$ coordinates from
the domain of $\hat F$ (see Equation \eqref{eq:F-reordered}). Finally, we remark that all derivatives considered in the argument below are computed at the point $(\bar{x}
_{0},\ldots,\bar{x}_{n},\bar{\tau},\bar{\alpha})$.

 With the above notation we see that%
\begin{equation*}
D\hat{F}=\left( 
\begin{array}{ccccc}
-\operatorname{Id} & 0 & \cdots & 0 & \frac{\partial\hat{F}_{1}}{\partial z} \\ 
dR^{1} & -\operatorname{Id} & \ddots & \vdots & \frac{\partial\hat{F}_{2}}{\partial z} \\ 
0 & \ddots & \ddots & 0 & \vdots \\ 
\vdots & \ddots & dR^{n-2} & -\operatorname{Id} & \frac{\partial\hat{F}_{n-1}}{\partial z}
\\ 
0 & \cdots & 0 & dR^{n-1} & \frac{\partial\hat{F}_{n}}{\partial z}%
\end{array}
\right),
\end{equation*}
and $D\hat{F}$ is an isomorphism since $D\tilde{F}$ is an isomorphism. To see this define a sequence of vectors $v^{1},\ldots,v^{d}\in\mathbb{R}%
^{nd}$ of the form%
\begin{equation*}
v^{i}=\left( 
\begin{array}{c}
v_{1}^{i} \\ 
\vdots \\ 
v_{n}^{i}%
\end{array}
\right) \in\mathbb{R}^{d}\times\ldots\times\mathbb{R}^{d}=\mathbb{R}%
^{nd}\qquad\text{for }i=1,\ldots,d,
\end{equation*}
with $v_{1}^{i}$,$v_{n}^{i}\in\mathbb{R}^{d}$ chosen as 
\begin{equation}
v_{1}^{i}=\frac{\partial\hat{F}_{1}}{\partial z_{i}},\qquad\qquad
v_{n}^{i}=\left( 
\begin{array}{ccccccc}
0 & \cdots & 0 & \overset{i}{1} & 0 & \cdots & 0%
\end{array}
\right) ^{\top},  \label{eq:vin-choice}
\end{equation}
and $v_{2}^{i},\ldots,v_{n-1}^{i}\in\mathbb{R}^{d}$ defined inductively as%
\begin{equation}
v_{k}^{i}=dR^{k-1}v_{k-1}^{i}+\frac{\partial\hat{F}_{k}}{\partial z_{i}}\quad%
\text{for }k=2,\ldots,n-1.  \label{eq:v-ik}
\end{equation}
Note that from the choice of $v_{n}^{i}$ in (\ref{eq:vin-choice}) the
vectors $v^{1},\ldots,v^{d}$ are linearly independent.

By direct computation\footnote{%
From (\ref{eq:multi-prob}) and (\ref{eq:v-ik}) follow the cancellations when
multiplying the vector $v^i$ by $D\hat F$.} it follows that%
\begin{equation}
D\hat{F}v^{i}=\left( 
\begin{array}{c}
0 \\ 
dR^{n-1}v_{n-1}^{i}+\frac{\partial\hat{F}_{n}}{\partial z_{i}}%
\end{array}
\right) \qquad\text{for }i=1,\ldots,d,  \label{eq:proof-shooting-0}
\end{equation}
where the zero is in $\mathbb{R}^{\left( n-1\right) d}$.

Looking at (\ref{eq:multi-prob}), since $\hat{F}_{1},\ldots \hat{F}_{n-1}$
do not depend on $x_{n}$, we see that for $i\in \left\{ 1,\ldots
,d_{n}\right\} $ we have $\frac{\partial \hat{F}_{1}}{\partial z_{i}}=\ldots
=\frac{\partial \hat{F}_{n-1}}{\partial z_{i}}=0,$ so 
\begin{eqnarray}
dR^{n-1}v_{n-1}^{i}+\frac{\partial \hat{F}_{n}}{\partial z_{i}}
&=&dR^{n-1}\left( dR^{n-2}v_{n-2}^{i}+\frac{\partial \hat{F}_{n-1}}{\partial
z_{i}}\right) -\frac{\partial P_{n}}{\partial x_{n,i}}
\label{eq:proof-shooting-1} \\
&=&dR^{n-1}\left( dR^{n-2}v_{n-2}^{i}+0\right) -\frac{\partial P_{n}}{%
\partial x_{n,i}}  \notag \\
&=&\cdots  \notag \\
&=&dR^{n-1}\ldots dR^{1}v_{1}^{i}-\frac{\partial P_{n}}{\partial x_{n,i}} 
\notag \\
&=&dR^{n-1}\ldots dR^{1}\frac{\partial \hat{F}_{1}}{\partial z_{i}}-\frac{%
\partial P_{n}}{\partial x_{n,i}}  \notag \\
&=&-\frac{\partial P_{n}}{\partial x_{n,i}}\qquad \text{for }i=1,\ldots
,d_{n}.  \notag
\end{eqnarray}

Similarly, for $j=i-d_{n}\in \left\{ 1,\ldots ,d_{0}\right\} $ from (\ref%
{eq:multi-prob}) we see that $\frac{\partial \hat{F}_{1}}{\partial z_{i}}=%
\frac{\partial P_{0}}{\partial x_{0,j}}$ and $\frac{\partial \hat{F}_{2}}{%
\partial z_{i}}=\ldots =\frac{\partial \hat{F}_{n}}{\partial z_{i}}=0$, so 
\begin{align}
dR^{n-1}v_{n-1}^{i}+\frac{\partial \hat{F}_{n}}{\partial z_{i}}&
=dR^{n-1}dR^{n-2}\ldots dR^{1}\frac{\partial P_{0}}{\partial x_{0,j}}=\frac{%
\partial \left( R_{\bar{\tau},\bar{\alpha}}\circ P_{0}\right) }{\partial
x_{0,j}}  \label{eq:proof-shooting-2} \\
& \qquad \qquad \qquad \qquad \qquad \left. \text{for }i=d_{n}+1,\ldots
,d_{n}+d_{0}.\right.  \notag
\end{align}

The index $i=d_{n}+d_{0}+1$ corresponds to $\tau $. Similarly to (\ref%
{eq:proof-shooting-1}), by inductively applying the chain rule, it follows
that 
\begin{equation}
dR^{n-1}v_{n-1}^{i}+\frac{\partial \hat{F}_{n}}{\partial z_{i}}=\frac{%
\partial R}{\partial \tau }\qquad \text{for }i=d_{n}+d_{0}+1.
\label{eq:proof-shooting-3}
\end{equation}

Finally, for $j=i-d_{n}-d_{0}-1\in \left\{ 1,\ldots ,k\right\} $, the variable $z_i$ corresponds to $\alpha_j$, and also by
applying the chain rule we obtain that%
\begin{equation}
dR^{n-1}v_{n-1}^{i}+\frac{\partial \hat{F}_{n}}{\partial z_{i}}=\frac{%
\partial R}{\partial \alpha _{j}}\qquad \text{for }i=d_{n}+d_{0}+2,\ldots ,d.
\label{eq:proof-shooting-4}
\end{equation}

Combining Equations \eqref{eq:proof-shooting-0}--\eqref{eq:proof-shooting-4} we see that%
\begin{equation}
\left( 
\begin{array}{ccc}
D\hat{F}v^{1} & \cdots & D\hat{F}v^{d}%
\end{array}%
\right) =\left( 
\begin{array}{cccc}
0 & 0 & 0 & 0 \\ 
-\frac{\partial P_{n}}{\partial x_{n}} & \frac{\partial \left( R_{\bar{\tau},%
\bar{\alpha}}\circ P_{0}\right) }{\partial x_{0}} & \frac{\partial R}{%
\partial \tau } & \frac{\partial R}{\partial \alpha }%
\end{array}%
\right) . \label{eq:DF-multi-final}
\end{equation}%
Since $v^{1},\ldots ,v^{d}$ are linearly independent and since $D\hat{F}$ is
an isomorphism, the rank of the above matrix is $d$. Looking at Equation \eqref{eq:multi-prob} we see that the lower part of the matrix in Equation \eqref{eq:DF-multi-final} corresponds to $DF$ which implies that $DF$ is of rank $d$, hence is an isomorphism.
\end{proof}

We see that we can validate assumptions of Theorem \ref{th:single-shooting}
by setting up a multiple shooting problem (\ref{eq:multi-prob}) and applying
Lemma \ref{lem:multiple-shooting-2}. To do so, one needs to additionally
check whether $\alpha $ is an unfolding parameter for the sequence $R_{\tau
,\alpha }^{i}.$



\section{Regularization of collisions in the PCRTBP} \label{sec:PCRTBP}

\label{sec:CAPS} In this section we formally introduce the equations 
of motion for the PCRTBP as discussed in Section \ref{sec:intro}.
\correction{comment 11}{We stress that all the material in this section,
and in Subsections \ref{sec:reg1} and \ref{sec:regSecondPrimary} are 
completely standard, and we follow the normalization conventions 
as in \cite{theoryOfOrbits}.  That being said, since we use this material to implement
computer assisted proofs, it is important to explicitly state every formula correctly and 
to recall some important but well known facts. The reader who is familiar with 
the CRTBP and its Levi-Civita regularization may want to skim this section 
while jumping ahead to Section \ref{sec:ejectionToCollision}.}

Recall that the problem describes a three body system, where two massive 
primaries are on circular orbits about their center of mass, and
a third massless particle moves in their field.   
The equations of motion for the massless particle 
are expressed in a co-rotating 
frame with the frequency of the primaries. 
Writing Newton's laws in the co-rotating frame leads to 
\begin{align}
x^{\prime \prime }& =2y^{\prime }+\partial _{x}\Omega (x,y),
\label{eq:NewtonPCRTBP} \\
y^{\prime \prime }& =-2x^{\prime }+\partial _{y}\Omega (x,y),  \notag
\end{align}%
where 
\begin{equation*}
\Omega (x,y)=(1-\mu )\left( \frac{r_{1}^{2}}{2}+\frac{1}{r_{1}}\right) +\mu
\left( \frac{r_{2}^{2}}{2}+\frac{1}{r_{2}}\right) ,
\end{equation*}%
\begin{equation*}
r_{1}^2=(x-\mu )^{2}+y^{2},\quad \quad \text{and}\quad \quad
r_{2}^2=(x+1-\mu )^{2}+y^{2}.
\end{equation*}%
Here $x,y$ are the positions of the massless particle on the plane. The $\mu 
$ and $1-\mu $ are the masses of the primaries (normalized so that the
total mass of the system is $1$). 
The rotating frame is oriented so that the
primaries lie on the $x$-axis, with the center of mass at the origin. 
We take $\mu \in (0, \frac{1}{2}]$ so that the large body
is always to the right of the origin.
The larger primary has mass $m_{1}=1-\mu $ and is located at the position
$(\mu ,0)$. Similarly the smaller primary has mass $m_{2}=\mu $ 
and is located at position $(\mu -1,0)$.
The top frame of Figure \ref{fig:PCRTBP_coordinates} provides a schematic
for the positioning of the primaries and the massless particle.

\begin{figure}[t]
\begin{center}
\includegraphics[height=7.5cm]{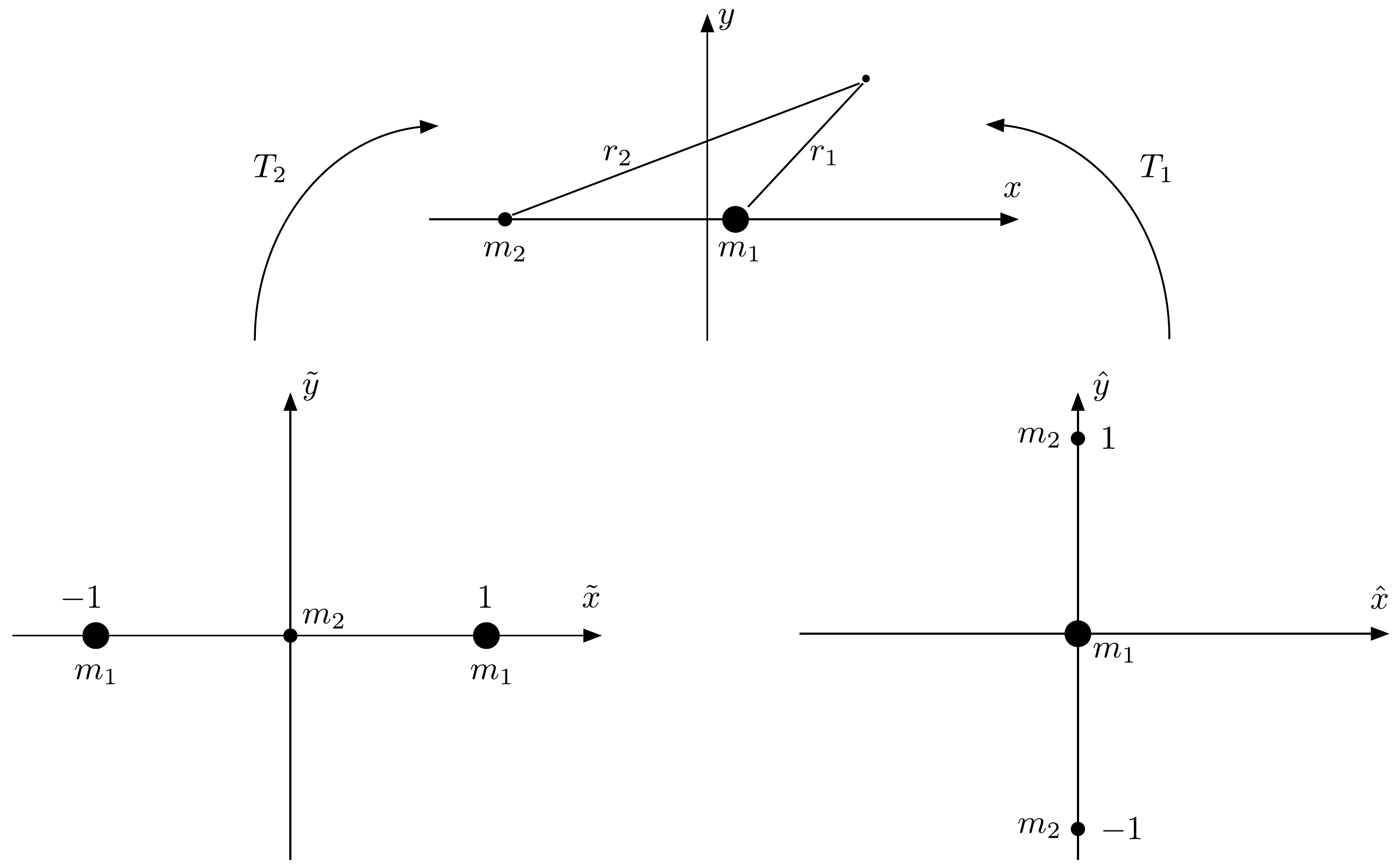}
\end{center}
\caption{Three coordinate frames for the PCRTBP: the center top image
depicts the classical PCRTBP in the rotating frame. The bottom left and
right frames depict the restricted three body problem in Levi-Civita
coordinates: regularization of collisions with $m_{2}$ on the left and with $%
m_{1}$ on the right. Observe that in these coordinates the regularized body
has been moved to the origin. The Levi-Civita transformations $T_{1}$ and $%
T_{2}$ provide double covers of the original system, so that in the
regularized frames there are singularities at the two copies of the
remaining body. }
\label{fig:PCRTBP_coordinates}
\end{figure}

Let $U\subset \mathbb{R}^{4}$ denote the open set 
\begin{equation*}
U:=\left\{ (x,p,y,q)\in \mathbb{R}^{4}\,|\,\left( x,y\right) \not\in \{\left(
\mu ,0\right) ,\left( \mu -1,0\right) \}\right\} .
\end{equation*}%
The vector field $f\colon U\rightarrow \mathbb{R}^{4}$ defined by 
\begin{equation}
f(x,p,y,q):=\left( 
\begin{array}{c}
p \\ 
2q+x-\frac{(1-\mu )\left( x-\mu \right) }{((x-\mu )^{2}+y^{2})^{3/2}}-\frac{%
\mu \left( x+1-\mu \right) }{((x+1-\mu )^{2}+y^{2})^{3/2}} \\ 
q \\ 
-2p+y-\frac{(1-\mu )y}{((x-\mu )^{2}+y^{2})^{3/2}}-\frac{\mu y}{((x+1-\mu
)^{2}+y^{2})^{3/2}}%
\end{array}%
\right)  \label{eq:PCRTBP}
\end{equation}%
is equivalent to the second order system given in \eqref{eq:NewtonPCRTBP}.
Note that 
\begin{equation*}
\Vert f(x,p,y,q)\Vert \rightarrow \infty \quad \quad \text{as either}\quad
\quad (x,y)\rightarrow (\mu ,0)\quad \text{or}\quad (x,y)\rightarrow (\mu
-1,0).
\end{equation*}
Let $\mathbf{x}=(x,p,y,q)$ denote the coordinates in $U$ and 
denote by $\phi (\mathbf{x},t)$ the flow generated by $f$ on $U$.
The system (\ref{eq:PCRTBP}) has an integral of motion $E\colon U\rightarrow 
\mathbb{R}$ given by
\begin{equation}
E\left( \mathbf{x}\right) =-p^{2}-q^{2}+2\Omega (x,y),
\label{eq:JacobiIntegral}
\end{equation}%
which is refered to as the Jacobi integral.

We are interested in orbits with initial conditions $\mathbf{x} \in  U$
with the property that their positions limit to either $m_1 :=(\mu, 0)$ 
or $m_2 := (\mu -1, 0)$ in finite time.  Such orbits, which reach a 
singularity of the vector field
$f$ in finite time, are called collisions.
It has long been known that if we fix our attention to a specific 
level set of the Jacobi integral
for some fixed $c\in \mathbb{R}$, then it is possible to 
make a change of coordinates which ``removes'' or regularizes the 
singularities.  
This idea is reviewed in the next sections.

\subsection{Regularization of collisions with $m_{1}$} \label{sec:reg1}

To regularize a collision with $m_{1}$, define the complex variables $%
z=x+iy, $ and the new ``regularized''
variables $\hat{z}=\hat{x}+i\hat{y},$ related to $z$ by the transformation 
\begin{equation*}
\hat{z}^{2}=z-\mu .
\end{equation*}%
One also rescales time in the
regularized coordinates with the rescaled time $\hat{t}$ related to the original time $t$ by the formula
\begin{equation*}
\frac{dt}{d\hat{t}}=4|\hat{z}|^{2}.
\end{equation*}
Let $U_{1}\in \mathbb{R}^{4}$ denote the open set 
\begin{equation*}
U_{1}=\left\{ \mathbf{\hat{x}}=(\hat{x},\hat{p},\hat{y},\hat{q})\in \mathbb{R%
}^{4} : \left( \hat{x},\hat{y}\right) \notin \left\{ \left( 0,-1\right)
,\left( 0,1\right) \right\} \right\} .
\end{equation*}%
This set will be the domain of the regularized vector field which allows us to 
``flow through'' collisions with $m_1$ but not
with $m_{2}$.

A lengthy calculation (see \cite{theoryOfOrbits}), applying the change of
coordinates and time rescaling just described to the vector field $f$
defined in Equation \eqref{eq:PCRTBP} leads to the regularized Levi-Civita vector
field $f_{1}^{c}\colon U_{1}\rightarrow \mathbb{R}^{4}$ with the ODE $%
\mathbf{\hat{x}}^{\prime }=f_{1}^{c}\left( \mathbf{\hat{x}}\right) $ given by%
\begin{eqnarray} 
\hat{x}^{\prime } &=&\hat{p}, \notag \\
\hat{p}^{\prime } &=&8\left( \hat{x}^{2}+\hat{y}^{2}\right) \hat{q}+12\hat{x}%
(\hat{x}^{2}+\hat{y}^{2})^{2}+16\mu \hat{x}^{3}+4(\mu -c)\hat{x} \notag\\
&&+\frac{8\mu (\hat{x}^{3}-3\hat{x}\hat{y}^{2}+\hat{x})}{((\hat{x}^{2}+\hat{y%
}^{2})^{2}+1+2(\hat{x}^{2}-\hat{y}^{2}))^{3/2}}, \notag\\
\hat{y}^{\prime } &=&\hat{q}, \label{eq:regularizedSystem_m1} \\
\hat{q}^{\prime } &=&-8\left( \hat{x}^{2}+\hat{y}^{2}\right) \hat{p}+12\hat{v%
}\left( \hat{x}^{2}+\hat{y}^{2}\right) ^{2}-16\mu \hat{y}^{3}+4\left( \mu
-c\right) \hat{y} \notag \\
&&+\frac{8\mu (-\hat{y}^{3}+3\hat{x}^{2}\hat{y}+\hat{y})}{((\hat{x}^{2}+\hat{%
y}^{2})^{2}+1+2(\hat{x}^{2}-\hat{y}^{2}))^{3/2}}, \notag
\end{eqnarray}
where the parameter $c$ in the above ODE is $c=E(x,p,y,q)$.
The main observation is that the regularized vector field is well defined at
the origin $\left( \hat{x},\hat{y}\right) =\left( 0,0\right) $, and that the
origin maps to the collision with $m_{1}$ when we invert the Levi-Civita
coordinate transformation.

Let $\psi _{1}^{c}(\mathbf{\hat{x}},\hat{t})$ denote the flow generated by $%
f_{1}^c$. The flow conserves the first integral $E_{1}^{c}\colon
U_{1}\rightarrow \mathbb{R}$ given by 
\begin{eqnarray}
E_{1}^{c}(\mathbf{\hat{x}}) &=&-\hat{q}^{2}-\hat{p}^{2}+4(\hat{x}^{2}+\hat{y}%
^{2})^{3}+8\mu (\hat{x}^{4}-\hat{y}^{4})+4(\mu -c)(\hat{x}^{2}+\hat{y}^{2}) 
\notag \\
&&+8(1-\mu )+8\mu \frac{(\hat{x}^{2}+\hat{y}^{2})}{\sqrt{(\hat{x}^{2}+\hat{y}%
^{2})^{2}+1+2(\hat{x}^{2}-\hat{y}^{2})}}.  \label{eq:reg_P_energy}
\end{eqnarray}
Note that the parameter $c$ appears both in the formulae for $f_{1}^{c}$ and 
$E_{1}^{c}$.  We write $\psi _{1}^{c}$ to stress that the flow depends 
explicitly on the choice of $c$.  We choose 
$c \in \mathbb{R}$ and then, after regularization, have new coordinates 
which allow us to study collisions only in the level set 
	\begin{equation}
		M:=\left\{ \mathbf{x}\in U : E(\mathbf{x})=c\right\} .
		\label{eq:M-level-set-c}
\end{equation}

We define the linear subspace $\mathcal{C}_{1}\subset \mathbb{R}^{4}$ by 
\begin{equation*}
\mathcal{C}_{1}=\left\{ (\hat{x},\hat{p},\hat{y},\hat{q})\in \mathbb{R}%
^{4}\,|\,\hat{x}=\hat{y}=0\right\} ,
\end{equation*}%
The change of coordinates between the two coordinate systems is given by the
transform $T_{1}\colon U_{1}\backslash \mathcal{C}_{1}\rightarrow U$, 
\begin{equation}
\mathbf{x}=T_{1}(\mathbf{\hat{x}}):=\left( 
\begin{array}{c}
\hat{x}^{2}-\hat{y}^{2}+\mu  \\ 
\frac{\hat{x}\hat{p}-\hat{y}\hat{q}}{2(\hat{x}^{2}+\hat{y}^{2})} \\ 
2\hat{x}\hat{y} \\ 
\frac{\hat{y}\hat{p}+\hat{x}\hat{q}}{2(\hat{x}^{2}+\hat{y}^{2})}%
\end{array}%
\right) ,  \label{eq:T1-def}
\end{equation}%
and is a local diffeomorphism on $U_{1}\backslash \mathcal{C}_{1}$. The
following theorem collects results from 
\cite{theoryOfOrbits}, and relates the dynamics of the original and the regularized
systems.  

\begin{theorem}
\label{thm:LeviCivitta}
Let $c$ be the fixed parameter
determining the level set $M$ in Equation \eqref{eq:M-level-set-c}. Assume that $%
\mathbf{x}_{0}\in U$ satisfies $E(\mathbf{x}_{0})=c,$ and assume that $%
\mathbf{\hat{x}}_{0}\in U_{1}\setminus \mathcal{C}_{1}$ is such that $%
\mathbf{x}_{0}=T_{1}\left( \mathbf{\hat{x}}_{0}\right) $. Then the curve%
\begin{equation*}
\gamma \left( s\right) :=T_{1}\left( \psi _{1}^{c}(\hat{\mathbf{x}}%
_{0},s)\right)
\end{equation*}%
parameterizes the following possible solutions of the PCRTBP in $M$:

\begin{enumerate}
\item If for every $\hat{t}\in \lbrack -\hat{T},\hat{T}]$ we have $\psi
_{1}^{c}(\hat{\mathbf{x}}_{0},\hat{t})\in U_{1}\setminus \mathcal{C}_{1}$,
then $\gamma \left( s\right) ,$ for $s\in \lbrack -\hat{T},\hat{T}]$ lies on
a trajectory of the PCRTBP which avoids collisions. Moreover, the time $t$
in the original coordinates that corresponds to the time $\hat{t}\in \lbrack
-\hat{T},\hat{T}]$ in the regularised coordinates is recovered by the
integral 
\begin{equation}
t=4\int_{0}^{\hat{t}}\left( \hat{x}(s)^{2}+\hat{y}(s)^{2}\right) ds,
\label{eq:time-recovery}
\end{equation}%
i.e. 
\begin{equation*}
\phi \left( t,\mathbf{x}_{0}\right) =T_{1}\left( \psi _{1}^{c}(\hat{\mathbf{x%
}}_{0},\hat{t})\right) .
\end{equation*}

\item If for $\hat{T}>0$, for every $\hat{t}\in \lbrack 0,\hat{T})$ we have $%
\psi _{1}^{c}(\hat{\mathbf{x}}_{0},\hat{t})\in U_{1}\setminus \mathcal{C}%
_{1} $ and $\psi _{1}^{c}(\hat{\mathbf{x}}_{0},\hat{T})\in \mathcal{C}_{1}$,
then in the original coordinates the trajectory starting from $\mathbf{x}%
_{0} $ reaches the collision with $m_{1}$ at time $T>0$ given by%
\begin{equation}
T=4\int_{0}^{\hat{T}}\left( \hat{x}(s)^{2}+\hat{y}(s)^{2}\right) \,ds.
\label{eq:time-to-collision}
\end{equation}

\item If for $\hat{T}<0$, for every $\hat{t}\in (\hat{T},0]$ we have $\psi
_{1}^{c}(\hat{\mathbf{x}}_{0},\hat{t})\in U_{1}\setminus \mathcal{C}_{1}$
and $\psi _{1}^{c}(\hat{\mathbf{x}}_{0},\hat{T})\in \mathcal{C}_{1}$, then
in the original coordinates the backward trajectory starting from $\mathbf{x}%
_{0}$ reaches the collision with $m_{1}$ at time $T<0$ expressed in Equation \eqref{eq:time-to-collision}.
\end{enumerate}
\end{theorem}

Orbits satisfying condition 2 from Theorem \ref%
{thm:LeviCivitta} are collision orbits, while orbits satisfying condition
3 from Theorem \ref{thm:LeviCivitta} are called ejection orbits. From Theorem \ref{thm:LeviCivitta} we see that for regularized orbits $\psi
_{1}^{c}\left( \mathbf{\hat{x}}_{0},\hat{t}\right) $ to have a physical
meaning in the original coordinates we need to choose $c=E\left( T_{1}\left( 
\mathbf{\hat{x}}_{0}\right) \right)$ for the regularization energy. The following lemma, whose proof is a standard calculation (see \cite{theoryOfOrbits}), addresses this
choice.

\begin{lemma}
\label{lem:energies-cond}For every $\mathbf{\hat{x}}\in U_{1}$, we have
\begin{equation}
E\left( T_{1}\left( \mathbf{\hat{x}}\right) \right) =c\qquad \text{if and only if} \qquad
E_{1}^{c}\left( \mathbf{\hat{x}}\right) =0.  \label{eq:energies-cond-m1}
\end{equation}
\end{lemma}
The following corollary of Lemma \ref{lem:energies-cond} is a consequence of evaluating the expression 
for the energy at zero when the positions are zero.
\begin{corollary}
\label{cor:collisions-m1}If we consider $\mathbf{\hat{x}}=\left( \hat{x},%
\hat{p},\hat{y},\hat{q}\right) $ with $\hat{x}=\hat{y}=0$, which corresponds
to a collision with $m_{1}$, then from $E_{1}^{c}\left( \mathbf{\hat{x}}%
\right) =0$ we see that for a trajectory $\psi _{1}^{c}\left( \mathbf{\hat{x}%
},\hat{t}\right) $ starting from a collision point $\mathbf{\hat{x}}=\left(
0,\hat{p},0,\hat{q}\right) $ to have a physical meaning in the original
coordinates it is necessary and sufficient that 
\begin{equation}
\hat{q}^{2}+\hat{p}^{2}=8(1-\mu ).  \label{eq:collision-m1}
\end{equation}
\end{corollary}

\begin{definition}
\label{def:ejection-collision-manifolds}We refer to 
\begin{equation*}
\left\{ \psi _{1}^{c}\left( \mathbf{\hat{x}},\hat{t}\right) :\hat{q}^{2}+%
\hat{p}^{2}=8(1-\mu ),\,\hat{t}\geq 0\text{ and }\psi _{1}^{c}(\mathbf{\hat{x%
}},[0,\hat{t}])\cap \mathcal{C}_{1}=\emptyset \right\}
\end{equation*}%
as the ejection manifold from $m_{1},$ and%
\begin{equation*}
\left\{ \psi _{1}^{c}\left( \mathbf{\hat{x}},\hat{t}\right) :\hat{q}^{2}+%
\hat{p}^{2}=8(1-\mu ),\,\hat{t}\leq 0\text{ and }\psi _{1}^{c}(\mathbf{\hat{x%
}},[\hat{t},0])\cap \mathcal{C}_{1}=\emptyset \right\}
\end{equation*}%
 as the collision manifold to $m_{1}$.
\end{definition}

Note that both the collision and the ejection manifolds depend on the
choice of $c$. That is, we have a family of collision/ejection manifolds, 
parameterized by the Jacobi constant $c$. For a
fixed $c$ the collision manifold, when viewed in the original coordinates,
consists of points with energy $c$, whose forward trajectory reaches the
collision with $m_{1}$. Similarly, for fixed $c$, the ejection manifold, in
the original coordinates, consists of points with energy $c$ whose backward
trajectory collide with $m_{1}$. 
Thus, the circle defined in Corollary \ref{cor:collisions-m1} is a sort 
of ``fundamental domain'' for ejections/collisions to $m_1$ with energy $c$.

\subsection{Regularization of collisions with $m_{2}$} \label{sec:regSecondPrimary}

To regularize at the second primary, we define the coordinates $\tilde{z}=%
\tilde{x}+i\tilde{y}$ through $\tilde{z}^{2}=z+1-\mu $ and consider the time
rescaling $dt/d\tilde{t}=4|\tilde{z}|^{2}$.
As in the previous section, define 
\begin{eqnarray*}
U_{2}&:= &\left\{ \mathbf{\tilde{x}}=(\tilde{x},\tilde{p},\tilde{y},\tilde{q}%
)\in \mathbb{R}^{4}\,|\,\left( \tilde{x},\tilde{y}\right) \notin \left\{
\left( -1,0\right) ,\left( 1,0\right) \right\} \right\} , \\
\mathcal{C}_{2}&:= &\left\{ \mathbf{\tilde{x}}=(\tilde{x},\tilde{p},\tilde{y}%
,\tilde{q})\in \mathbb{R}^{4}\,|\,\tilde{x}=\tilde{y}=0\right\},
\end{eqnarray*}%
so that  $U_{2}$ consists of points in the regularized coordinates which do
not collide with $m_{1}$, and $\mathcal{C}_{2}$ consists of
points which collide with $m_{2}$.

The regularized Levi-Civita vector field $f_{2}^{c}:U_{2}\rightarrow \mathbb{%
R}^{4}$ with the ODE $\mathbf{\tilde{x}}^{\prime }=f_{2}^{c}\left( \mathbf{%
\tilde{x}}\right) $ is of the form (see \cite{theoryOfOrbits})%
\begin{eqnarray}
\tilde{x}^{\prime } &=&\tilde{p},  \label{eq:reg_S_field} \notag \\
\tilde{p}^{\prime } &=&8\left( \tilde{x}^{2}+\tilde{y}^{2}\right) \tilde{q}%
+12\tilde{x}(\tilde{x}^{2}+\tilde{y}^{2})^{2}-16(1-\mu )\tilde{x}%
^{3}+4\left( (1-\mu )-c\right) \tilde{x}  \notag \\
&&+\frac{8(1-\mu )\left( -\tilde{x}^{3}+3\tilde{x}\tilde{y}^{2}+\tilde{x}%
\right) }{((\tilde{x}^{2}+\tilde{y}^{2})^{2}+1+2(\tilde{y}^{2}-\tilde{x}%
^{2}))^{3/2}},  \notag \\
\tilde{y}^{\prime } &=&\tilde{q},   \label{eq:regularizedSystem_m2} \\
\tilde{q}^{\prime } &=&-8\left( \tilde{u}^{2}+\tilde{y}^{2}\right) \tilde{p}%
+12\tilde{y}(\tilde{x}^{2}+\tilde{y}^{2})^{2}+16(1-\mu )\tilde{y}%
^{3}+4\left( (1-\mu )-c\right) \tilde{y}  \notag \\
&&+\frac{8(1-\mu )\left( \tilde{y}^{3}-3\tilde{x}^{2}\tilde{y}+\tilde{y}%
\right) }{((\tilde{x}^{2}+\tilde{y}^{2})^{2}+1+2(\tilde{y}^{2}-\tilde{x}%
^{2}))^{3/2}}, \notag 
\end{eqnarray}%
with the integral of motion 
\begin{align}
E_{2}^{c}\left( \mathbf{\tilde{x}}\right) & =-\tilde{p}^{2}-\tilde{q}^{2}+4(%
\tilde{x}^{2}+\tilde{y}^{2})^{3}+8(1-\mu )(\tilde{y}^{4}-\tilde{x}%
^{4})+4\left( (1-\mu )-c\right) (\tilde{x}^{2}+\tilde{y}^{2})  \notag \\
& \quad +8(1-\mu )\frac{\tilde{x}^{2}+\tilde{y}^{2}}{\sqrt{(\tilde{x}^{2}+%
\tilde{y}^{2})^{2}+1+2(\tilde{y}^{2}-\tilde{x}^{2})}}+8\mu .  \label{eq:E2}
\end{align}

We write $\psi_2^c( \mathbf{\tilde x},\tilde t)$ for the flow induced
by (\ref{eq:reg_S_field}).

The change of coordinates from the regularized coordinates $\mathbf{\tilde{x}%
}$ to the original coordinates $\mathbf{x}$ is given by $T_{2}:U_{2}%
\setminus \mathcal{C}_{2}\rightarrow \mathbb{R}^{4}$\textbf{\ }of the form%
\begin{equation}
\mathbf{x}=T_{2}\left( \mathbf{\tilde{x}}\right) =\left( 
\begin{array}{c}
\tilde{x}^{2}-\tilde{y}^{2}+\mu-1 \\ 
\frac{\tilde{x}\tilde{p}-\tilde{y}\tilde{q}}{2(\tilde{x}^{2}+\tilde{y}^{2})}
\\ 
2\tilde{x}\tilde{y} \\ 
\frac{\tilde{y}\tilde{p}+\tilde{x}\tilde{q}}{2(\tilde{x}^{2}+\tilde{y}^{2})}%
\end{array}%
\right) .  \label{eq:T2-def}
\end{equation}

A theorem analogous to Theorem \ref{thm:LeviCivitta} characterizes solution
curves in the two coordinate systems and the collisions with the second
primary $m_{2}$. Also, analogously to Lemma \ref{lem:energies-cond} and
Corollary \ref{cor:collisions-m1} for every $\mathbf{\ \tilde{x}}\in U_{2}$
we have%
\begin{equation}
E\left( T_{2}\left( \mathbf{\tilde{x}}\right) \right) =c\qquad \text{if and only if} \qquad
E_{2}^{c}\left( \mathbf{\tilde{x}}\right) =0,  \label{eq:energies-cond-m2}
\end{equation}%
and a trajectory $\psi _{2}^{c}\left( \mathbf{\tilde{x}},\tilde{t}\right) $
starting from a collision point $\mathbf{\tilde{x}}=\left( 0,\tilde{p},0,%
\tilde{q}\right) $ with $m_{2}$ has physical meaning in the original
coordinates if and only if 
\begin{equation}
\tilde{q}^{2}+\tilde{p}^{2}=8\mu .  \label{eq:collision-m2}
\end{equation}

We introduce the notions of the ejection and collision manifolds for $m_{2}$
analogously to Definition \ref{def:ejection-collision-manifolds}.


\begin{figure}[t!]
\begin{center}
\includegraphics[height=6.0cm]{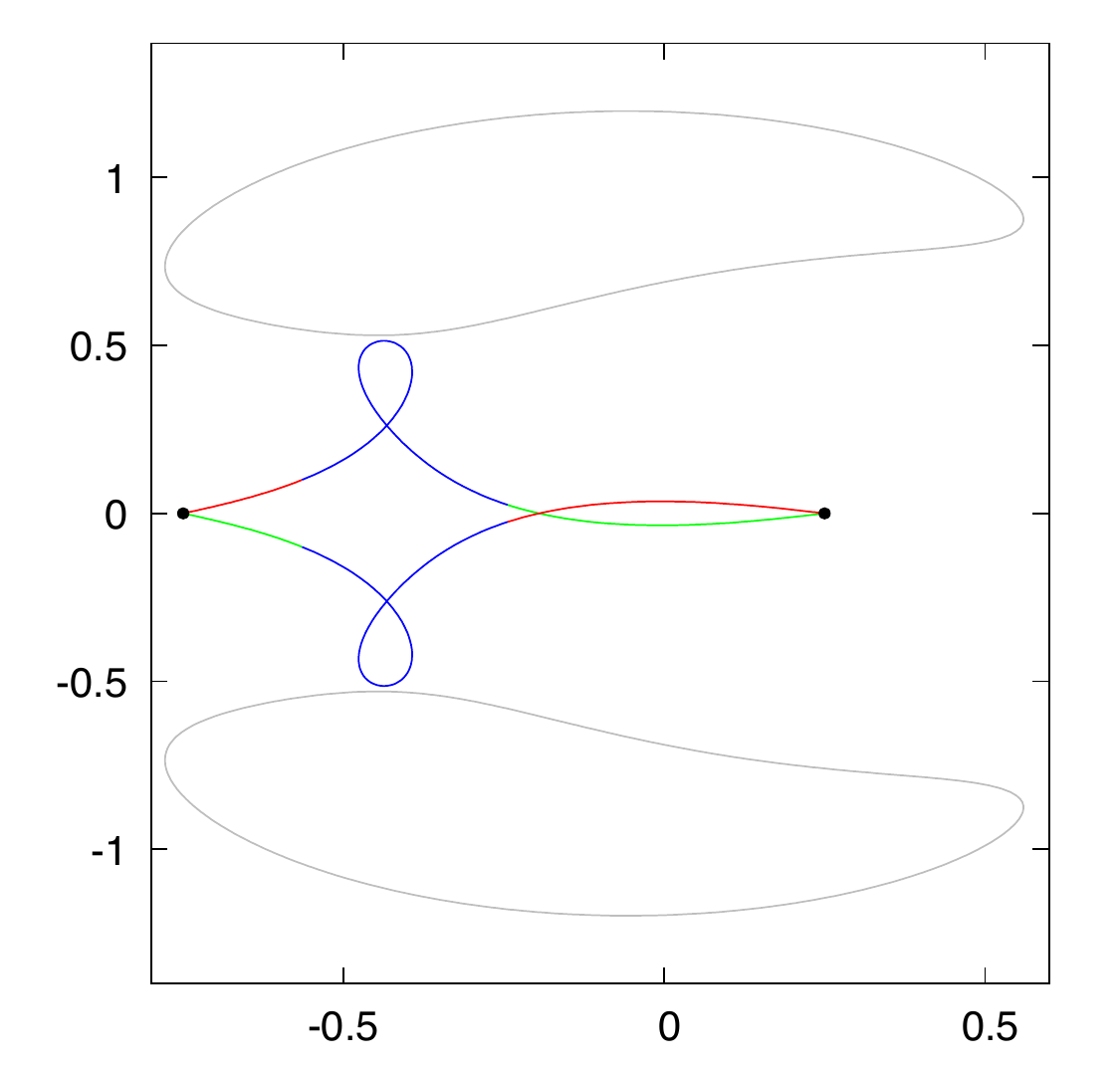} %
\end{center}
\caption{Ejection collision orbits in the PCRTBP when $\mu = 1/4$ 
and $C = 3.2$.  The grey curves at the top and bottom of the figure illustrate 
the zero velocity curves, i.e. the boundaries of the prohibited Hill's regions, for 
this value of $C$.    The black dots at $x = \mu$ and $x = -1+\mu$ depict
the locations of the primary bodies. The curves in the middle of the figure 
represent two ejection-collision orbits: $m_2$ to $m_1$ (bottom) and 
$m_1$ to $m_2$ (top). (Recall that $m_2$ is on the left and $m_1$ on the right; 
compare with Figure \ref{fig:PCRTBP_coordinates}.)  These orbits are computed by numerically
locating an approximate zero of the function defined in Equation \eqref{eq:collisionOperator}.
In setting up the BVP we choose to spend $s = 0.35$ time units in each of the 
regularized coordinate systems (red and green orbit segments)
but this transforms to unequal amounts of time in the original/synodic 
coordinates.  
The blue portion of the orbit is in the 
original coordinates.  The curves
are plotted by changing all points back to the original coordinates.
The entire ejection-collision takes about $2.427$ time units in the original/synodic
coordinates.}
\label{fig:ejectionCollisions}
\end{figure}

\section{Ejection-collision orbits} \label{sec:ejectionToCollision}

We now define a level set multiple shooting operator whose zeros
correspond to transverse ejection-collision orbits from the body $m_{k}$ 
to the body $m_{l}$ for $k,l\in\left\{ 1,2\right\}$ in the PCRTBP.
\correction{comment 12}{
Two such orbits in the PCRTBP are illustrated in Figure \ref{fig:ejectionCollisions}.
}

Note that the PCRTBP has the form discussed in Example 
\ref{ex:dissipative-unfolding}, so that a dissipative unfolding
is given by the one parameter family of ODEs 
\begin{equation}
f_{\alpha}(x,p,y,q)=f(x,p,y,q)+\alpha\left( 0,p,0,q\right),
\label{eq:unfoldedPCRTBP}
\end{equation}
where $f$ is as defined in Equation \eqref{eq:PCRTBP}.
Let $\phi_{\alpha}(\mathbf{x},t)$ denote the flow generated by the
the vector field of Equation \eqref{eq:unfoldedPCRTBP}. 
 For $c\in\mathbb{R}$ consider the fixed energy level
set $M$. Then $\alpha$ is an unfolding parameter for 
the mapping 
\begin{equation*}
R_{\tau,\alpha}\left( \mathbf{x}\right) =\phi_{\alpha}(\mathbf{x},\tau)
\end{equation*}
from $M$ to $M$. (Here 
$R_{\tau,\alpha}:\mathbb{R}^{4}\rightarrow \mathbb{R}^{4}$ for fixed $\alpha,\tau\in\mathbb{R}$.)

Define the functions $P_{i} \colon \mathbb{R} \to \mathbb{R}^4$ for $i = 1,2$ by 
\begin{equation}
P_{i}\left( \theta \right) :=\left\{ 
\begin{array}{lll}
(0,\sqrt{8\left( 1-\mu \right) }\cos \left( \theta \right) ,0,\sqrt{8\left(
1-\mu \right) }\sin \theta ) &  & \text{for }i=1,\medskip \\ 
(0,\sqrt{8\mu }\cos \left( \theta \right) ,0,\sqrt{8\mu }\sin \theta ) &  & 
\text{for }i=2.%
\end{array}%
\right.  \label{eq:collisions-par-Pi}
\end{equation}%
By Equations \eqref{eq:collision-m1} and \eqref{eq:collision-m2} the function $%
P_{i}\left( \theta \right) $ parameterizes the collision set for the primary $%
m_{i}$, with $i=1,2$.
Fix $k,l\in \left\{ 1,2\right\} $ and consider level sets $%
M_{1},\ldots ,M_{6}\subset \mathbb{R}^{4}$ defined by %
\begin{align*}
M_{1}& =M_{2}=\left\{ E_{k}^{c}=0\right\} , \\
M_{3}& =M_{4}=\left\{ E=c\right\} , \\
M_{5}& =M_{6}=\left\{ E_{l}^{c}=0\right\}.
\end{align*}%
Choose $s>0$, 
and for $i = 1,2$ recall the definition of the coordinate transformations 
$T_{i} \colon U_i \backslash \mathcal{C}_i \to \mathbb{R}^4$ defined 
in Equations \eqref{eq:T1-def} and \eqref{eq:T2-def}.
Taking the maps $R_{\tau ,\alpha }^{1},\ldots ,R_{\tau ,\alpha
}^{5}:\mathbb{R}^{4}\rightarrow \mathbb{R}^{4}$ as%
\begin{align*}
R_{\tau ,\alpha }^{1}\left( x_{1}\right) & =\psi _{k}^{c}\left(
x_{1},s\right) , \\
R_{\tau ,\alpha }^{2}\left( x_{2}\right) & =T_{k}\left( x_{2}\right) , \\
R_{\tau ,\alpha }^{3}\left( x_{3}\right) & =\phi _{\alpha }\left( x_{3},\tau
\right) , \\
R_{\tau ,\alpha }^{4}\left( x_{4}\right) & =T_{l}^{-1}\left( x_{4}\right) ,
\\
R_{\tau ,\alpha }^{5}\left( x_{5}\right) & =\psi _{l}^{c}\left(
x_{5},s\right),
\end{align*}
we let 
\begin{equation*}
F:\mathbb{R\times }\underset{5 \ \text{copies}}{\underbrace{\mathbb{R}^{4}\mathbb{\times }%
\ldots \mathbb{\times R}^{4}}}\mathbb{\times R\times R\times R\rightarrow }%
\underset{6 \ \text{copies}}{\underbrace{\mathbb{R}^{4}\mathbb{\times }\ldots \mathbb{\times
R}^{4}}}
\end{equation*}%
be defined as %
\begin{equation}\label{eq:collisionOperator}
F\left( x_{0},x_{1},\ldots x_{5},x_{6},\tau ,\alpha \right):=
\left( 
\begin{array}{r@{\,\,\,}l}
P_{k}\left( x_{0}\right) & -\,\,\,x_{1}  \\ 
R_{\alpha ,\tau }^{1}\left(x_{1}\right) &- \,\,\, x_{2} \\ 
R_{\alpha ,\tau }^{2}\left(x_{2}\right) &- \,\,\, x_{3} \\ 
R_{\alpha ,\tau }^{3}\left(x_{3}\right) &- \,\,\, x_{4} \\ 
 R_{\alpha ,\tau }^{4}\left( x_{4}\right) &- \,\,\, x_{5} \\
 R_{\alpha ,\tau }^{5}\left( x_{5}\right) &- \,\,\, P_{l}\left( x_{6}\right)
\end{array}
\right),
\end{equation}
where $x_{0},x_{6},\tau ,\alpha \in \mathbb{R}$ and $%
x_{1},\ldots ,x_{5}\in \mathbb{R}^{4}$.
We also write $\left( x_{k},p_{k},y_{k},q_{k}\right) $ and $\left(
x_{l},p_{l},y_{l},q_{l}\right) $ to denote the regularized coordinates given by the
coordinate transformations $T_{k}$ and $T_{l}$, respectively.

\begin{lemma}\label{lem:collision-connections}
Let $\mathbf{x}^{\ast }=\left( x_{0}^{\ast },\ldots ,x_{6}^{\ast }\right) 
$ and $\tau ^{\ast }>0$. If   
\begin{equation*}
DF\left( \mathbf{x}^{\ast },\tau ^{\ast },0\right)
\end{equation*}%
is an isomorphism and 
\begin{equation*}
F\left( \mathbf{x}^{\ast },\tau ^{\ast },0\right) =0,
\end{equation*}%
then the orbit of the point $x_3^{\ast}$
is ejected from the primary body $m_{k}$ and collides with 
the primary body $m_{l}.$  (The same is true of the orbit of 
the point $x_4^{\ast}$.)
Moreover, intersection of the collision and ejection manifolds is
transversal on the energy level $\left\{ E=c\right\} $ and the time 
from the ejection to the collision is 
\begin{equation} 
\tau ^{\ast }+4\int_{0}^{s}\left\Vert \pi _{x_{k},y_{k}}\psi _{k}^{c}\left(
x_{1}^{\ast },u\right) \right\Vert ^{2}du+4\int_{0}^{s}\left\Vert \pi
_{x_{l},y_{l}}\psi _{l}^{c}\left( x_{5}^{\ast },u\right) \right\Vert ^{2}du.
\label{eq:time-between-collisions}
\end{equation}%
(Above we use the Euclidean norm.)
\end{lemma}

\begin{proof}
We have $d_{0}=d_{6}=k=1$ and $d=4$, so the condition in Equation \eqref
{eq:dimensions-multiple-shooting} is satisfied.
We now show that $\alpha $ is an unfolding parameter for $R_{\tau
,\alpha }=R_{\tau ,\alpha }^{5}\circ \ldots \circ R_{\tau ,\alpha }^{1}$.
Since $E_{i}^{c}$ is an integral of motion for the flow $\psi _{i}^{c}$, for 
$i=1,2$, we see that%
\begin{equation*}
\begin{array}{rcl}
x_{1}\in M_{1}=\left\{ E_{k}^{c}=0\right\}  & \qquad \iff \qquad  & R_{\tau
,\alpha }^{1}\left( x_{1}\right) =\psi _{k}^{c}\left( x_{1},s\right) \in
M_{2}=\left\{ E_{k}^{c}=0\right\} ,\medskip  \\ 
x_{5}\in M_{5}=\left\{ E_{l}^{c}=0\right\}  & \qquad \iff \qquad  & R_{\tau
,\alpha }^{5}\left( x_{5}\right) =\psi _{l}^{c}\left( x_{5},s\right) \in
M_{6}=\left\{ E_{l}^{c}=0\right\} .%
\end{array}%
\end{equation*}%
Also, by Equations \eqref{eq:energies-cond-m1} and \eqref{eq:energies-cond-m2} we see
that%
\begin{equation*}
\begin{array}{rcl}
x_{2}\in M_{2}=\left\{ E_{k}^{c}=0\right\}  & \qquad \iff \qquad  & R_{\tau
,\alpha }^{2}\left( x_{2}\right) =T_{k}\left( x_{2}\right) \in M_{3}=\left\{
E=c\right\} ,\medskip  \\ 
x_{4}\in M_{4}=\left\{ E=c\right\}  & \qquad \iff \qquad  & R_{\tau ,\alpha
}^{4}\left( x_{2}\right) =T_{l}^{-1}\left( x_{4}\right) \in M_{5}=\left\{
E_{l}^{c}=0\right\} .%
\end{array}%
\end{equation*}%
Moreover $\alpha $ is an unfolding parameter for the PCRTBP, and hence
for 
\begin{equation*}
R_{\tau ,\alpha }^{3}\left( x_{3}\right) =\phi _{\alpha }\left( x_{3},\tau
\right).
\end{equation*}%
Note that for $i=1,2,4,5$, the maps$R_{\tau ,\alpha }^{i}$ takes the level sets $M_{i}$ into the level set $M_{i+1}$ and this does not depend on the choice of $\alpha$.
Then, since $\alpha $ is an unfolding parameter for $R_{\tau ,\alpha }^{3}$,
it follows directly from Definition \ref{def:unfolding} that $%
\alpha $ is an unfolding parameter for $R_{\tau ,\alpha }=R_{\tau ,\alpha
}^{5}\circ \ldots \circ R_{\tau ,\alpha }^{1}.$

By applying Lemma \ref{lem:multiple-shooting-2} to
\begin{equation*}
\tilde{F}\left( x_{0},x_{6},\tau ,\alpha \right) :=R_{\tau ,\alpha }\left(
P_{k}\left( x_{0}\right) \right) -P_{l}\left( x_{6}\right)
\end{equation*}%
we obtain that $D\tilde{F}\left( x_{0}^{\ast },x_{6}^{\ast },\tau ^{\ast
},0\right) $ is an isomorphism and that $\tilde{F}\left( x_{0}^{\ast
},x_{6}^{\ast },\tau ^{\ast },0\right) =0$. Since 
\begin{equation*}
\tilde{F}\left( x_{0}^{\ast },x_{6}^{\ast },\tau ^{\ast },0\right) =\psi
_{l}^{c}\left( T_{l}^{-1}\left( \phi \left( T_{k}\left( \psi _{k}^{c}\left(
P_{k}(x_{0}^{\ast }),s\right) \right) ,\tau ^{\ast }\right) \right)
,s\right) -P_{l}\left( x_{6}^{\ast }\right) ,
\end{equation*}%
we see that, by Theorem \ref{thm:LeviCivitta} (and its mirror counterpart for
the collision with $m_{2}$) we have an orbit originating 
at the point $P_{k}(x_{0}^{\ast })$ on the collision set for $m_k$, and 
terminating at the point $P_{l}\left( x_{6}^{\ast }\right) $
on the collision set for $m_l$. The
transversality of the intersection between the 
ejection manifold of $m_{k}$ and the collision manifold
of $m_{l}$ follows from Theorem \ref{th:single-shooting}.
The time between collisions in Equation \eqref{eq:time-between-collisions} follows
from Equation \eqref{eq:time-to-collision}.
\end{proof}

\begin{remark}[Additional shooting steps] \label{rem:additionalShooting}
{\em
We remark that in practice, computing accurate enclosures
of flow maps requires shortening the time step.  Consider for example 
the third and fourth component of $F$ as defined in 
Equation \eqref{eq:collisionOperator}, and  
suppose that time step of length $\nicefrac{\tau}{N}$ is desired.  By the properties
of the flow map, solving the sub-system of equations 
\begin{equation}\label{eq:colOp_comp3}
\begin{aligned}
R_{\alpha, \tau}^3(x_3) - x_4  = \phi_\alpha(x_3, \tau) - x_4 &= 0 \\
R_{\alpha, \tau}^4(x_4) - x_5  = T^{-1}_l(x_4) - x_5 &= 0
\end{aligned}
\end{equation}
is equivalent to solving 
\begin{align*}
\phi_\alpha(x_3, \nicefrac{\tau}{N})  - y_1 &= 0\\
\phi_\alpha(y_1, \nicefrac{\tau}{N})  - y_2 &= 0\\
&\vdots \\
\phi_\alpha(y_{N-2}, \nicefrac{\tau}{N})  - y_{N-1} &= 0 \\
\phi_{\alpha}(y_{N-1}, \nicefrac{\tau}{N})  - x_4 &= 0 \\
T_l^{-1}(x_4) - x_5 &= 0,
\end{align*}
and we can append these new variables and 
components to the map $F$ defined in Equation \eqref{eq:collisionOperator}
without changing the zeros of the operator.  
Moreover, by Lemma \ref{lem:multiple-shooting-2} the 
transversality result for the operator is not changed by the addition of  
additional steps.  Indeed, by the same reasoning we can (and do) add intermediate 
shooting steps in the regularized coordinates
to reduce the time steps to any desired tolerance.
}
\end{remark}


\section{Connections between collisions and libration points $L_{4}$, $L_{5}$%
}\label{sec:L4_to_collision}

For each value of $\mu \in (0, 1/2]$, 
the PCRTBP has exactly five equilibrium solutions.
For traditional reasons, these are referred to as libration points
of the PCRTBP. Three of these are collinear with the primary bodies,
and lie on the $x$-axis. These are referred to as $L_1, L_2$ and 
$L_3$, and they correspond to the co-linear relative equilibrium
solutions discovered by Euler. The remaining two libration points
are located at the third vertex of the equilateral triangles whose other two
vertices are the primary and secondary bodies. These are referred to as $L_4$
and $L_5$, and correspond to the equilateral triangle solutions of Lagrange. 
Figure \ref{fig:PCRTBP_librations} illustrates the locations of the 
libration points in the phase space.

\begin{figure}[!t]
\centering
\includegraphics[height=5cm]{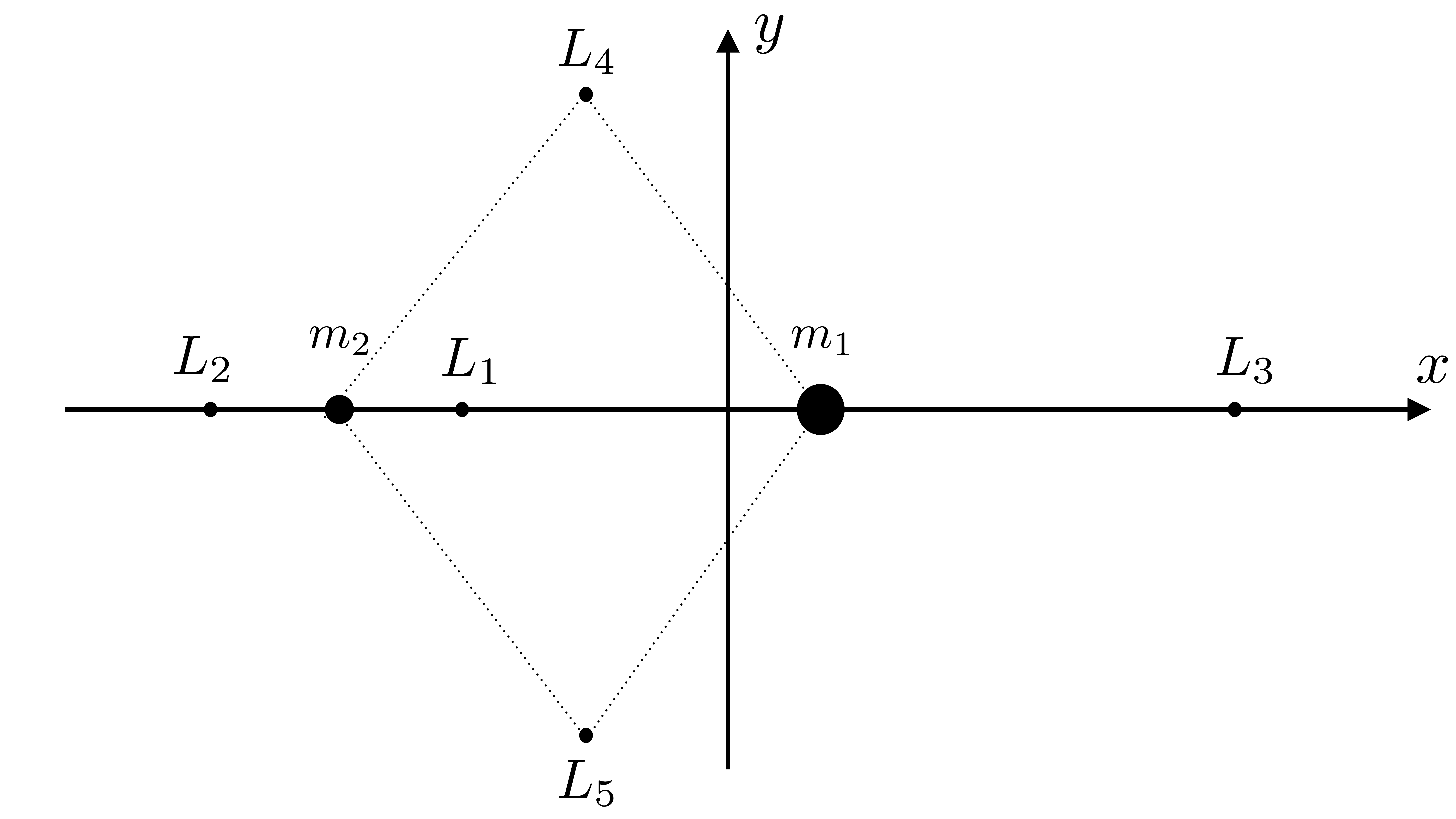}
\caption{The three collinear libration points $L_{1,2,3}$ and the
equilateral triangle libration points $L_{4, 5}$, relative to the positions 
of the primary masses $m_1$ and $m_2$.}
\label{fig:PCRTBP_librations}
\end{figure}

For all values of the mass ratio, the collinear libration points have saddle 
$\times$ center stability. The center manifolds give rise to important
families of periodic orbits known as Lyapunov families. The stability of $%
L_4 $ and $L_5$ depend on the mass ratio $\mu$. For 
\begin{equation*}
0 < \mu < \mu_* \approx 0.04,
\end{equation*}
where the exact value is $\mu_* = 2/(25 + \sqrt{621})$,  the triangular
libration points have center $\times$ center stability. That is, they are
stable in the the sense of Hamiltonian systems and exhibit the full ``zoo''
of nearby KAM objects.

When $\mu > \mu_*$, the triangular libration points $L_4$ and $L_5$ have
saddle-focus stability. That is, they have a complex conjugate pair of
stable and a complex conjugate pair of 
unstable eigenvalues. The four
eigenvalues then have the form 
\begin{equation*}
\lambda = \pm \alpha \pm i \beta,
\end{equation*}
for some $\alpha, \beta > 0$. In this case, each libration point has an
attached two dimensional stable and two dimensional unstable manifold. Since
these two dimensional manifolds live in the three dimensional energy level
set of $L_{4,5}$, 
there exists the possibility
that they intersect the two dimensional collision or ejection
manifolds of the primaries transversely. It is also possible that the
 stable/unstable manifolds of $L_{4,5}$  intersect one other transversely giving rise to
homoclinic or heteroclinic connecting orbits. 

In fact, in this paper we prove that both of these phenomena occur and in this section we discuss our method for proving the existence of intersections between a stable/unstable manifold of $L_{4,5}$,  and
an ejection/collision manifold of a primary body.  Any point of intersection between these 
manifolds gives rise to an orbit which is asymptotic  
to $L_4$, but which collides or is ejected from one of the 
massive bodies.  
Two such orbits are illustrated in Figure \ref{fig:EC_to_collision}.

\begin{figure}[!t]
\centering
\includegraphics[height=4.75cm]{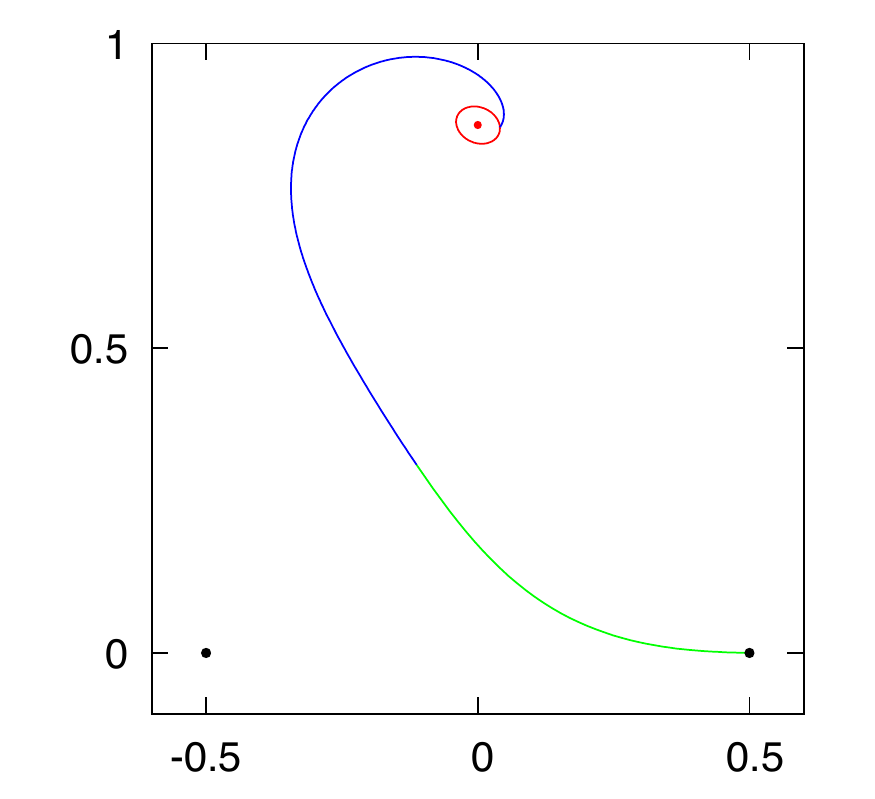}\includegraphics[height=4.75cm]{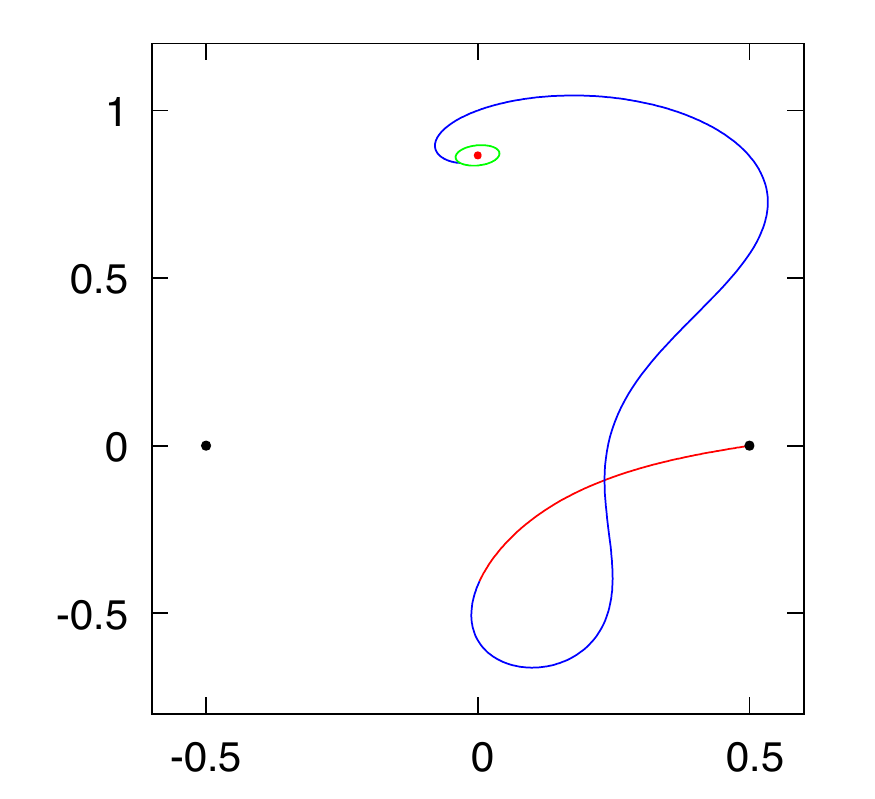}
\caption{Libration-to-collision and ejection-to-libration orbits for 
$\mu = 1/2$ and $c = 3$ (which is the $L_4$ value of the Jacobi constant
in the equal mass problem).  The left
frame illustrates an ejection to $L_4$ orbit, and the right frame an $L_4$ 
to collision.  In each frame $m_1$ is depicted as a black dot and $L_4$ 
as a red dot.  The boundary of a parameterized local unstable manifold for 
$L_4$ is depicted as the red circle; stable boundary the green circle.  The 
orbits are found by computing an approximate zero of the map
defined in Equation \eqref{eq:EC_to_L4_operator}. 
The green portion of the left, and red portion of the right 
curves are computed in regularized coordinates
for the body $m_1$, where we have fixed $s = 0.5$ regularized
time units before the change back to original/synodic coordinates.  
These orbit segments are transformed back to the original 
coordinates for the plot.}
\label{fig:EC_to_collision}
\end{figure}

Let $\overline{B}\subset \mathbb{R}^{2}$ denote a closed ball with radius $1$.
 %
 Assume that%
\begin{equation*}
w_{j}^{\kappa }:\overline{B}\rightarrow \mathbb{R}^{4}\qquad \text{%
for }j\in \{4,5\}\text{ and }\kappa \in \left\{ u,s\right\},
\end{equation*}%
parameterize the two dimensional local stable/unstable manifolds of $L_{j}$. 
We assume that the charts  are normalized so
that $w_{j}^{\kappa }\left( 0\right) =L_{j}$.  Then 
\begin{equation*}
w_{j}^{\kappa }\left( \overline{B}\right) =W_{\text{loc}}^{\kappa
}\left( L_{j}\right) \qquad \text{for }j\in \{4,5\},\text{ }\kappa \in
\left\{ u,s\right\} .
\end{equation*}%
Define the functions%
\begin{equation*}
P_{j}^{\kappa }:\mathbb{R}\rightarrow \mathbb{R}^{4}\qquad \text{for }j\in
\{4,5\}\text{ and }\kappa \in \left\{ u,s\right\},
\end{equation*}%
by%
\begin{equation}
P_{j}^{\kappa }\left( \theta \right) :=w_{j}^{\kappa }\left( \cos
\theta ,\sin \theta \right) .  \label{eq:Pj-lib}
\end{equation}
For $i\in \{1,2\}$ consider $P_{i}$ as defined in Equation \eqref{eq:collisions-par-Pi}. 

For 
\begin{equation*}
\mathbf{x}=\left( x_{0},x_{1},x_{2},x_{3},x_{4}\right) \in \mathbb{R}^{14},
\end{equation*}%
where $x_{0},x_{4}\in \mathbb{R}, x_{1},x_{2},x_{3}\in \mathbb{R}^{4}$, and $j \in \left\{ 4,5\right\} $ we
define 
\begin{equation*}
F_{i,j}^{u},F_{i,j}^{s}:\mathbb{R}^{16}\rightarrow \mathbb{R}^{16},
\end{equation*}%
by the formulas
\begin{equation} \label{eq:EC_to_L4_operator}
F_{i,j}^{u}\left( \mathbf{x},\tau ,\alpha \right) =\left( 
\begin{array}{r@{\,\,-\,\,}l}
P_{j}^{u}\left( x_{0}\right) & x_{1} \\ 
\phi _{\alpha }\left( x_{1},\tau \right) & x_{2} \\ 
T_{i}^{-1}(x_{2}) & x_{3} \\ 
\psi _{i}^{c_{j}}\left( x_{3},s\right) & P_{i}(x_{4})%
\end{array}%
\right) ,\quad F_{i,j}^{s}\left( \mathbf{x},\tau ,\alpha \right) =\left( 
\begin{array}{r@{\,\,-\,\,}l}
P_{i}(x_{0}) & x_{1} \\ 
\psi _{i}^{c_{j}}\left( x_{1},s\right) & x_{2} \\ 
T_{i}(x_{2}) & x_{3} \\ 
\phi _{\alpha }\left( x_{3},\tau \right) & P_{j}^{s}\left( x_{4}\right)%
\end{array}%
\right).
\end{equation}%
Here $\tau ,\alpha \in \mathbb{R}$ and the constant $c_{j}$ in $\psi
_{i}^{c_{j}}$ is chosen as $c_{j}=E\left( L_{j}\right) $.

Zeros of the operator $F_{i,j}^{u}$ correspond to intersections of the
unstable manifold of $L_{j}$ with the collision manifold of mass $m_{i}.$ 
We also refer to this as a heteroclinic connection from $L_{j}$ to $m_{i}$.
Similarly, zeros of the operator $F_{i,j}^{s}$ correspond to intersections between
the stable manifold of $L_{j}$ with the ejection manifold of mass $m_{i}.$ In other
words, they lead to heteroclinic connections ejected from $%
m_{i}$ and limiting to the libration point $L_{j}$ in forward time. This is expressed formally in the following lemma.

\begin{lemma}\label{lem:Li-collisions}
Fix $i\in \left\{ 1,2\right\} ,$ $j\in \{4,5\}$, and $\kappa \in \left\{
u,s\right\} $.  Suppose there exists
$\mathbf{x}^{\ast} = (x_0^{\ast}, x_1^{\ast}, x_2^{\ast}, x_3^{\ast}, x_4^{\ast}) \in \mathbb{R}^{14}$ 
and $\tau ^{\ast } > 0 $ satisfying
\begin{equation*}
	F_{i,j}^{\kappa }\left( \mathbf{x}^{\ast },\tau ^{\ast },0\right) =0,
\end{equation*}
and such that 
\begin{equation*}
DF_{i,j}^{\kappa }\left( \mathbf{x}^{\ast },\tau ^{\ast },0\right)
\end{equation*}%
is an isomorphism. Then we have the following two cases. 
\begin{enumerate}
\item If $\kappa = u$, then the orbit of $x_1^{\ast}$ is heteroclinic from the libration point $L_{j}$
to collision  with $m_{i}$ and the intersection of $W^{u}\left( L_{j}\right)$ with the collision manifold of $m_{i}$ is transverse with respect to the energy level $\left\{ E=c_j\right\} $.
\item If $\kappa = s$, then the orbit of $x_3^{\ast}$ is heteroclinic from the libration point $L_{j}$
to ejection with $m_{i}$ and the intersection of $W^{s}\left( L_{j}\right)$ with the ejection manifold of $m_i$ is transverse with respect to the energy level $\left\{ E=c_j\right\} $.
\end{enumerate}
%
\end{lemma}

\begin{proof}
The proof follows from an argument similar to the proof of Lemma \ref{lem:collision-connections}.
\end{proof}


\bigskip

By a small modification of the operator just defined, we can 
study orbits homoclinic or heteroclinic to the libration points as well.  Such orbits 
arise as intersections of the stable/unstable manifolds of 
the libration points, and lead naturally to two point BVPs.
Three such orbits, homoclinic to $L_4$ in the 
PCRTBP, are illustrated in Figure \ref{fig:PCRTBP_L4_homoclinics}.

Note that homoclinic/heteroclinic connections between equilibrium solutions 
do not require changing to regularized coordinates as
such orbits exists for all forward and backward time and cannot 
have any collisions.  While this claim is mathematically correct, 
any homoclinic/heteroclinic orbit which passes sufficiently close to a
collision with $m_{i}$ for $i\in \left\{ 1,2\right\} $ becomes
difficult to continue numerically. Consequently, these orbits may still be difficult or impossible to validate via computer assisted proof. 
In this case regularization techniques are an asset even when studying orbits 
which 
pass near a collision.   
The left and center homoclinic orbits in
 Figure \ref{fig:PCRTBP_L4_homoclinics} for example are computed 
 entirely in the usual PCRTBP coordinates, while the right orbit was 
 computed using both coordinate systems.  With this in mind we express the homoclinic/heteroclinic problem in the framework set up in the previous sections.  
\begin{figure}[!t]
\centering
\includegraphics[height=3.8cm]{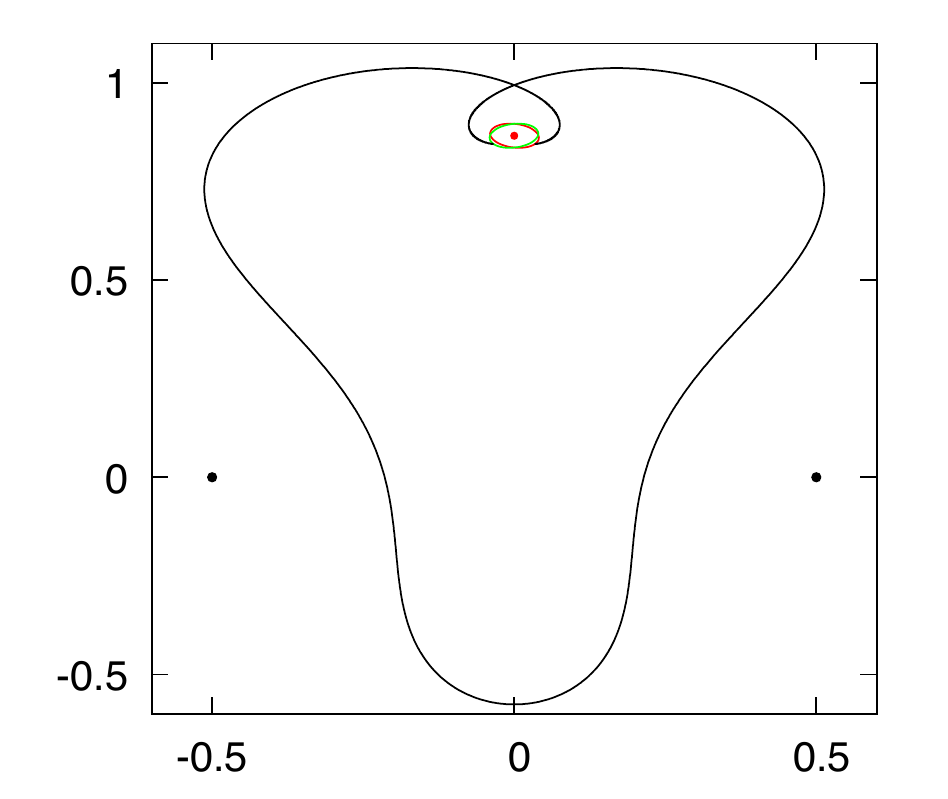}\includegraphics[height=3.8cm]{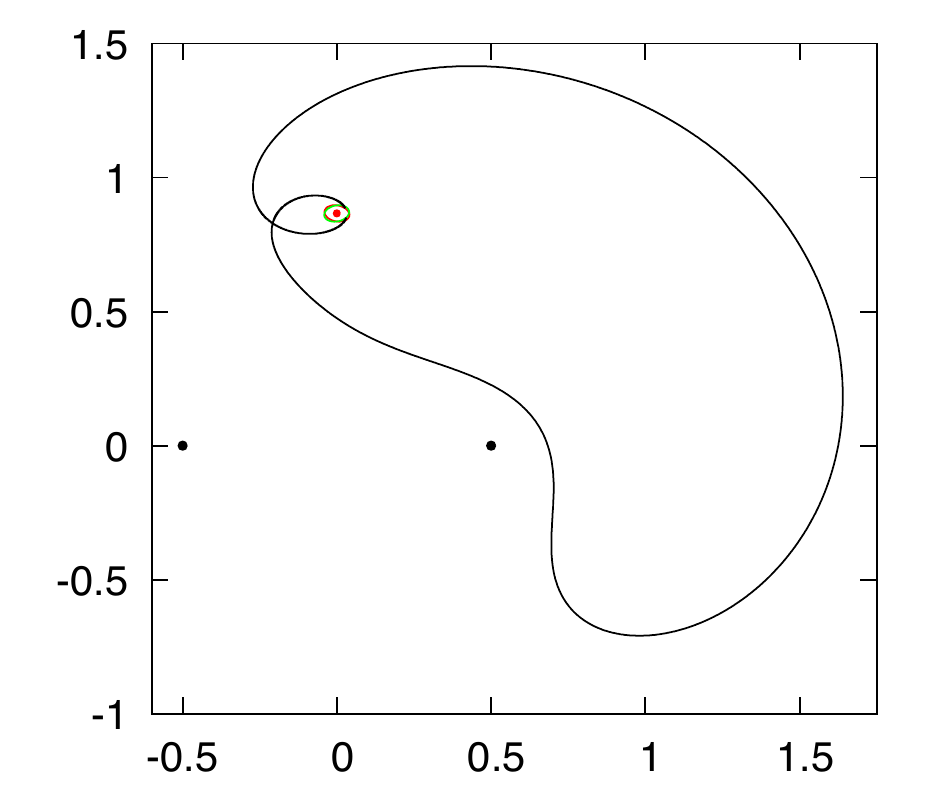}\includegraphics[height=3.8cm]{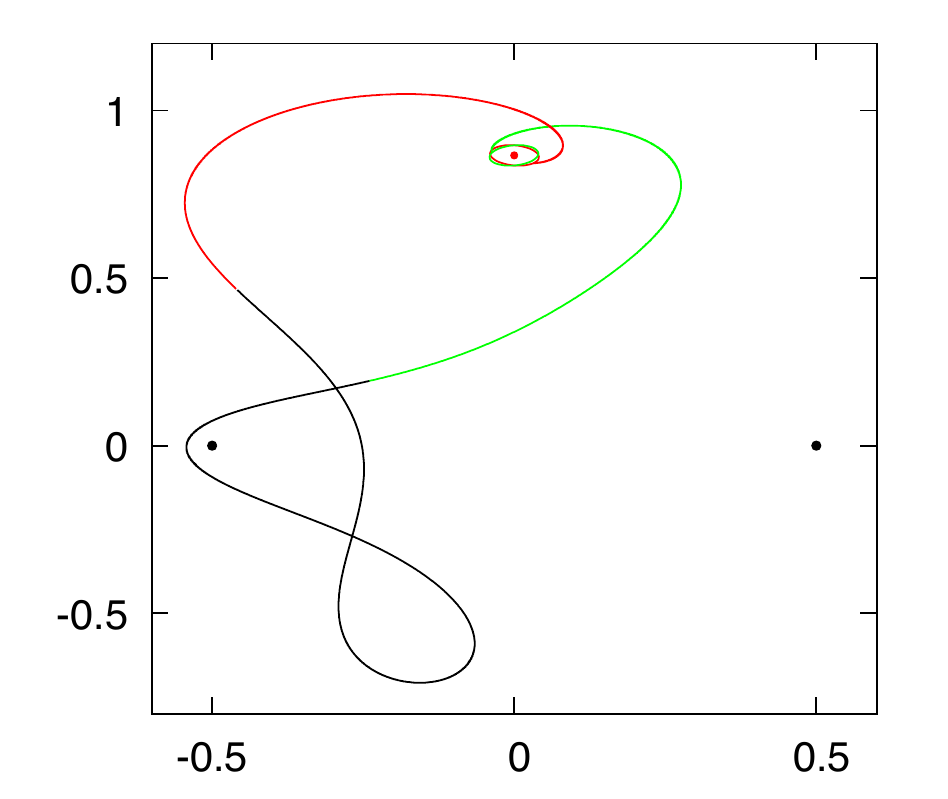}
\caption{
Transverse homoclinic orbits at $L_4$
for $\mu = 1/2$ in the $C = 3$ energy level. 
Each orbit traverses the illustrated curves in a clockwise fashion.  
The left and center orbits were known to 
Stromgren and Szebeheley.  The center and right 
orbits possess no symmetry, and the orbit on the right
passes close to collision with $m_2$.
Each orbit is found by approximately computing a zero of the 
map defined in Equation \eqref{eq:homoclinicOperator}.  
The left and center orbits are computed in only the standard 
coordinate system.  In the definition of the shooting template, we allow
the orbit to spend $s_1 = 1.8635$ regularized time units in Levi-Civita
coordinates and to flow for 
$s_2 =5$ time units in the original/synodic coordinates
before reaching the stable manifold.
}
\label{fig:PCRTBP_L4_homoclinics}
\end{figure}

Let $P^{\kappa}_{j}:\mathbb{R}\rightarrow \mathbb{R}^{4}$, for $j\in \left\{
4,5\right\} $ be the functions defined in Equation \eqref{eq:Pj-lib} and consider 
\begin{equation*}
\mathbf{x}=\left( x_{0},\ldots ,x_{6}\right) \in \mathbb{R}^{22},
\end{equation*}%
where $x_{0},x_{6}\in \mathbb{R}$ and $x_{1},\ldots ,x_{5}\in \mathbb{R}^{4}$%
, and fix $s_{1},s_{2}>0$. Let 
\begin{equation*}
F_{i,j,k}:\mathbb{R}^{24}\rightarrow \mathbb{R}^{24},\qquad \text{for }%
j,k\in \{4,5\},i\in \left\{ 1,2\right\} ,
\end{equation*}%
be defined as%
\begin{equation} \label{eq:homoclinicOperator}
F_{i,j,k}\left( \mathbf{x},\tau ,\alpha \right) :=\left( 
\begin{array}{r@{\,\,-\,\,}l}
P_{j}^{u}\left( x_{0}\right)  & x_{1} \\ 
\phi _{\alpha }\left( x_{1},\tau \right)  & x_{2} \\ 
T_{i}^{-1}(x_{2}) & x_{3} \\ 
\psi _{i}^{c_{j}}\left( x_{3},s_{1}\right)  & x_{4} \\ 
T_{i}(x_{4}) & x_{5} \\ 
\phi _{\alpha }\left( x_{5},s_{2}\right)  & P_{k}^{s}\left( x_{6}\right) 
\end{array}%
\right) .
\end{equation}
One can formulate an analogous result to the Lemmas 
\ref{lem:collision-connections} and \ref{lem:Li-collisions}, 
so that
\[
F_{i,j,k}\left( \mathbf{x}^{\ast },\tau ^{\ast },0\right) =0,
\]
together with $DF_{i,j,k}\left( \mathbf{x}^{\ast },\tau ^{\ast },0\right) $
an isomorphism implies that the manifolds $W^{u}\left( L_{j}\right) $ and 
$W^{s}\left( L_{k}\right) $ intersect transversally.

Again, the advantage of
solving $F_{i,j,k}=0$ over parallel shooting in the original coordinates
is that one can establish the existence of connections which pass
arbitrarily close to a collision $m_{1}$ and/or $m_2$.
Indeed, the operator defined in Equation \eqref{eq:homoclinicOperator}
can be generalized to study homoclinic orbits which make any 
finite number of flybys of the primaries in any order before returning 
to $L_{4,5}$ by making additional changes of variables to regularized
coordinates every time the orbit passes near collision.  



\section{Symmetric periodic orbits passing through collision\label%
{sec:symmetric-orbits}}

In this section we show that our method applies to the study of 
families of periodic orbits which pass through a
collision. By this we mean the following. We will prove the existence of 
a family of orbits parameterized by the value of the Jacobi constant on an interval.
	As in the introduction, we refer to this as a tube of 
periodic orbits.  For all values in the interval except one, the intersection
of the energy level set with the tube is a periodic orbit. For a single
isolated value of the energy the intersection of the energy level set with  
the tube is an ejection-collision orbit involving $m_{1}$. The situation
is depicted in Figure \ref{fig:Lyap}.

\begin{figure}[tbp]
\begin{center}
\includegraphics[height=3.95cm]{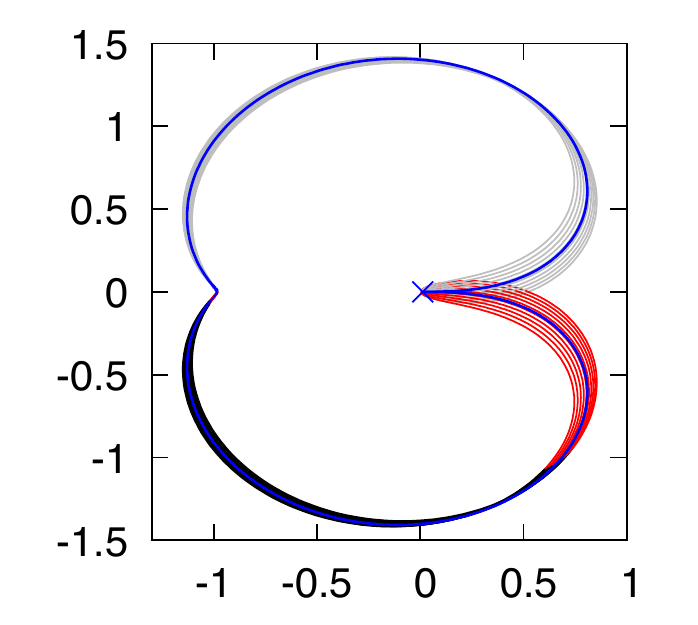} %
\includegraphics[height=3.95cm]{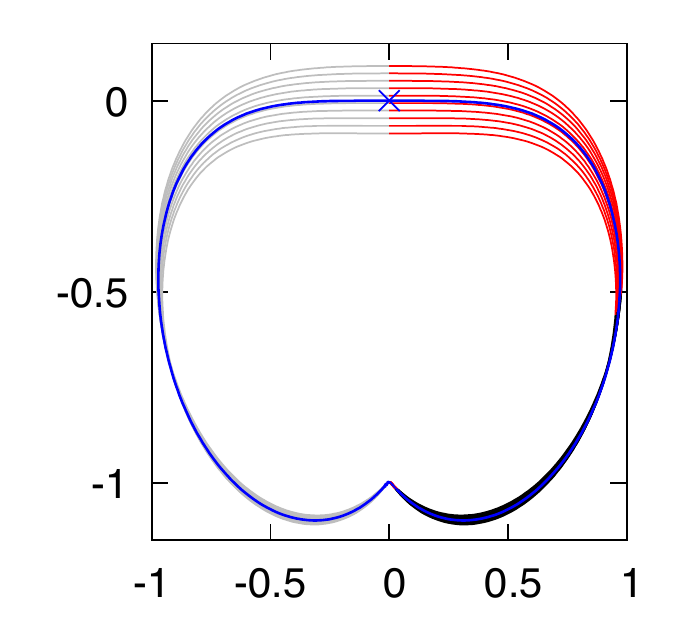} %
\includegraphics[height=3.95cm]{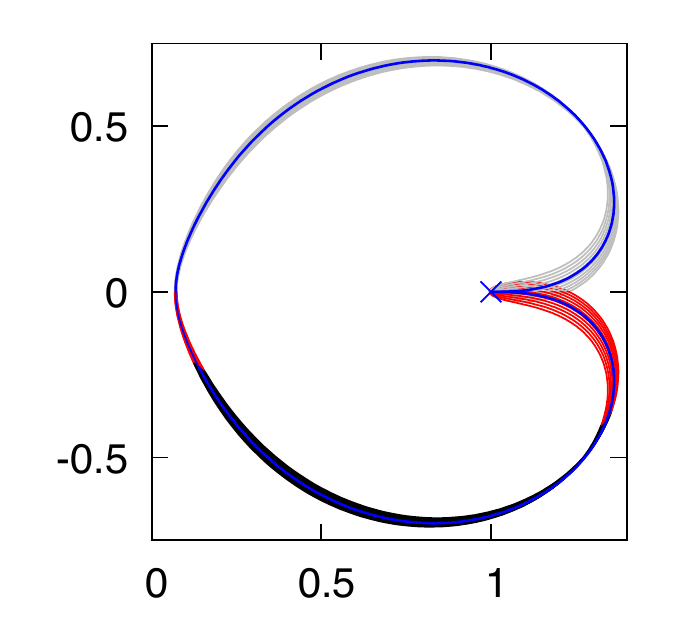}
\par
\includegraphics[height=3.95cm]{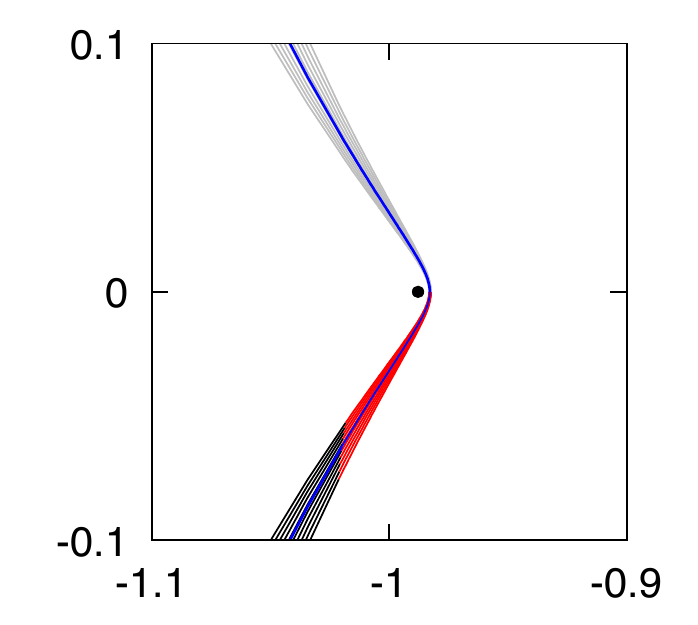} %
\includegraphics[height=3.95cm]{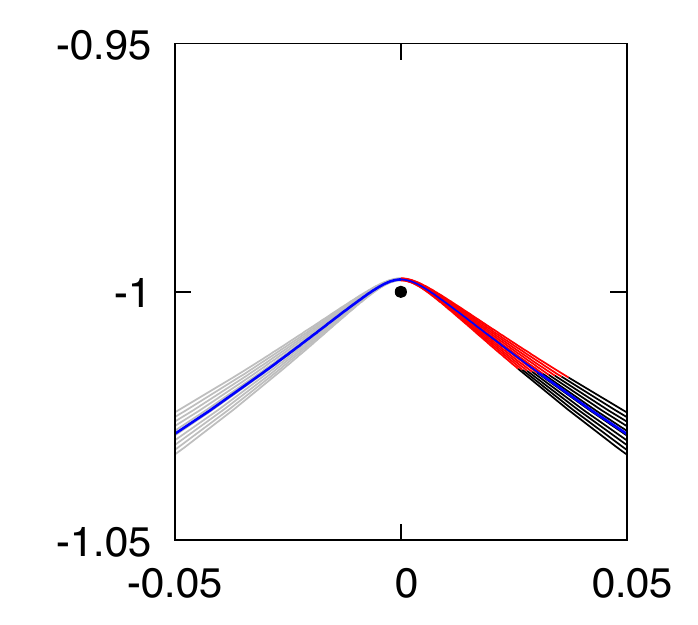} %
\includegraphics[height=3.95cm]{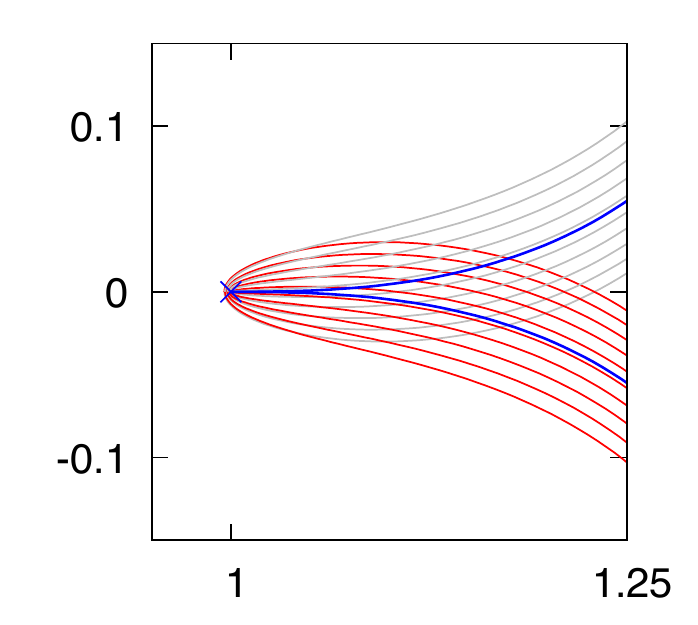}
\end{center}
\caption{A family of Lyapunov periodic orbits passing through a collision.
The left two figures are in the original coordinates, the middle two are in
the regularised coordinates at $m_{1}$ and the right two are in 
regularised coordinates at $m_{2}$. (Compare with Figure \protect\ref%
{fig:PCRTBP_coordinates}.) The trajectories computed in the original
coordinates are in black, and the trajectories computed in the regularized
coordinates are in red. The collision with $m_1$ is indicated by a cross.
The mass $m_2$ is added in the closeup figures as a black dot. The operator (%
\protect\ref{eq:Fc-choice}) gives half of a periodic orbit in red and black.
The second half, which follows from the symmetry, is depicted in grey. The
plots are for the Earth-moon system.}
\label{fig:Lyap}
\end{figure}

\begin{figure}[tbp]
\begin{center}
\includegraphics[height=3.95cm]{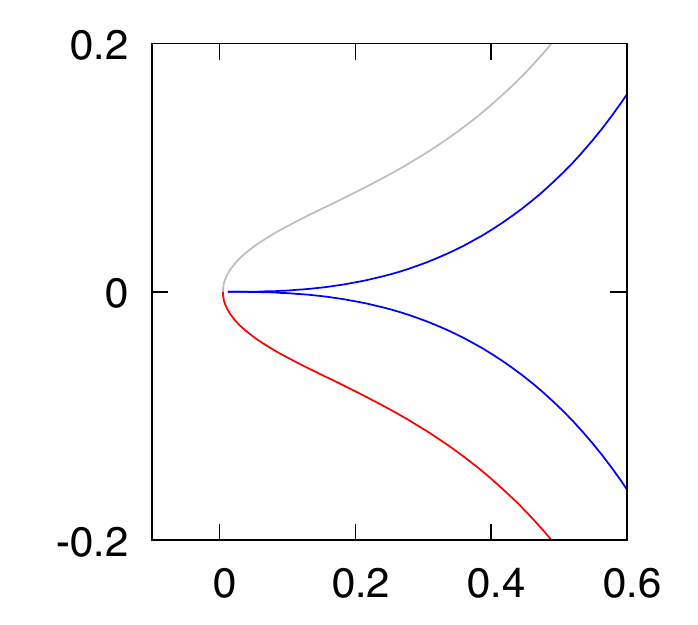} %
\includegraphics[height=3.95cm]{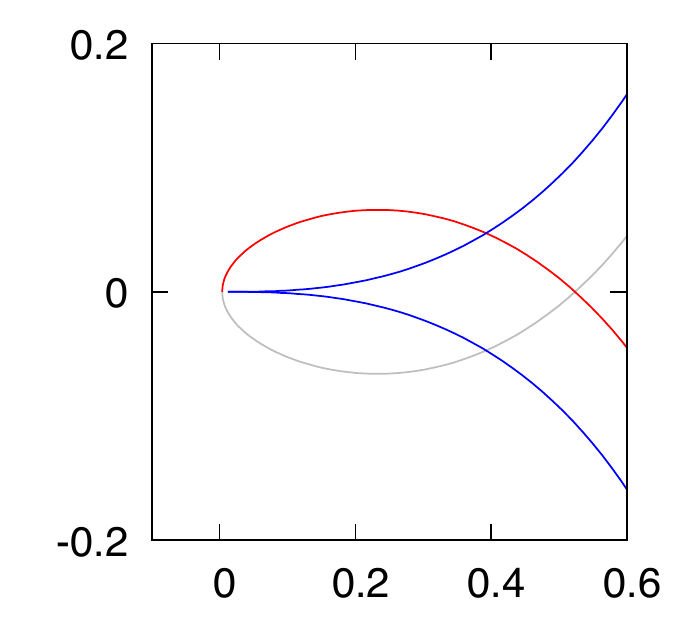} %
\end{center}
\caption{A closeup of a Lyapunov orbit before (left) and after (right) passing through the collision. 
The plot is in the original coordinates.}
\label{fig:Lyap-closeup}
\end{figure}

To establish such a family of periodic orbits we make use of the time
reversing symmetry of the PCRTBP. Recall that for%
\begin{equation*}
S\left( x,p,y,q\right) :=\left( x,-p,-y,q\right)
\end{equation*}%
and for the flow $\phi \left( \mathbf{x},t\right) $ of the PCRTBP we have 
that%
\begin{equation}
S\left( \phi \left( \mathbf{x},t\right) \right) =\phi \left( S\left( \mathbf{%
x}\right) ,-t\right) .  \label{eq:symmetry-prop}
\end{equation}%
Let us introduce the notation $\mathcal{S}$ to stand for the set of self $S$%
-symmetric points%
\begin{equation*}
\mathcal{S:}=\left\{ \mathbf{x}\in \mathbb{R}^{4}:\mathbf{x}=S\left( \mathbf{%
x}\right) \right\} .
\end{equation*}

The property in Equation \eqref{eq:symmetry-prop} is used to find periodic orbits as
follows. Suppose  $\mathbf{x},\mathbf{y}\in \mathcal{S}
$ satisfy $\mathbf{y}=\phi \left( \mathbf{x},t\right)$.  Then by Equation \eqref{eq:symmetry-prop}, 
we have 
\begin{equation}
\phi \left( \mathbf{x},2t\right) =\phi \left( \mathbf{y},t\right) =\phi
\left( S\left( \mathbf{y}\right) ,t\right) =S\left( \phi \left( \mathbf{y}%
,-t\right) \right) =S\left( \mathbf{x}\right) =\mathbf{x},
\label{eq:S-symm-periodic}
\end{equation}%
meaning that $\mathbf{x}$ lies on a periodic orbit. Our strategy is then 
to set up a boundary value problem which shoots from $\mathcal{S%
}$ to itself.

The set $\mathcal{S}$ lies on the $x$-axis in the $\left( x,y\right)$
coordinate frame. From the left plot in Figure \ref{fig:Lyap} it is clear that 
we are interested in points on $\mathcal{S}$ which will pass
through collision with $m_{1}$ and close to the collision with $m_{2}$.
We therefore consider the set $\mathcal{S}$ transformed to the 
regularized coordinates of $m_1$ and $m_2$.

\begin{lemma}
Let $\mathcal{\hat{S}},\mathcal{\tilde{S}}\subset \mathbb{R}^{4}$ be defined
as%
\begin{eqnarray*}
\mathcal{\hat{S}} &=&\left\{ \left( 0,\hat{p},\hat{y},0\right) :\hat{p},\hat{%
y}\in \mathbb{R}\right\} , \\
\mathcal{\tilde{S}} &=&\left\{ \left( \tilde{x},0,0,\tilde{q}\right) :\tilde{%
x},\tilde{q}\in \mathbb{R}\right\} .
\end{eqnarray*}%
Then $T_{1}(\mathcal{\hat{S}})=\mathcal{S}$ and $T_{2}(\mathcal{\tilde{S}})=%
\mathcal{S}$.
\end{lemma}

\begin{proof}
The proof follows directly from the definition of $T_{1}$ and $T_{2}$. (See Equations \eqref{eq:T1-def} and \eqref{eq:T2-def}.)
\end{proof}

The intuition behind the choice of $\mathcal{\hat{S}},$ $\mathcal{\tilde{S}}$
is seen in Figure \ref{fig:PCRTBP_coordinates}. From the figure we see that
the set $\mathcal{\hat{S}}$ is the vertical axis $\{\hat{x}=0\}$ and $\mathcal{\tilde{S}}$ is
the horizontal axis $\left\{ \tilde{y}=0\right\} $, which join the primaries
in the regularized coordinates.

To find the desired symmetric periodic orbits we fix an energy level $c\in \mathbb{R}$
and introduce an appropriate shooting operator, whose zero implies the
existence of an orbit with energy $c$. Slightly abusing notation,
let us first define two functions $\hat{p},\tilde{q}:\mathbb{R}%
^{2}\rightarrow \mathbb{R}$ as%
\begin{eqnarray*}
\hat{p}\left( \hat{y},c\right)  & := &\sqrt{4\hat{y}^{6}-8\mu \hat{y}%
^{4}+4(\mu -c)\hat{y}^{2}+\frac{8\mu \hat{y}^{2}}{\sqrt{\hat{y}^{4}+1-2\hat{y%
}^{2}}}+8(1-\mu )}, \\
\tilde{q}\left( \tilde{x},c\right)  &:=&\sqrt{4\tilde{x}^{6}-8(1-\mu )\tilde{%
x}^{4}+4\left( (1-\mu )-c\right) \tilde{x}^{2}+\frac{8(1-\mu )\tilde{x}^{2}}{%
\sqrt{\tilde{x}^{4}+1-2\tilde{x}^{2}}}+8\mu }.
\end{eqnarray*}%
Observe that from Equations \eqref{eq:reg_P_energy} and \eqref{eq:E2}
we have
\begin{align}
E_{1}^{c}\left( 0,\hat{p}\left( \hat{y},c\right) ,\hat{y},0\right) & =0,
\label{eq:pc-implicit} \\
E_{2}^{c}\left( \tilde{x},0,0,\tilde{q}\left( \tilde{x},c\right) \right) &
=0.  \label{eq:qc-implicit}
\end{align}%
Next, we define $P_{1}^{c},P_{2}^{c}:\mathbb{R}\rightarrow 
\mathbb{R}^{4}$ by
\begin{align*}
\hat{P}_{1}^{c}\left( \hat{y}\right) & :=\left( 0,\hat{p}\left( \hat{y}%
,c\right) ,\hat{y},0\right) , \\
\tilde{P}_{2}^{c}\left( \tilde{x}\right) & :=\left( \tilde{x},0,0,\tilde{q}%
\left( \tilde{x},c\right) \right),
\end{align*}%
and note that $P_{1}^{c}\left( \mathbb{R}\right) \subset \mathcal{\hat{S}}$ and $%
P_{2}^{c}\left( \mathbb{R}\right) \subset \mathcal{\tilde{S}}$. 
Taking 
\begin{equation*}
\mathbf{x}=(x_{0}, x_{1},\ldots ,x_{5},x_{6})\in 
\mathbb{R}\times \underset{5 \  \text{copies}}{\underbrace{\mathbb{R}^{4}\times \ldots \times 
\mathbb{R}^{4}}}\times \mathbb{R}=\mathbb{R}^{22}\mathbb{,}
\end{equation*}%
we define the shooting operator $F_{c}:\mathbb{R}^{24}\rightarrow 
\mathbb{R}^{24}$ as 
\begin{equation}
F_{c}\left( \mathbf{x},\tau ,\alpha \right) =\left( 
\begin{array}{r@{\,\,-\,\,}l}
\hat{P}_{1}^{c}\left( x_{0}\right) & x_{1} \\ 
\psi _{1}^{c}\left( x_{1},s\right) & x_{2} \\ 
T_{1}\left( x_{2}\right) & x_{3} \\ 
\phi _{\alpha }\left( x_{3},\tau \right) & x_{4} \\ 
T_{2}^{-1}\left( x_{4}\right) & x_{5} \\ 
\psi _{2}^{c}\left( x_{5},s\right) & \tilde{P}_{2}^{c}\left( x_{6}\right)%
\end{array}%
\right) .  \label{eq:Fc-choice}
\end{equation}%
We have the following result.

\begin{lemma}
\label{lem:Lyap-existence} Suppose that for $c\in \mathbb{R}$ we have
an $\mathbf{x}\left( c\right) \in\mathbb{R}^{22}$ and $\tau \left( c\right)
\in \mathbb{R}$ for which 
\begin{equation*}
F_{c}\left( \mathbf{x}\left( c\right) ,\tau \left( c\right) ,0\right) =0,
\end{equation*}
then we have one of the following three cases:

\begin{enumerate}
\item If $x_{0}\left( c\right) \neq 0$ and $x_{6}\left( c\right) \neq 0$,
then the orbit  through
$T_{1}( \hat{P}_{1}^{c}\left( x_{0}\left( c\right) \right) )$ 
is periodic.

\item If $x_{0}\left( c\right) =0$ and $x_{6}\left( c\right) \neq 0$, then
then the orbit through
$T_{1}( \hat{P}_{1}^{c}\left( x_{0}\left( c\right) \right) )$ is an ejection-collision with $m_1$. 

\item If $x_{0}\left( c\right) \neq 0$ and $x_{6}\left( c\right) =0$, then
then the orbit  through
$T_{1}( \hat{P}_{1}^{c}\left( x_{0}\left( c\right) \right) )$ is an ejection-collision with $m_2$. 

\end{enumerate}
\end{lemma}

\begin{proof}
The result follows immediately from the definition of $F_{c}$
in Equation \eqref{eq:Fc-choice} and from Theorem \ref{thm:LeviCivitta} (or the analogous theorem for $m_2$).  We highlight the fact that due to Equations \eqref{eq:pc-implicit}--\eqref{eq:qc-implicit} we have $E_{1}^{c}(\hat{P}%
_{1}^{c}\left( x_{0}\right) )=0$ and $E_{2}^{c}( \tilde{P}%
_{2}^{c}\left( x_{6}\right) ) =0$, so the trajectories in the
regularized coordinates correspond to the physical trajectories in the
physical coordinates of the PCRTBP. 
\end{proof}

We can use the implicit function theorem to compute the derivative of $\mathbf{x}%
\left( c\right) $ with respect to $c$. Let us write $\mathbf{y}\left(
c\right) :=\left( \mathbf{x}\left( c\right) ,\tau \left( c\right) ,\alpha
\left( c\right) \right) $ and suppose $F_c(\mathbf{y}(c))=0$. (Note that in fact we must also have that $\alpha \left( c\right) =0$ since $\alpha $ is unfolding.) Then $\frac{d}{dc}\mathbf{x}\left(
c\right) $ is computed from the first coordinates of the vector $\frac{d}{dc}%
\mathbf{y}\left( c\right) $ and is given by the formula
\begin{equation}
\frac{d}{dc}\mathbf{y}\left( c\right) =-\left( \frac{\partial F_{c}}{%
\partial \mathbf{y}}\right) ^{-1}\frac{\partial F_{c}}{\partial c}.
\label{eq:implicit-dx-dc}
\end{equation}


\begin{theorem}
\label{th:Lyap-through-collision}Assume that for $c\in \left[ c_{1},c_{2}%
\right] $ the functions $\mathbf{x}\left( c\right) $ and $\tau \left(
c\right) $ solve the implicit equation 
\begin{equation*}
F_{c}\left( \mathbf{x}\left( c\right) ,\tau \left( c\right) ,0\right) =0.
\end{equation*}
If%
\begin{eqnarray}
	\label{eq:Bolzano-condition-Lyap} 
x_{0}\left( c_{1}\right) >0>x_{0}\left( c_{2}\right) , \\
\label{eq:x6-nonzero} 
x_{6}\left( c\right) \neq 0\qquad \text{for all }c\in \left[ c_{1},c_{2}\right],
\end{eqnarray} 
and 
\begin{equation}
\frac{d}{dc}x_{0}\left( c\right) <0\qquad \text{for all }c\in \left[
c_{1},c_{2}\right] ,  \label{eq:der-cond-Lyap}
\end{equation}%
then there exists a unique energy parameter $c^{\ast }\in \left(
c_{1},c_{2}\right) $ for which we have have an intersection of the ejection
and collision manifolds of $m_{1}$. Moreover, for all remaining $c\in \left[
c_{1},c_{2}\right] \setminus \left\{ c^{\ast }\right\} $ the orbit of the point $%
T_{1}( \hat{P}_{1}^{c}\left( x_{0}\left( c\right) \right) ) $
is periodic.
\end{theorem}
\begin{proof}
The result follows directly from the Bolzano theorem and Lemma \ref%
{lem:Lyap-existence}.
\end{proof}

Theorem \ref{th:Lyap-through-collision} is deliberately formulated so that its hypotheses 
can be validated via computer assistance. Specifically, rigorous enclosures of Equation 
\eqref{eq:implicit-dx-dc} are rigorously computed and Equations 
\eqref{eq:Bolzano-condition-Lyap}-\eqref{eq:der-cond-Lyap} are rigorously verified using interval arithmetic.

\medskip

We finish this section with an example of a similar approach, which can be
used for the proofs of double collisions in the case when $m_{1}=m_{2}=\frac{%
1}{2}$. That is, we establish the existence of a family of periodic orbits,
parameterized by energy (the Jacobi constant), which are symmetric with respect to
the $y$-axis, and such that for a single parameter from the family we have a
double collision as in Figure \ref{fig:eq}.

\begin{figure}[tbp]
\begin{center}
\includegraphics[height=3.95cm]{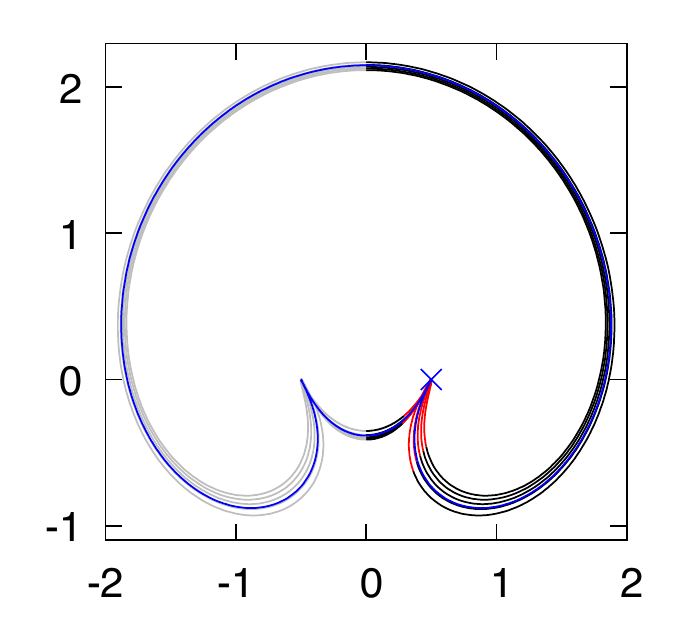} %
\includegraphics[height=3.95cm]{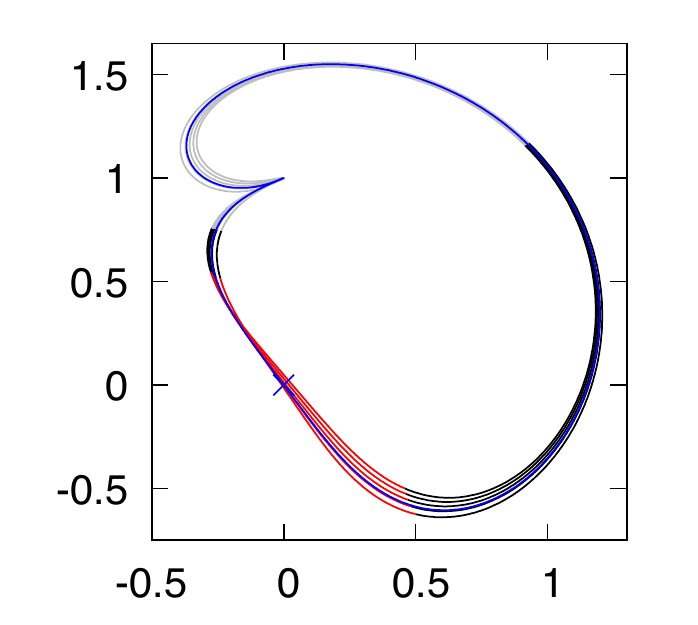}
\end{center}
\caption{A family of periodic orbits passing through a double collision. The
left figure is in the original coordinates and the right figure is in the
regularised coordinates at $m_{1}$. The trajectories computed in the
original coordinates are in black, the trajectories computed in the
regularized coordinates are in red, and the collision orbit is in blue. The
second half of an orbit, which follows from the $R$-symmetry, is depicted in
grey. The plots are for the system with equal masses.}
\label{fig:eq}
\end{figure}

%

In this case consider $R:\mathbb{R}^{4}\rightarrow \mathbb{R}^{4}$ defined as%
\begin{equation*}
R\left( x,p,y,q\right) =\left( -x,p,y,-q\right) .
\end{equation*}%
For the case of two equal masses, we have the time reversing symmetry 
\begin{equation}
R\left( \phi \left( \mathbf{x},t\right) \right) =\phi \left( R\left( \mathbf{%
x}\right) ,-t\right) .  \label{eq:R-symmetry.}
\end{equation}%
We denote by $\mathcal{R}$ the set of all points which are $R$-self
symmetric, i.e. $\mathcal{R}=\{\mathbf{x}=R\left( \mathbf{x}\right) \}$. An
argument mirroring Equation \eqref{eq:S-symm-periodic} shows that if 
two points $\mathbf{x},\mathbf{y}\in \mathcal{R}$ have $%
\mathbf{y}=\phi \left( \mathbf{x},t\right) ,$ then these points must lie on a
periodic orbit.

To obtain the existence of the family of orbits 
depicted in Figure \ref{fig:eq}, define $p:%
\mathbb{R}^{2}\rightarrow \mathbb{R}$ and $P_{1}^{c},P_{2}^{c}:\mathbb{R}%
\rightarrow \mathbb{R}^{4}$ as%
\begin{eqnarray*}
p\left( y,c\right)  &:=&\sqrt{2\Omega (0,y)-c}, \\
P_{1}^{c}\left( y\right)  &:=&\left( 0,p\left( y,c\right) ,y,0\right) , \\
P_{2}^{c}\left( y\right)  &:=&\left( 0,-p\left( y,c\right) ,y,0\right) .
\end{eqnarray*}%
Note that $P_{1}^{c}\left( y\right) ,P_{2}^{c}\left( y\right) \in \mathcal{R}
$ and $E\left( P_{1}^{c}\left( y\right) \right) =E\left( P_{2}^{c}\left(
y\right) \right) =c$ (see Equation \eqref{eq:JacobiIntegral}). Consider $%
x_{0},x_{7}\in \mathbb{R}$ and $x_{1},\ldots ,x_{6}\in \mathbb{R}^{4},$
where 
\begin{equation}
x_{4}=\left( s_{4},\hat{p}_{4},\hat{y}_{4},\hat{q}_{4}\right) \in \mathbb{R}%
^{4}. \label{eq:s4}
\end{equation}%
We emphasize that the first coordinate in $x_{4}$ will be used here in a
slightly less standard way than in the previous examples. We define also 
\begin{equation*}
\mathrm{\hat{x}}_{4}:=\left( 0,\hat{p}_{4},\hat{y}_{4},\hat{q}_{4}\right)
\in \mathbb{R}^{4}.
\end{equation*}%
%
%
We now choose some fixed $s_{2},s_{5}\in \mathbb{R}$, $s_{2},s_{5}>0$, and for
\begin{equation*}
\mathbf{x}=\left( x_{0},\ldots ,x_{7}\right) \in \mathbb{R}\times \underset{6%
}{\underbrace{\mathbb{R}^{4}\times \mathbb{\ldots }\times \mathbb{R}^{4}}}%
\times \mathbb{R}=\mathbb{R}^{26}
\end{equation*}%
define the operator $F_{c}:\mathbb{R}^{26}\times \mathbb{R}\times \mathbb{R%
}\rightarrow \mathbb{R}^{28}$ as 
\begin{equation}
F_{c}\left( \mathbf{x},\tau ,\alpha \right) =\left( 
\begin{array}{r@{\,\,-\,\,}l}
P_1^{c}\left( x_{0}\right)  & x_{1} \\ 
\phi _{\alpha }\left( x_{1},s_{2}\right)  & x_{2} \\ 
T_{1}^{-1}\left( x_{2}\right)  & x_{3} \\ 
\psi _{1}^{c}\left( x_{3},s_{4}\right)  & \mathrm{\hat{x}}_{4} \\ 
\psi _{1}^{c}\left( \mathrm{\hat{x}}_{4},s_{5}\right)  & x_{5} \\ 
T_{1}\left( x_{5}\right)  & x_{6} \\ 
\phi _{\alpha }\left( x_{6},\tau \right)  & P_2^{c}\left( x_{7}\right) 
\end{array}%
\right) .  \label{eq:Fc-equal}
\end{equation}

Note that in Equation \eqref{eq:Fc-equal} the $s_{2},s_{5}$ are some fixed
	parameters, and $s_{4}$ is one of the coordinates of $\mathbf{x}$. We claim that if $F_{c}\left( \mathbf{x},\tau ,0,0\right) =0$ and $\pi _{\hat{y}_{4}}%
\mathbf{x}=0$, then the orbit of $x_2$ passes through the collision with 
$m_{1}$.  This is because $\mathrm{\hat{x}}_{4}=\left( 0,\hat{p}_{4},\hat{y}
_{4},\hat{q}_{4}\right) $, so that $F_{c}=0$ ensures that the point 
$\psi _{1}^{c}\left( x_{3},s_{4}\right)$ is zero on the $\hat x_4$ coordinate.
So, if $F_{c}(\mathbf{x})=0$ and $\pi _{\hat{y}_{4}}\mathbf{x}=0$, 
then $\pi_{\hat x_4, \hat y_4}\psi _{1}^{c}\left( x_{3},s_{4}\right)=0$ and we arrive at the collision.
Moreover, by the $R$-symmetry of the system in this case we
also establish heteroclinic connections between collisions with 
$m_{1}$ and $m_{2}$ (see Figure \ref{fig:eq}). 

If on the other hand $F_{c}=0$ and $\pi _{\hat{y}_{4}}\mathbf{x}\neq 0,$ 
then we have a periodic orbit passing near
the collisions with $m_{1}$ and $m_{2}$.
One can prove a result analogous to Theorem \ref{th:Lyap-through-collision}
 with the minor difference being that instead of using $x_{0}$ in 
 Equations \eqref{eq:Bolzano-condition-Lyap} and \eqref{eq:der-cond-Lyap} we take $
\hat{y}_{4}.$ We omit the details in order not to repeat the same argument.



\section{Computer assisted proofs for collision/near collision orbits}

\label{sec:CAP}

\subsection{Newton-Krawczyk method}

For a smooth mapping $F : \mathbb{R}^n \to \mathbb{R}^n$, the following
theorem provides sufficient conditions for the existence of a solution of $%
F(x)=0$ in the neighborhood of a \textquotedblleft good
enough\textquotedblright\ approximate solution. The hypotheses of the
theorem require measuring the defect associated with the approximate
solution, as well as the quality of a certain condition number for an
approximate inverse of the derivative. Theorems of this kind are used widely
in computer assisted proofs, and we refer the interested reader to the works
of \cite{MR0231516,MR1100928,MR1057685,MR2807595,MR2652784,
MR3971222,MR3822720,jpjbReview} for a more complete overview.

Let $\left\Vert \cdot \right\Vert $ be a norm in $\mathbb{R}^{n}$ and let $%
\overline{B}(x_{0},r)\subset \mathbb{R}^{n}$ denote a closed ball of radius $r \geq 0$ centered at $x_0$  in that norm.

\begin{theorem}[Newton-Krawczyk]
\label{thm:NK} \label{thm:aPosteriori} Let $U\subset \mathbb{R}^{n}$ be an
open set and $F\colon U\rightarrow \mathbb{R}^{n}$ be at least of class $C^2$. Suppose
that $x_{0}\in U$ and let $A$ be a $n\times n$ matrix. Suppose that $Y,Z,r>0$
are positive constants such that $\overline{B}(x_{0},r)\subset U$ and 
\begin{eqnarray}
\Vert AF(x_{0})\Vert &\leq &Y,  \label{eq:Krawczyk-Y} \\
\sup_{x\in \overline{B}(x_{0},r)}\Vert \mathrm{Id}-ADF(x)\Vert &\leq &Z.
\label{eq:Krawczyk-Z}
\end{eqnarray}%
If  
\begin{equation}
Zr-r+Y\leq 0,  \label{eq:Krawczyk-ineq}
\end{equation}%
then there is a unique $\hat{x}\in \overline{B}(x_{0},r)$ for which $F(\hat{x%
})=0.$ Moreover, $DF(\hat{x})$ is invertible.
\end{theorem}

\begin{proof}
The proof is included in \ref{sec:proof} for the sake of completeness.
\end{proof}

\bigskip

The theorem is well suited for applications to computer assisted proofs. To
validate the assumptions its enough to compute interval enclosures of the
quantities $F(x_{0})$ and $DF(B)$, where $B$ is a suitable 
ball. These enclosures are done using interval arithmetic,
and the results are returned as sets (cubes in $\mathbb{R}^{n}$ and $\mathbb{%
R}^{n\times n}$) enclosing the correct values. A good choice for the matrix $%
A$ is any floating point approximate inverse of the derivative of $F$ at $%
x_{0}$, computed with standard linear algebra packages. The advantage of
working with such an approximation  is that there is no need to compute a rigorous interval enclosure of
a solution of a linear equation
(as in the interval Newton
method). In higher dimensional problems,
solving linear equations can lead to large overestimation (the so called 
``wrapping effect'').

In our work the evaluation of $F$ and its derivative involves integrating
ODEs and variational equations. There are well know general purpose
algorithms for solving these problems, and we refer the 
interested reader to \cite{c1Lohner,cnLohner,MR2807595}. 
For parameterizing the invariant manifolds attached to 
$L_4$ with interval enclosures, we exploit the techniques discussed in \cite{myNotes}
(validated integration is also discussed in this reference).

We remark that our implementations use the IntLab laboratory running under MatLab\footnote{%
https://www.tuhh.de/ti3/rump/intlab/} and/or the CAPD\footnote{%
Computer Assisted Proofs in Dynamics, http://capd.ii.uj.edu.pl} C\texttt{++}
library, and recall that the source codes are found at the homepage of MC.
See \cite{Ru99a} and \cite{CAPD_paper} as references for the usage
and the functionality of the libraries.

\subsection{Computer assisted existence proofs for ejection-collision orbits}
\label{sec:EC}

The methodology of Section \ref{sec:ejectionToCollision}, and
especially Lemma \ref{lem:collision-connections}, is combined with Theorem \ref{thm:NK}
to obtain the following.

\begin{maintheorem}\label{thm:ejectionCollision}
\label{thm:CAP-ejCol} Consider the 
planar PCRTBP with $\mu = 1/4$ 
and $c = 3.2$.  Let 
\[
\overline{p} = 
\left(
\begin{array}{c}
-0.564897282072410 \\ 
\phantom{-}0.978399619177283 \\
-0.099609551141525 \\
-0.751696444982537
\end{array}
\right), 
\]
\[
r = 2.7 \times 10^{-13}, 
\]
and 
\[
B_r = \left\{ x \in \mathbb{R}^4 \, : \|x - \overline{p}\| \leq r\right\},
\]
where the norm is the maximum norm on components. 
Then, there exists a unique $p_* \in B_r$ such that the orbit of $p_*$
is ejected from $m_2$ (at $x = -1 + \mu, y= 0$),  collides with  $m_1$ (at $x = \mu, y= 0$), and 
\correction{comment 10}{
the total time $T$ (in synodic/un-regularized coordinates)
 from ejection to collision satisfies
\begin{equation*}
2.42710599795 \leq T \leq 2.42710599796.
\end{equation*}
}


In addition, the ejection manifold of $m_2$
intersects the collision manifold of $m_1$ transversely 
along the orbit of $p_*$, where transversality is 
relative to the level set $\setof*{E = 3.2}$. Moreover, there exists a transverse $S$-symmetric counterpart ejected from $m_1$ and colliding with $m_2$.
\end{maintheorem}

\begin{proof}
The first step in the proof is to define an appropriate version of the map $%
F $ in Equation \eqref{eq:collisionOperator},
whose zeros correspond to ejection-collision
orbits from $m_2$ to $m_1$.
In particular we set $k = 2$
and $l = 1$, and choose (somewhat arbitrarily) the parameter $s = 0.35$ in
the definition of the component maps $R_{\tau, \alpha}^1$ and $R_{\tau,
\alpha}^5$. The parameter $s$ determines how long to integrate/flow in the
regularized coordinates.

Next we compute an approximate zero 
$\overline{x} \in \mathbb{R}^{24}$ 
of $F$ using Newton's method.  Note that interval arithmetic is not 
required in this step. The resulting numerical data is recorded in Table 
\ref{table:th1}, and we note that $\overline{x}_3$ in the table 
corresponds to $\overline{p}$ in the hypothesis of the theorem.
Note also that we take $\bar\alpha$ in the approximate solution to be zero.

\begin{table}[tbp]
{\scriptsize 
\begin{tabular}{cllll}
\hline
$\overline{x}_0 = $ & $\phantom{(-}2.945584780500716$ &  &  &    \\ 
$\overline{x}_1 = $ & $(\phantom{-} 0.0,$ & $-1.387134030283961,$ & $\phantom{-}0.0,$ & $%
\phantom{-}0.275425456390970)$   \\ 
$\overline{x}_2 = $ & $(-0.444581369966432,$ & $-1.038375926396089,$
& $\phantom{-}0.112026231721142,$ & $\phantom{-}0.449167625710802)$   \\ 
$\overline{x}_3 = $ & $(-0.564897282072410,$ & $\phantom{-}0.978399619177283,$ & 
$-0.099609551141525,$ & $-0.751696444982537)$   \\ 
$\overline{x}_4 = $ & $(-0.244097430449606,$ & $\phantom{-}0.878139982728136,$ & $-0.025435855606099,$ & $\phantom{-}0.543608549989376)$   \\ 
$\overline{x}_5 = $ & $( \phantom{-}0.018086991443589,$ & $-0.732714475912918,$
& $-0.703153304556756,$ & $\phantom{-}1.254598547822042)$   \\ 
$\overline{x}_6 = $ & $\phantom{(-}1.459760691418490$ &  &  &    \\ 
$\overline{\tau} = $ & $\phantom{(-}2.051635871465197$ &  &  &    \\ 
$\overline{\alpha} = $ & $\phantom{(-}0.0$ &  &  &   \\
\hline 
\end{tabular}
}
\caption{ Numerical data used in the proof of Theorem \protect\ref%
{thm:CAP-ejCol}, giving the approximate solution of $F=0$ for the operator \eqref{eq:collisionOperator},
whose zeros correspond to the ejection-collision orbits from $m_2$ to $m_1$. 
We set the mass ratio to $\mu = 1/4$ 
and Jacobi constant to $c = 3.2$.
The resulting orbit is illustrated in Figure \ref{fig:ejectionCollisions} (bottom curve).
\label{table:th1}\label{tab:ejColTab1} }
\end{table}

We define $A$ to be the numerically computed approximate inverse of $DF(%
\overline{x})$, and let 
\begin{equation*}
B = \overline{B}(\overline{x}, r_*),
\end{equation*}
denote the closed ball of radius 
\begin{equation*}
r_* =  2\times 10^{-12},
\end{equation*}
in the maximum norm about the numerical approximation. 
(The reader interested in the numerical entries of the Matrix can 
run the accompanying computer program).  We note that the choice of 
$r_*$ is somewhat arbitrary. (It should be small enough that there is not 
too much ``wrapping'', but not so small that there is no $r \leq r_*$
satisfying the hypothesis of Theorem \ref{thm:NK}).

Using interval
arithmetic and validated numerical integration we compute an interval
enclosure of the length $24$ vector of intervals $\mathbf{F}$ having 
\begin{equation*}
F(\overline{x}) \in \mathbf{F},
\end{equation*}
and an interval enclosure of a $24 \times 24$ interval matrix $\mathbf{M}$
with 
\begin{equation*}
DF(x) \in \mathbf{M} \quad \quad \mbox{for all } x \in B.
\end{equation*}
We then check, again using interval arithmetic, that 
\begin{equation*}
\|A \mathbf{F} \| \in   10^{-12} \times 
[   0.0,   0.26850976470521]
\end{equation*}
and that 
\begin{equation*}
\|\mbox{Id} - A \mathbf{M} \| \in 10^{-7} \times 
[   0.0,   0.23119622467860]. 
\end{equation*}
From these we have 
\begin{equation*}
\|A F(\overline{x}) \| \leq Y < 0.269 \times 10^{-12}
\end{equation*}
and 
\begin{equation*}
\sup_{x \in B} \|\mbox{Id} - A DF(x)\|\leq Z < 0.232 \times 10^{-7},
\end{equation*}
though the actual bounds stored in the computer are tighter than those
just reported (hence the inequality).

We let
\begin{equation*}
r = \sup\left( \frac{Y}{1- Z}\right) \leq 2.7 \times 10^{-13},
\end{equation*}
and note again that the actual bound stored in the computer is 
smaller than reported here.  
We then check, using interval arithmetic, that 
\begin{equation*}
Z r - r + Y \leq  - 5.048 \times 10^{-29} < 0.
\end{equation*}
We also note that, since $r \leq r_*$, we have that $\overline{B}(\overline{%
x}, r) \subset B$, so that 
\begin{equation*}
\sup_{x \in \overline{B}(\overline{x}, r)} \| \mbox{Id} - A DF(x)\|
\leq Z,
\end{equation*}
on the smaller ball as well.

From this we conclude, via Theorem \ref{thm:NK}, that there exists a unique $%
x_* \in \overline{B}(\overline{x}, r) \subset \mathbb{R}^{24}$
so that $F(x_*) =0$, and moreover that $DF(x_*)$ is invertible. 
Hence, it now follows from Lemma \ref{lem:collision-connections} that there exists a transverse ejection-collision 
from $m_2$ to $m_1$ in the PCRTBP.

Note that the integration time in the standard coordinates
\begin{equation*}
\bar \tau = 2.051635871465197,
\end{equation*}
is one of the variables of $F$ (we are simply reading this off the table).
The rescaled integration time in the regularized coordinates is fixed to be $%
s = 0.35$. Our programs compute validated bounds on the integrals in Equation %
\eqref{eq:time-between-collisions} and provide interval enclosures for the time each orbit spends in the regularized coordinate systems of $%
m_1$ and $m_2$ respectively.  This interval enclosure is 
\begin{equation*}
T_1 + T_2 \in [ 0.27116751585137, 0.27116751585615] +  [ 0.10430261063473, 0.10430261063793].
\end{equation*}
Since the true integration time $\tau_*$ is in an $r$-neighborhood 
of $\bar\tau$ it follows that 
\begin{equation*}
\tau_* \in [ 2.05163587146492, 2.05163587146547].
\end{equation*}
Interval addition of the three time intervals 
containing $T_1$, $T_2$ and $\tau_*$
provides the desired final bound on the total time of flight
given in the theorem.  

The connection in the other direction follows from the $S$-symmetry of the system (see Equation \eqref{eq:symmetry-prop}). The computational part of the proof is implemented in IntLab running under MatLab, and took 21 minutes
to run on a standard desktop computer.
\end{proof}
\medskip

The orbit whose existence is proven in Theorem \ref{thm:ejectionCollision} is illustrated in Figure \ref{fig:ejectionCollisions} (lower orbit of the two orbits illustrated in the figure).  The higher orbit follows from the $S$-symmetry of the PCRTBP.  
We remark that our implementation actually subdivides the time steps $s = 0.35$  in regularized 
coordinates 50 times, while the time step $\bar \tau$ is subdivided 200 times.  
This only enlarges the size of 
the system of equations as discussed in Remark \ref{rem:additionalShooting}.

Validation of the $50 + 200 + 50 = 300$ steps of Taylor integration, along with the spatial and 
parametric variational equations, takes most of the computational
time for the proof. 
The choice of the mass $\mu = 1/4$ and the energy $c = 3.2$ was more or less arbitrary and the existence of many similar orbits could be proven using the same method.

\subsection{Connections between ejections/collisions and the libration
points $L_4$, $L_5$} \label{sec:EC_to_L4_proofs}

We apply the methodology of Section \ref{sec:L4_to_collision}, and
especially Lemma \ref{lem:Li-collisions}, in conjunction with Theorem \ref%
{thm:NK} to obtain the following result. The
local stable (or unstable) manifolds at $L_4$ are computed 
using the methods and implementation of \cite{MR3906230}.
See \ref{sec:manifolds} for a few additional remarks concerning the parameterizations.

\begin{maintheorem}
\label{thm:CAP-L4-to-collision} Consider planar PCRTBP with $\mu = 1/2$ 
and $c = 3$ is the energy of $L_4$.  Let 
\[
\overline{p} = \left(
\begin{array}{c}
 \phantom{-}0.003213450375413 \\
 \phantom{-}0.197716496638868 \\
 -0.404375730348827 \\
   \phantom{-}0.696149210661807 \\
\end{array}
\right),
\]
\[
r = 8.2 \times 10^{-12},
\]
and 
\[
B_r = \left\{ x \in \mathbb{R}^4 \, \colon \, \| x - \bar{p}\| \leq r \right\}.
\]

Then there exists a unique point
\[
p_* \in B_r
\]
such that the orbit of $p_*$ accumulates to $L_4$ as $t \to - \infty$, collides with 
$m_1$ (located at $x = \mu, y= 0$) in finite forward time, and the unstable manifold of $L_4$ intersects 
the collision set of $m_1$ transversely along the orbit of
$p_*$, where transversality is relative to level set 
$\setof*{E = 3}$.

\end{maintheorem}
\begin{table}[t!] 
 {\scriptsize 
\begin{tabular}{cllll}
\hline
$\overline{x}_0 = $ & $\phantom{(-}0.329444389425640$ &  &  &    \\ 
$\overline{x}_1 = $ 
& $( -0.032305434322402,$ &
  $-0.044152238388004,$ &   $\phantom{-}0.843244687835647,$ &   $0.005057045291404)$   \\ 
$\overline{x}_2 = $ & $(  \phantom{-}0.003213450375413,$ &
   $\phantom{-}0.197716496638868,$ &
  $-0.404375730348827,$ &
   $0.696149210661807)$   \\ 
$\overline{x}_3 = $ & $ (  \phantom{-}0.268116630482827,$ &
  $-0.943915863314079,$ &
  $-0.754104155383092,$ &
   $0.671496024758153)$   \\ 
$\overline{x}_4 = $ & $\phantom{(-}1.696671399505923$ &  &  &    \\ 
$\overline{\tau} = $ & $\phantom{(-}7.034349085576677$ &  &  &    \\ 
$\overline{\alpha} = $ & $\phantom{(-}0.0$ &  &  &    \\ 
\hline
\end{tabular}}
\caption{ Numerical data providing an approximate zero of the 
 map $F^u_{i,j}$ defined in Equation \eqref{eq:EC_to_L4_operator}, 
for $i = 1$, $j= 4$, $c=3$, $\mu=1/2$ and $s = 0.5$.
The data is used in the proof of Theorem \protect\ref{thm:CAP-L4-to-collision}, 
and results in the existence of the $L_4$ to collision orbit illustrated in the right frame of 
Figure \ref{fig:EC_to_collision}.  \label{tab:L4-to-collisionTab1}  }
\end{table}
\begin{proof}
The proof is similar to the proof of Theorem \ref{thm:ejectionCollision}, 
and we only sketch the argument.  Orbits accumulating 
to $L_4$ in backward time and colliding with $m_1$ are equivalent to 
zeros of the mapping 
$F^u_{i,j}$ defined in Equation \eqref{eq:EC_to_L4_operator} with $j = 4$
and $i = 1$.  We also set the parameter $s = 0.5$, which is the integration time
in the regularized coordinates.  

The first step is to compute a numerical zero 
$\bar{x} = (\bar x_0,\bar x_1,\bar x_2,\bar x_3,\bar x_4,\bar \tau,\bar \alpha) \in \mathbb{R}^{16}$ of $F^u_{i,j}$.  
This step exploits Newton's method (no interval arithmetic necessary), and  
the resulting data is reported in Table \ref{tab:L4-to-collisionTab1}.
Note that $\bar x_1 \in \mathbb{R}^4$ from the table is the initial condition $\bar{p}$
in the statement of the theorem.
 We take $A$ to be a numerically computed approximate
inverse of the $16 \times 16$ matrix $DF_{i,j}^u(\bar{x})$.  Again, the
definition of $A$ does not require interval arithmetic. 

For the next step we compute interval enclosures of $F(\bar{x})$
and of $DF_{i,j}^u(x)$ for $x$ in a cube of radius 
$r_* = 5 \times 10^{-9}$ and obtain that 
\[
\|A F(\bar{x})\| \in  10^{-11} \times 
[   0.0,   0.82147145471154],
\]
and that 
\[
\sup_{x \in B_{r_*}(\bar{x})} \| \mbox{Id} - A D F_{i,j}^u(x) \| 
\in [   0.0,   0.00151459031904]. 
\]
Using interval arithmetic we compute 
\[
r = \frac{Y}{1-Z} \leq  
 8.3 \times 10^{-12},
\]
where the actual value stored in the computer is smaller than reported
here (and hence the inequality).  We then check, using interval arithmetic, 
that $Z r - r + Y < 0$.  Since $r < r_*$, we have that 
there exists a unique $x_* \in B_r(\bar{x})$ so that 
$F_{i,j}^u(x_*) = 0$.  
Moreover, transversality follows from the non-degeneracy of the derivative
of $F_{i,j}^u$.  

The proof is
implemented in IntLab running under MatLab, and took about 30 minutes
to run on a standard desktop computer.
\end{proof}

By replacing the operator $F_{i,j}^u$ with 
 the operator $F^s_{i,j}$ 
defined in Equation \eqref{eq:EC_to_L4_operator}, again with $j = 4$
and $i = 1$, we obtain a nonlinear map whose
zeros correspond to ejection-to-$L_4$ orbits. 
We compute an approximate numerical
zero of the resulting operator (the numerical data is given in Table 
\ref{tab:m1-to-L4_Tab3}) and repeat a nearly identical argument to
that above.   This results in the existence of a transverse
ejection-to-$L_4$ orbit in the PCRTBP with $\mu = 1/4$ and $c=3$. 
The validated error bound for the numerical data has 
\[
r \leq 1.8 \times 10^{-11},
\]
so that the desired orbit passes with in an $r$-neighborhood of the point 
\[
\bar{p} = 
\left(
\begin{array}{c}
 -0.112449038686947 \\
  -0.553321424594493 \\
  \phantom{-} 0.308527098616200 \\
   \phantom{-}0.727049637558896 \\
\end{array}
\right).
\]
In this way we prove the existence of both the orbits illustrated in Figure \ref%
{fig:EC_to_collision}.  More precisely, the orbit whose existence is established in 
Theorem \ref{thm:CAP-L4-to-collision} is illustrated in the right frame of the figure, 
and the orbit discussed in the preceding remarks is illustrated in the left frame.

\begin{table}[t!]
 {\scriptsize
\begin{tabular}{cllll}
\hline
$\overline{x}_0 = $ & $\phantom{(-}1.561515178070094$ &  &  &    \\ 
$\overline{x}_1 = $ 
& $( \phantom{-}0.0,$ & $  \phantom{-}0.018562030958889,$ & $0.0,$ & $
 1.999913860896684)$   \\ 
$\overline{x}_2 = $ 
& $( \phantom{-}0.191471460280817,$ & $\phantom{-}0.959639244531484,$ &
   $0.805673853857139,$ & $1.170011720749615)$   \\ 
$\overline{x}_3 = $ 
& $(-0.112449038686946,$ & $-0.553321424594493,$ &
   $0.308527098616200,$ & $0.727049637558895)$   \\ 
$\overline{x}_4 = $ 
& $\phantom{(-}5.229765599216696$ &  &  &    \\ 
$\overline{\tau} = $ & $\phantom{(-}4.673109099822270$ &  &  &    \\ 
$\overline{\alpha} = $ & $\phantom{(-}0.0$ &  &  &    \\  
\hline
\end{tabular}
}
\caption{ 
Numerical data for an approximate zero of the 
 map $F^s_{i,j}$ defined in Equation \eqref{eq:EC_to_L4_operator},
with $i = 1$, $j= 4$ and $s = 0.5$.
An argument similar to the proof of Theorem \ref{thm:CAP-L4-to-collision}, 
using the data in the table, leads to an existence proof for the 
ejection-to-$L_4$ orbit illustrated in the left frame of Figure \ref{fig:EC_to_collision}.
\label{tab:m1-to-L4_Tab3}}
\end{table}

\subsection{Transverse homoclinics for $L_4$ and $L_5$} \label{sec:homoclinicProofs}

Combining the methodology of Section \ref{sec:L4_to_collision}, and
especially Lemma \ref{lem:Li-collisions}, with Theorem \ref{thm:NK} we
obtain the following result.

\begin{maintheorem}
\label{thm:CAP-connections} Consider the planar PCRTBP with $\mu = 1/2$ and 
$c = 3$ is the energy level of $L_4$.  
Let 
\[
\bar{p} = 
\left(
\begin{array}{c}
   -0.037058535628028 \\
  -0.007623220519232 \\
   \phantom{-}0.873641524369283 \\
   \phantom{-}0.033084516464648
\end{array}
\right),
\]
and 
\[
B_r = \left\{ x \in \mathbb{R}^4 \, : \, \| x - \bar{p} \| \leq r \right\},
\]
where 
\[
r = 1.6 \times 10^{-9}.
\]
Then there exists a unique $p_* \in B_r$ so that the orbit of $p_*$ is homoclinic to $L_4$ and $W^s(L_4)$ intersects $W^u(L_4)$ transverseley along the 
orbit of  $p_*$, where transversality is relative to the level set $\setof*{E = 3}$. 

\end{maintheorem}
\begin{table}[t!]
\label{tab:L4-homoclinic} {\scriptsize 
\begin{tabular}{cllll}
\hline
$\overline{x}_0 = $ & $ \phantom{(-}1.411845524482813$ &  &  &    \\ 
$\overline{x}_1 = $ 
& $(-0.037058535628028,$ & $-0.007623220519232,$ & $\phantom{-}0.873641524369283,$ &  $\phantom{-}0.033084516464648)$   \\ 
$\overline{x}_2 = $ & $( -0.243792823114517,$ & $-1.231115802740768,$ & $\phantom{-}0.191555403283542,$ & $-0.508371511645513)$   \\ 
$\overline{x}_3 = $ & $(\phantom{-} 0.536705934592082,$ & $-1.502936895854406,$ & $\phantom{-}0.178454709494811,$ & $-0.106295188690239)$   \\ 
$\overline{x}_4 = $ & $( -0.504618223339967,$ & $-0.258236025635830,$ & $-0.463683951257916,$ & $-1.155517796520023)$   \\ 
$\overline{x}_5 = $ & $( -0.460363255327369,$ &$-0.431694933697799,$ & $\phantom{-}0.467966743350051,$ & $\phantom{-}0.748266448178995)$   \\ 
$\overline{x}_6 = $ & $\phantom{(-} 5.988827136344083$ &  &  &    \\ 
$\overline{\tau} = $ & $\phantom{(-}4.753189987600258$ &  &  &    \\ 
$\overline{\alpha} = $ & $\phantom{(-}0.0$ &  &  &    \\ 
\hline
\end{tabular}
}
\caption{ 
Numerical data for the proof of Theorem \protect\ref{thm:CAP-connections}, 
which provides an approximate zero of the 
$L_4$ homoclinic map $F_{i,j,k}$ defined in Equation \eqref{eq:homoclinicOperator}, 
when $i = k = 4$, $j = 2$, $s_1 = 1.8635$, and $s_2 = 5$.  The orbit is depicted on the right plot in Figure \ref{fig:PCRTBP_L4_homoclinics}.
}
\end{table}
\begin{proof}
As in the earlier cases, the argument hinges on proving the existence of a 
zero of a suitable nonlinear mapping, in this case the map 
$F_{i,j,k}$ defined in Equation \eqref{eq:homoclinicOperator}, 
with $i = k = 4$ and $j = 2$.  The integration time parameters 
are set as $s_1 = 1.8635$ and $s_2 = 5$.  These are the flow times
in the regularized coordinates and in the original coordinates (the second time)
respectively.  
With these choices, a zero of $F_{4,2,4}$ corresponds
to an orbit homoclinic to $L_4$ which passes through the 
Levi-Civita coordinates regularized at $m_2$.

The numerical data $\bar x \in \mathbb{R}^{24}$ providing an approximate zero of $F_{4,2,4}$ is 
reported in Table  \ref{tab:L4-homoclinic}.  
Note that $x_1$ corresponds to $\bar{p}$ in the hypothesis of the theorem.  
We let $A$ be a numerically computed approximate
inverse of the matrix $DF_{4,2,4}(\bar{x})$.  The table data and the 
matrix $A$ are computed using a numerical Newton scheme, and 
standard double precision floating point operations.  

Using validated numerical integration schemes, validated bounds on 
the local stable/unstable manifold parameterizations, and interval arithmetic, 
we compute interval enclosures of $ F_{4,2,4}(\bar{x})$ and of 
$D F_{4,2,4}(B_r(\bar{x}))$ with 
where $r =   1.659487745915747 \times 10^{-9}$.  We then check that 
\[
\| A F(\bar{x}) \| \in  10^{-8} \times 
[   0.0,   0.16432156145308],
\]
and that 
\[
\sup_{x \in B_r(\bar{x})} \| \mbox{Id} - A D F_{4,2,4}(B_r(\bar{x})) \| \in 
[   0.0,   0.00980551463848].
\]
Finally, we use interval arithmetic to verify that $Z r - r + Y < 0$ and transversality 
follows as in the earlier cases which completes the proof. 
\end{proof}
\bigskip

Note that, from a numerical perspective, this is the most difficult computer 
assisted argument presented so far.  This is seen in the fact that $Z \approx 10^{-2}$
and $r \approx 10^{-9}$.  That is, these constants are roughly three orders of magnitude 
less accurate than the previous theorems.  On the other hand, the orbit itself is
more complicated than those in the previous theorems.  We note that the accuracy 
of the result could be improved by taking smaller integration steps and/or
using higher order Taylor approximation. However, this would also increase the required computational time. 

Now, by symmetry, the result above gives a transverse homoclinic orbit for 
$L_5$ which passes near $m_1$.  We also observe that each of these transverse homoclinic orbits  also
satisfy the hypotheses of the theorems of Devaney and Henard discussed in 
Section \ref{sec:intro}.  In particular, Theorem \ref{thm:CAP-connections} also proves the existence of a chaotic 
subsystem in the $c = 3$ energy level of the PCRTBP near the orbit of $p_*$, and a tube of 
periodic orbits  parameterized by the Jacobi constant which accumulate to the homoclinic orbit through  $p_*$.

We remark that, using similar arguments, we are able to prove also the 
existence and transversality of of the homoclinic orbits in the left and center frames of 
Figure \ref{fig:PCRTBP_L4_homoclinics}.  
More precisely, let 
\[
\bar{p}_1 = \left(
\begin{array}{c}
 -0.033854025583296 \\
  -0.043110876471418 \\
   \phantom{-}0.844639632487862 \\
   \phantom{-}0.007320747846173
\end{array}
\right), \quad \quad \quad 
\bar{p}_2 = \left(
\begin{array}{c}
  \phantom{-}0.029871559148065 \\
  -0.006337684774610 \\
  \phantom{-} 0.850175365286339 \\
  -0.034734413580682
\end{array}
\right),
\]
and 
\[
r_1 = 2.03 \times 10^{-10}, \quad \quad \quad 
r_2 = 1.84 \times 10^{-8}.
\]
Then there exist unique points $p^1_* \in B(\bar{p}_1, r_1)$ and 
$p_*^2 \in B(\bar{p}_2, r_2)$ so that $W^{s,u}(L_4)$ intersect transversely 
along the orbits through these points. It is also interesting to note that $r_2$ is two orders of magnitude larger than $r_1$.  This is 
caused by the fact that the time of flight (integration time) 
is longer in this case and, more importantly, the fact that the second orbit passes very close to $m_1$.  Indeed,
the error bounds for the second orbit would very likely be improved by changing to 
regularized coordinates near $m_1$ and this may even be necessary to validate some homoclinics passing even closer to $m_1$ or $m_2$. Nevertheless, we were able to validate these orbits 
in standard coordinates so we have not done this here.  

The orbit of $p_*^1$ is illustrated in the left frame of Figure  \ref{fig:PCRTBP_L4_homoclinics}
appears to have $y$-axis symmetry, however we do not use this symmetry nor do we rigorously prove its existence. The 
orbit of $p_*^2$ is illustrated in the center frame of Figure  \ref{fig:PCRTBP_L4_homoclinics} has no apparent symmetry.  
The orbits illustrated in the left and center frames have appeared previously in the 
literature, as remarked in Section \ref{sec:intro}.  However, to the best of our knowledge this is the first mathematically rigorous proof of their existence.

\subsection{Periodic orbits passing through collision} \label{sec:PO_collisions}

We apply the methodology of section \ref{sec:symmetric-orbits}, namely Lemma %
\ref{lem:Lyap-existence} and Theorem \ref{th:Lyap-through-collision}, with
Theorem \ref{thm:NK} to obtain the following result. We consider the
Earth-Moon mass ratio largely for the sake of variety.

\begin{maintheorem}
\label{th:CAP-Lyap}Consider the Earth-Moon system \footnote{%
So named because this is the approximate mass ratio of the Moon relative to the Earth.} where $m_2$ has mass $\mu =0.0123/1.0123$ and $%
m_{1}$ has mass $1-\mu$. Let\footnote{%
In fact, our numerical calculations suggest that a more accurate value of
the Jacobi constant for which we have the collision is $1.434045949300768$.
However, since in the theorem we obtain only interval results, we round $%
c_{0}$ so that digits smaller than the width of the interval are not used.}%
\begin{equation*}
c_{0}=1.4340459493,\qquad \text{and\qquad }\delta =10^{-11}.
\end{equation*}%
There exists a single value $c^{\ast }\in \left( c_{0}-\delta ,c_{0}+\delta
\right) $ of the Jacobi integral, for which we have an orbit along the
intersection of the ejection and collision manifolds of $m_{1}$. Moreover,
for every $c\in \left[ c_{0}-\delta ,c_{0}+\delta \right] \setminus \left\{
c^{\ast }\right\} $ we have an $S$-symmetric Lyapunov orbit, that passes
close to the collision with $m_{1}$. In addition, for every $c\in \left\{ 1.2,1.25,1.3,\ldots ,1.65\right\} $
there exists a Lyapunov orbit, which passes close the collision with $m_{1}$%
. (These orbits are depicted in Figure \ref{fig:Lyap}.)
\end{maintheorem}

\begin{table}
{\scriptsize 
\begin{tabular}{r l l l l}
\hline 
$\bar x_0 = $ & \phantom{-(}0.0 &  &  &    \\
$\bar x_1 = $ & (\phantom{-}0.0, & \phantom{-}2.8111911379251, & \phantom{-}0.0, &\phantom{-}0.0)  \\ 
$\bar x_2 = $ & (\phantom{-}0.96886794638213, & -0.3219837525934, & -0.52587590839627, & -2.8644348266831)  \\
$\bar x_3 = $ & (\phantom{-}0.67431017475157, & -0.74811608844773, & -1.0190086228395, & -1.0721803622694)  \\
$\bar x_4 = $ & (-1.0199016713004, & \phantom{-}0.72482377063238, & -0.062207790440189, & \phantom{-}1.1639536137604) \\
$\bar x_5 = $ & (\phantom{-}0.1377088390491,& -0.32616835939217, & -0.22586709346235, & \phantom{-}0.6480010784062) \\
$\bar x_6 = $ & \phantom{(-}0.070375791076957 \\
$\bar \tau = $ & \phantom{(-}2.0972398526268 \\
$\bar \alpha = $ & \phantom{(-}0.0 \\
\hline 
\end{tabular}
}
\caption{Numerical data for the proof of Theorem \protect\ref{th:CAP-Lyap}, which gives an approximate solution to $F_c=0$ for the operator (\ref{eq:Fc-choice}), for which we have a collision of the family of Lyapunov orbits with $m_1$ for the Earth-Moon system (see Figure \ref{fig:Lyap}). This occurs for a unique value of the Jacobi constant $c^* \in \mathbf{c}$. \label{tabl:Lyap}}
\end{table}

\begin{proof}
The orbits for the Jacobi integral values in $\mathbf{c}:=\left[
c_{0}-\delta ,c_{0}+\delta \right] $ were established by means of Theorems %
\ref{th:Lyap-through-collision} and \ref{thm:aPosteriori}. We have first
pre-computed numerically (through a standard, non-interval, numerical
computation) an approximation $\mathbf{\bar{x}}\in \mathbb{R}^{22}$, $\bar{%
\tau}\in \mathbb{R}$ for the functions $\mathbf{x}\left( c\right) $ and $%
\tau \left( c\right) $, for $c\in \mathbf{c}$. (The $\mathbf{\bar{x}}$ and $%
\bar \tau $ are written out in Table \ref{tabl:Lyap}.) We then took $%
\bar x:=\left( \mathbf{\bar x},\bar \tau ,0\right) \in \mathbb{R}^{24},$ and
a ball $\overline{B}\left( \bar x,r\right) $, in the maximum norm, with $%
r=10^{-11}.$ We established using Theorem \ref%
{thm:aPosteriori} that $\mathbf{x}\left( c\right) $ and $\tau \left(
c\right) $ satisfying
\begin{equation*}
F_{c}\left( \mathbf{x}\left( c\right) ,\tau \left( c\right) ,0\right)
=0,\qquad \text{for }c\in \mathbf{c},
\end{equation*}%
are $r$-close to $\mathbf{\bar{x}}$ and $\bar{\tau}.$ To apply Theorem \ref%
{thm:aPosteriori} we have used the matrix $A$ to be an approximation of $%
\left( DF_{c}(\mathbf{\bar{x}},\bar{\tau},0)\right) ^{-1}$ (computed with
standard numerics, without interval arithmetic).

We  also checked using interval arithmetic that%
\begin{eqnarray*}
x_{0}\left( c_0-\delta \right) &\in &[3.2261\cdot 10^{-12},5.2262\cdot
10^{-12}]>0, \\
x_{0}\left( c_0+\delta \right) &\in &[-4.6229\cdot 10^{-12},-2.6228\cdot
10^{-12}]<0.
\end{eqnarray*}

By using Equation \eqref{eq:implicit-dx-dc}, we have established the following
interval arithmetic bound for the derivative of $x_0$ with respect to the
parameter 
\begin{equation*}
\frac{d}{dc}x_{0}\left( c\right) \in \left[ -0.53146,-0.25344\right]
<0\qquad \text{for }c\in \mathbf{c}.
\end{equation*}%
We also verified that%
\begin{equation*}
x_{6}\left( c\right) \in \left[ 0.07037579,0.07037580\right] ,\qquad \text{%
for }c\in \mathbf{c},
\end{equation*}%
so $x_{6}\left( c\right) \neq 0$. This proves all necessary hypotheses of Theorem \ref%
{th:Lyap-through-collision} are satisfied for the interval $\mathbf{c}$,
which finishes the first part of the proof.

The Lyapunov orbits for $c\in \left\{ 1.2,1.25,1.3,\ldots ,1.65\right\} $
were estabilshed in a similar way. For each value of the Jacobi constant we
have non-rigorously computed an approximation of a point for which $F_{c}$
is close to zero, and validated that we have $F_{c}=0$ for a point in a
given neighbourhood of each approximation by means of Theorem \ref%
{th:Lyap-through-collision}. Then each Lyapunov orbit followed from Lemma %
\ref{lem:Lyap-existence}. The proof was conducted by using the CAPD library \cite{CAPD_paper} and took
under 4 seconds on a standard laptop. 
\end{proof}

In a similar way we have used the operator in Equation \eqref{eq:Fc-equal} to prove the
following result. \begin{table}
{\scriptsize 
\begin{tabular}{r l l l l}
\hline 
$\bar x_0 = $ & \phantom{(-}2.1500812504263 &  &  &    \\
$\bar x_1 = $ & (\phantom{-}0.0, &\phantom{-}1.9284591731628, & \phantom{-}2.1500812504263, &\phantom{-}0.0) \\
$\bar y_1 = $ & (\phantom{-}0.69048473611567, &\phantom{-}1.7931365837031, &\phantom{-}2.0235432631366, &-0.68131264815823) \\
$\bar y_2 = $ & (\phantom{-}1.2840491252838, &\phantom{-}1.4060903194974, &\phantom{-}1.6633024005717, &-1.2578372410208) \\
$\bar y_3 = $ &(\phantom{-}1.6975511373876, & \phantom{-}0.82331762641153, &\phantom{-}1.1255430505039, &-1.635312833307) \\
$\bar y_4 = $ &(\phantom{-}1.8749336204161, &\phantom{-}0.13626785074409, & \phantom{-}0.4974554541058, &-1.7408028751654) \\
$\bar y_5 = $ & (\phantom{-}1.7998279644685, &-0.53073278614628,& -0.11297480280335, & -1.5366473737295) \\
$\bar y_6 = $ &(\phantom{-}1.5061749347656, & -1.0305902992759,& -0.59342931060715, & -1.0405479042095) \\
$\bar y_7 = $ &(\phantom{-}1.0818972907729,& -1.2225719420862, &-0.85102013466618, &-0.34180581034401) \\
$\bar y_8 = $ & (\phantom{-}0.65897461363208,& -1.0129455565064, &-0.83705911740279, &\phantom{-}0.41484122714387) \\
$\bar x_2 = $ & (\phantom{-}0.39363679634804, &-0.35214129843918,& -0.55459777216455, &\phantom{-}1.118144276789) \\
$\bar x_3 = $ & (\phantom{-}0.47871801188109, &-1.6325298121847, &-0.5792530867862, &\phantom{-}0.66259374214967) \\
$\bar x_4 = $ & (\phantom{-}{\bf 0.40239981358785},& -1.0164469492932, &\phantom{-}0.0, &\phantom{-}1.7224504635177) \\
$\bar x_5 = $ & (-0.25865224139372,& -0.43561054122851,& \phantom{-}0.51876042853484, &\phantom{-}1.7861707478994) \\
$\bar x_6 = $ &(\phantom{-}0.29778859976434,& -1.2111468567795, &-0.2683570951738, &-1.0237309759288) \\
$\bar x_7 = $ & \phantom{(}-0.38367247647373 & & & \\
$\bar \tau = $ & \phantom{(-}0.24444305938687 & & & \\
$\bar \alpha = $ & \phantom{(-}0.0 & & & \\
\hline 
\end{tabular}
}
\caption{Numerical data for the proof of Theorem \protect\ref{thm:doubleCollision} giving an approximate solution to $F_c=0$, for the operator (\ref{eq:Fc-equal}),  for $c=2.05991609689$ for which we have a double collision of a family of $R$-symmetric periodic orbits for the equal masses system; see Figure \ref{fig:eq}. In the bold font we have singled out the first coefficient of $x_4$, which is the time $s_4$ and not the physical coordinate of the collision point, for which we have $\hat x=0$. (See Equations \eqref{eq:s4} and \eqref{eq:Fc-equal}.)\label{tabl:eq}}
\end{table}

\begin{maintheorem} \label{thm:doubleCollision}
Consider the equal masses system where $\mu =\frac{1}{2}$. Let%
\footnote{%
We believe that a more accurate value of the Jacobi constant for which we
have the double collision is $2.059916096889689$.}%
\begin{equation*}
c_{0}=2.05991609689,\qquad \text{and\qquad }\delta =10^{-11}.
\end{equation*}%
There exists a single value $c^{\ast }\in \left( c_{0}-\delta ,c_{0}+\delta
\right) $ of the Jacobi integral, for which we have two intersections of the
ejection and collision manifolds of $m_{1}$ and $m_{2}$ (a double
collision). Moreover, for every $c\in \left[ c_{0}-\delta ,c_{0}+\delta %
\right] \setminus \left\{ c^{\ast }\right\} $ we have an $R$-symmetric
periodic orbit, that passes close to the collision with both $m_{1}$ and $%
m_{2}$.

In addition, for every $c\in \left\{ 2,2.05,2.1,2.15,2.2\right\} $ there
exists an $R$-symmetric periodic orbit, which passes close the collisions
with $m_{1}$ and $m_{2}$. (See Figure \ref{fig:eq}.)
\end{maintheorem}

\begin{proof}
The proof follows along the same lines as the proof of Theorem \ref%
{th:CAP-Lyap}. We do not write out the details of all the estimates since we
feel that this brings little added value\footnote{%
The code for the proof is made available on the personal web page of Maciej
Capi\'{n}ski.}. In the operator $F_{c}$ from Equation \eqref{eq:Fc-equal} we have
taken $s_{2}=3.3$ and $s_{5}=0.3$. The fact that $s_{2}$ involves a long
integration time caused a technical problem for us in obtaining an estimate
for $\frac{d}{dc}\pi _{\hat{y}}\mathbf{x}\left( c\right) $. To get a good
enough estimate to establish that $\frac{d}{dc}\pi _{\hat{y}}\mathbf{x}%
\left( c\right) >0$ we needed to include additional points $y_{1},\ldots
,y_{m}$ in the shooting scheme and extend $F_{c}$ to include 
\begin{equation*}
\phi _{\alpha }\left( x_{1},s\right) -y_{1},\quad \phi _{\alpha }\left(
y_{1},s\right) -y_{2},\quad \ldots \quad \phi _{\alpha }\left(
y_{m-1},s\right) -y_{m},\quad \phi _{\alpha }\left( y_{m},s\right) -x_{2},
\end{equation*}%
where $s=s_{2}/\left( m+1\right) .$ We took $m=8$, and the point $X_{0}$
wich serves as our approximation for $F_{c}=0$ is written out in Table \ref%
{tabl:eq}. The proof took under 10 seconds on a standard laptop.
\end{proof}

\begin{remark}[MatLab with IntLab versus CAPD] 
{\em
We note that the computer programs implemented in C\texttt{++} using the 
CAPD library run much faster than the programs implemented in MatLab using
IntLab to manage the interval arithmetic.  This is not surprising, as compiled 
programs typically run several hundred times faster than MatLab programs, 
and the use of interval arithmetic only complicates things. Moreover, 
CAPD is a well tested, optimized, general purpose package, while our 
IntLab codes were written specifically for this project. The CAPD library, due to its efficient integrators, allowed us to perform almost all of the proofs without subdividing the time steps, which was needed for the MatLab code (see Remark \ref{rem:additionalShooting} and comments at the end of section \ref{sec:EC}), except for the proof of Theorem \ref{thm:doubleCollision} (see Table \ref{tabl:eq}). In particular, little time 
has been spent on optimizing these codes. Nevertheless, it is nice to have 
rigorous integrators implemented in multiple languages, and the codes for 
validating the 2D stable/unstable manifolds at $L_4$ were written in IntLab
and have not been ported to $C^{++}$.
}
\end{remark}

\section{Acknowledgments}
The authors gratefully acknowledge conversations with Pau Martin, Immaculada Baldoma, and Marian Gidea
at the $60^{\rm th}$ birthday conference of Rafael de la Llave, \textit{Llavefest} in Barcelona in 
the summer of 2017.  We also offer our sincere thanks to an anonymous referee whose 
thorough review and thoughtful suggestions greatly improved the quality of the final 
manuscript.

\appendix 
{\small

\section{\label{sec:proof}}

\begin{proof}[Proof of Theorem \protect\ref{thm:aPosteriori}]
From Equation \eqref{eq:Krawczyk-ineq} and since $r>0$ we see that $Z+\frac{Y}{r}\leq
1,$ which since $Y,r>0$ gives 
\begin{equation}
Z<1.  \label{eq:Z-smaller-than-one}
\end{equation}

Now, define the Newton operator 
\begin{equation}
T(x)=x-AF(x).  \label{eq:T-def-Krawczyk}
\end{equation}%
For $x_{1},x_{2}\in \overline{B}(x_{0},r)$, by the mean value theorem and (%
\ref{eq:Krawczyk-Z}), we see that 
\begin{align*}
\Vert T(x_{1})-T(x_{2})\Vert & \leq \sup_{x\in \overline{B}(x_{0},r)}\Vert
DT(x)\Vert \Vert x_{1}-x_{2}\Vert  \\
& =\sup_{x\in \overline{B}(x_{0},r)}\Vert \mathrm{Id}-ADF(x)\Vert \left\Vert
x_{1}-x_{2}\right\Vert  \\
& \leq Z\Vert x_{1}-x_{2}\Vert ,
\end{align*}%
and since $Z<1$ we conclude that $T$ is a contraction on $\overline{B}%
(x_{0},r)$.

To see that $T$ maps $\overline{B}(x_{0},r)$ into itself, for $x\in 
\overline{B}(x_{0},r)$ by Equations \eqref{eq:Krawczyk-Y}--\eqref{eq:Krawczyk-ineq} we
have 
\begin{align*}
\Vert T(x)-x_{0}\Vert & \leq \Vert T(x)-T(x_{0})\Vert +\Vert
T(x_{0})-x_{0}\Vert \\
& \leq \sup_{x\in \overline{B}(x_{0},r)}\Vert DT(z)\Vert \Vert x-x_{0}\Vert
+\Vert AF(x_{0})\Vert \\
& \leq Zr+Y \\
& \leq r
\end{align*}%
hence $T(x)\in \overline{B}(x_{0},r)$.

By the Banach contraction mapping theorem there is a unique $\hat{x}\in 
\overline{B}(x_{0},r)$ so that 
\begin{equation}
T(\hat{x})=\hat{x}.  \label{eq:x-hat-zero}
\end{equation}

Now observe that for every $x\in \overline{B}(x_{0},r)$, including $\hat{x}$%
, by Equations \eqref{eq:Krawczyk-Z} and \eqref{eq:Z-smaller-than-one} we have that 
\begin{equation*}
\Vert \mathrm{Id}-ADF(\hat{x})\Vert \leq Z<1.
\end{equation*}%
Then 
\begin{equation*}
ADF(\hat{x})=\mathrm{Id}-\left( \mathrm{Id}-ADF(\hat{x})\right) =\mathrm{Id}%
-B
\end{equation*}%
with $\Vert B\Vert <1$. By the Neumann series theorem we see that $ADF(\hat{x%
})$ is invertible. It therefore follows that both $A$ and $DF(\hat{x})$ are
also invertible.

From Equations \eqref{eq:T-def-Krawczyk} and \eqref{eq:x-hat-zero} we see that $AF(%
\hat{x})=0.$ But $A$ is invertible, so it follows that $F(\hat{x})=0$, as
required.
\end{proof}


\section{\label{sec:manifolds}}
Here follows a terse description of the local stable/unstable 
manifold parameterizations used in the proofs in 
Sections \ref{sec:EC_to_L4_proofs} and \ref{sec:homoclinicProofs}.  
Much more complete information is found in \cite{MR3906230,MR2177465,mamotreto}.
In the present discussion 
$f \colon U \to \mathbb{R}^d$ denotes the (real analytic)
PCRTB vector field, and $L_{j}$ is one of the 
equilateral triangle libration points -- so that $j = 4,5$.
We are interested in parameter values where $Df(L_{4,5})$ has complex 
conjugate stable/unstable eigenvalues 
\[
\pm \alpha \pm i\beta,
\]
with $\alpha, \beta > 0$.  We write $\lambda = -\alpha + i \beta$
when considering the stable manifold, and $\lambda = \alpha + i \beta$
when considering the unstable.  

Our goal is to develop a formal series expansion of the 
form 
\begin{equation} \label{eq:powerSeries}
w_j^{\kappa}(z_1, z_1) = \sum_{m = 0}^\infty \sum_{n = 0}^\infty 
p_{mn} z_1^m  z_2^n, 
\end{equation}
where $j = 4$ or $5$ depending on wether we are based at 
$L_4$ or $L_5$, and $\kappa = s$ or $u$ depending on wether we
considering the stable or unstable manifold. 
Here $p_{mn} \in \mathbb{C}^4$ for all $(m,n) \in \mathbb{N}^2$.  
Moreover, we take 
\[
p_{00} = L_j,
\]
where $j = 4,5$, and 
\[
p_{10}  = \xi, \quad \quad \mbox{and} \quad \quad p_{01} = \overline{\xi}, 
\]
where $\xi, \overline{\xi} \in \mathbb{C}^4$ are complex conjugate eigenvectors 
associated with the complex conjugate eigenvalues $
\lambda, \bar\lambda \in \mathbb{C}$.

We use the parameterization method to characterize $w_j^\kappa$. 
While we  refer the interested reader to 
\cite{MR2177465,mamotreto} for much more complete discussion of 
this method, we remark that the main idea is to 
solve the invariance equation 
\begin{equation}\label{eq:invEq}
\lambda z_1 \frac{\partial}{\partial z_1} w_j^\kappa(z_1, z_2) + \overline{\lambda} z_2
w_j^\kappa (z_1,z_2) = f(w_j^\kappa(z_1,z_2)),
\end{equation}
subject to the constraints 
\[
w_j^\kappa (0,0) = L_j, \quad \quad 
\frac{\partial}{\partial v} w_j^\kappa (0,0) = \xi, 
\quad \quad \mbox{and} \quad \quad  
\frac{\partial}{\partial w} w_j^\kappa (0,0) = \overline{\xi}.
\]
It can be show that if $w_j^{\kappa}$ solves
Equation \eqref{eq:invEq} subject to these constraints, 
then it parameterizes a local stable/unstable manifold at $L_j$.

To solve Equation \eqref{eq:invEq} numerically we 
insert the power series ansatz of Equation \eqref{eq:powerSeries}, 
expand the nonlinearities, and match like powers of $z_1$ and $z_2$.
This procedure leads to homological equations of the form 
\[
\left(Df(p_{00}) - (m \lambda + n \overline{\lambda})\mbox{Id}\right) p_{mn} = \mathcal{R}_{mn}, 
\]
describing the power series coefficients $p_{mn}$ for $m + n \geq 2$.  Here
$\mathcal{R}_{mn}$ is a nonlinear function of the coefficients of order 
less than $m + n$, whose computation in the case of the PCRTBP
 is discussed in more detail in \cite{MR3906230}. 
Note that if $f$ is real analytic, then the coefficients have the 
symmetry
\[
p_{nm} = \overline{p_{mn}},
\]
and we obtain the real image of $\mathcal{P}$ by evaluating on complex conjugate 
variables $w = \overline{v}$.

Since the order zero and order 1 coefficients are determined by $L_j$ and its eigendata,
we can compute $p_{mn}$ for all $2 \leq m + n \leq N$ by recursively solving the 
linear homological equations to any desired order $N \geq 2$.  
We obtain the approximation 
\[
w_j^{\kappa, N}(z_1,z_2) = \sum_{m+n = 0}^N p_{mn} z_1^m z_2^n.
\]

{\tiny
\begin{table}[tbp]
{\scriptsize
\begin{tabular}{clll}
\hline
 $m$ $\setminus$ $n$ & \qquad 0 &\qquad 1 &\qquad 2 
 \\
 \hline
 $0$ & 
 $\phantom{10^{-4}}
\left(
\begin{array}{l}
\phantom{-}0.0 \\
\phantom{-}0.0 \\
\phantom{-}0.866 \\
\phantom{-}0.0 
\end{array}
\right)
$   &
$\phantom{10^{-4}}
\left(
\begin{array}{l}
\phantom{-}0.012 + 0.018i \\
-0.025  \\
-0.015 + 0.0067i \\
\phantom{-}0.0034 - 0.019i 
  \end{array}
  \right)$
 &  
 $
  10^{-3} 
  \left(
  \begin{array}{l}
 -0.050 + 0.076i \\
 -0.081 - 0.190i \\
 -0.054 - 0.041i\\
  \phantom{-}0.140 - 0.051i
  \end{array}
  \right)
 $
 \\
 $1$ &
 $\phantom{10^{-4}}\left(
\begin{array}{l}
\phantom{-}0.012 - 0.018i \\
 -0.025  \\
 -0.015 - 0.0067i \\
  \phantom{-}0.0034 + 0.019i 
  \end{array}
  \right)$
 &  
 $
 10^{-3}
 \left(
 \begin{array}{l}
  \phantom{-}0.37 \\
 -0.47 \\
 -0.09  \\
 \phantom{-}0.12 
  \end{array}
  \right)
  $
 &  
 $
 10^{-4}
 \left(
 \begin{array}{l}
  \phantom{-}0.041 + 0.055 i \\
 -0.130 - 0.070i \\
 -0.036 - 0.048i \\
  \phantom{-}0.110 + 0.060i
  \end{array}
  \right) $
 \\
 $2$ & 
  $
  10^{-3} 
  \left(
  \begin{array}{l}
 -0.050 - 0.076i \\
 -0.081 + 0.190i \\
 -0.054 + 0.041i\\
  \phantom{-}0.140 + 0.051i
  \end{array}
  \right)
 $
  & 
   $
 10^{-4}
 \left(
 \begin{array}{l}
  \phantom{-}0.041 - 0.055 i \\
 -0.130 + 0.070 i \\
 -0.036 + 0.048i \\
  \phantom{-}0.110 - 0.060i
  \end{array}
  \right) $
   & 0 
 \\
\hline
\end{tabular}
\[
p_{03}=
  10^{-5} \left(
  \begin{array}{l}
 -0.14 + 0.18i \\
 -0.26 - 0.76i \\
 \phantom{-}0.11 + 0.09i \\
 -0.47 + 0.11i
 \end{array}
 \right)
\qquad\qquad\qquad
 p_{30}=10^{-5} \left(
  \begin{array}{l}
 -0.14 - 0.18i \\
 -0.26 + 0.76i \\
  \phantom{-}0.11 - 0.09i \\
 -0.47 - 0.11i
 \end{array}
 \right)
 \]
}
\caption{ Approximate power series coefficients $p_{mn}$ for the 
parameterization of the local stable manifold of $L_4$ for the equal
 masses case $\mu=1/2$.
\label{tab:parm_coeff1} }
\end{table}
}
For example, in the 
PCRTBP with $\mu = 1/2$,  Table \ref{tab:parm_coeff1} shows approximate coefficients 
for the stable manifold at $L_4$, computed to order $N = 3$.  The data has been 
truncated at only two or three significant figures to make it fit in the table.  
Note that the complex conjugate structure of the coefficients is seen in the table. 
The table is included to give the reader a sense of the form of the data in these calculations,
and could be used to very approximately reproduce some of the results in the present 
work.

For the calculations in the main body of the text, we take $N= 12$ and 
 compute the $p_{mn}$ by recursively solving the homological equations 
using interval arithmetic.   Moreover, using the a-posteriori analysis
developed in \cite{MR3792792}, we obtain a bound of the form 
\[
\sum_{m+n = 13}^\infty \| p_{mn} \| \leq 1.4 \times 10^{-13},
\]
on the norm of the tail of the parameterization.  
The analysis is very similar to the a-posteriori analysis of the 
Newton Krawczyk Theorem \ref{thm:NK} promoted in 
the present work, adapted to the context of Banach spaces
 of infinite sequences. 

Note that this ``little ell one'' norm
bounds the $C^0$ norm of the truncation error on the unit disk,
and that Cauchy bounds can be used to estimate derivatives of the parameterization 
on any smaller disk.  Thus we actually take 
\[
P_j^\kappa(\theta) = w_j^{\kappa}(0.9 \cos(\theta) + 0.9 \sin(\theta) i, 
0.9 \cos(\theta) - 0.9 \sin(\theta) i),
\] 
as our local parameterization, where
\[
w_j^{\kappa}(z_1,z_2) = w_j^{\kappa, N}(z_1,z_2) + w_j^{\kappa, \infty}(z_1,z_2),
\] 
is a polynomial plus a tail which has 
\[
 w_j^{\kappa, \infty}(z_1,z_2) = \sum_{n + m = N+1}^\infty p_{mn} z_1^m z_2^n, 
\]
and 
\[
\sup_{|z_1|,|z_2| < 1} \left\| w_j^{\kappa, \infty}(z_1,z_2) \right\| \leq 1.4 \times 10^{-13}.
\]
The $0.9$ gives up a portion of the disk, allowing us to bound the derivatives 
needed in the Newton-Kantorovich argument.

}

\bibliographystyle{plain}
\bibliography{papers}

\begin{thebibliography}{100}
\expandafter\ifx\csname url\endcsname\relax
  \def\url#1{\texttt{#1}}\fi
\expandafter\ifx\csname urlprefix\endcsname\relax\def\urlprefix{URL }\fi
\expandafter\ifx\csname href\endcsname\relax
  \def\href#1#2{#2} \def\path#1{#1}\fi

\bibitem{MR1194622}
H.~Poincar\'{e}, New methods of celestial mechanics. {V}ol. 1, Vol.~13 of
  History of Modern Physics and Astronomy, American Institute of Physics, New
  York, 1993.

\bibitem{MR1194623}
H.~Poincar\'{e}, New methods of celestial mechanics. {V}ol. 2, Vol.~13 of
  History of Modern Physics and Astronomy, American Institute of Physics, New
  York, 1993.

\bibitem{MR1194624}
H.~Poincar\'{e}, New methods of celestial mechanics. {V}ol. 3, Vol.~13 of
  History of Modern Physics and Astronomy, American Institute of Physics, New
  York, 1993.

\bibitem{earthMoonRock}
U.~S.~R. Association.
\newblock
  \href{https://solarsystem.nasa.gov/news/820/earths-oldest-rock-found-on-the-moon/}{Earth's
  oldest rock found onthe moon} [online] (January 2019).

\bibitem{next100Years}
G.~Friedman, {T}he {N}ext {H}undred {Y}ears, no. ISBN 9780385517058, Doubleday,
  2009.

\bibitem{theMoonIsHarsh}
R.~A. Heinlein, The {M}oon {I}s {A} {H}arsh {M}istress, no. ISBN 0312863551,
  G.P. Putnam's Sons, 1966.

\bibitem{MR682839}
J.~Llibre, \href{https://doi.org/10.1007/BF01230662}{On the restricted
  three-body problem when the mass parameter is small}, Celestial Mech.
  28~(1-2) (1982) 83--105.
\newblock \href {https://doi.org/10.1007/BF01230662}
  {\path{doi:10.1007/BF01230662}}.
\newline\urlprefix\url{https://doi.org/10.1007/BF01230662}

\bibitem{MR0365628}
J.~Henrard, \href{https://doi-org.ezproxy.fau.edu/10.1007/BF01227510}{Proof of
  a conjecture of {E}. {S}tr\"omgren}, Celestial Mech. 7 (1973) 449--457.
\newline\urlprefix\url{https://doi-org.ezproxy.fau.edu/10.1007/BF01227510}

\bibitem{MR0442990}
R.~L. Devaney, Homoclinic orbits in {H}amiltonian systems, J. Differential
  Equations 21~(2) (1976) 431--438.

\bibitem{stromgrenMoulton}
E.~Str\"{o}mgren, Forms of periodic motion in the restricted problem and in the
  general problem of three bodies, according to the researches exscuted at the
  observatory of {C}openhagen, Publikationer og mindre Meddelelser fra
  Kobenhavns Observatorium~(39) (1922).

\bibitem{stromgrenRef}
E.~Str\"{o}mgren, Connaissance actuelle des orbites dans le probleme des trois
  corps, Bull. Astron. 9 (1934) 87--130.

\bibitem{szebehelyTriangularPoints}
V.~Szebehely, F.~T.V., A family of retegrade orbits around the triangular
  equilibrium points, The Astronomical Journal 72~(3) (1967) 373--379.

\bibitem{onMoulton_Szebehely}
V.~Szebehely, On {M}oulton's orbits in the restricted problem of three bodies,
  Proceedings of the {N}ational {A}cademy of (S)ciences of the {U}nided
  {S}tates of {A}merica 56~(6) (1966) pp. 1641--1645.

\bibitem{theoryOfOrbits}
V.~Szebehely, Theory of Orbits, Academic Press Inc., 1967.

\bibitem{MR1879221}
A.~J. Maciejewski, S.~M. Rybicki,
  \href{https://doi.org/10.1023/A:1013276830424}{Global bifurcations of
  periodic solutions of the {H}ill lunar problem}, Celestial Mech. Dynam.
  Astronom. 81~(4) (2001) 279--297.
\newblock \href {https://doi.org/10.1023/A:1013276830424}
  {\path{doi:10.1023/A:1013276830424}}.
\newline\urlprefix\url{https://doi.org/10.1023/A:1013276830424}

\bibitem{MR2042173}
A.~J. Maciejewski, S.~M. Rybicki,
  \href{https://doi.org/10.1023/B:CELE.0000017193.10060.ac}{Global bifurcations
  of periodic solutions of the restricted three body problem}, Celestial Mech.
  Dynam. Astronom. 88~(3) (2004) 293--324.
\newblock \href {https://doi.org/10.1023/B:CELE.0000017193.10060.ac}
  {\path{doi:10.1023/B:CELE.0000017193.10060.ac}}.
\newline\urlprefix\url{https://doi.org/10.1023/B:CELE.0000017193.10060.ac}

\bibitem{MR2969866}
E.~P\'{e}rez-Chavela, S.~a. Rybicki,
  \href{https://doi.org/10.1016/j.nonrwa.2012.07.027}{Topological bifurcations
  of central configurations in the {$N$}-body problem}, Nonlinear Anal. Real
  World Appl. 14~(1) (2013) 690--698.
\newblock \href {https://doi.org/10.1016/j.nonrwa.2012.07.027}
  {\path{doi:10.1016/j.nonrwa.2012.07.027}}.
\newline\urlprefix\url{https://doi.org/10.1016/j.nonrwa.2012.07.027}

\bibitem{MR3007103}
C.~Garc\'{\i}a-Azpeitia, J.~Ize,
  \href{https://doi.org/10.1016/j.jde.2012.08.022}{Global bifurcation of planar
  and spatial periodic solutions from the polygonal relative equilibria for the
  {$n$}-body problem}, J. Differential Equations 254~(5) (2013) 2033--2075.
\newblock \href {https://doi.org/10.1016/j.jde.2012.08.022}
  {\path{doi:10.1016/j.jde.2012.08.022}}.
\newline\urlprefix\url{https://doi.org/10.1016/j.jde.2012.08.022}

\bibitem{MR2821620}
C.~Garc\'{\i}a-Azpeitia, J.~Ize,
  \href{https://doi.org/10.1007/s10569-011-9354-2}{Global bifurcation of planar
  and spatial periodic solutions in the restricted {$n$}-body problem},
  Celestial Mech. Dynam. Astronom. 110~(3) (2011) 217--237.
\newblock \href {https://doi.org/10.1007/s10569-011-9354-2}
  {\path{doi:10.1007/s10569-011-9354-2}}.
\newline\urlprefix\url{https://doi.org/10.1007/s10569-011-9354-2}

\bibitem{MR1509241}
J.~Chazy, \href{http://www.numdam.org/item?id=ASENS_1922_3_39__29_0}{Sur
  l'allure du mouvement dans le probl\`eme des trois corps quand le temps
  cro\^{\i}t ind\'{e}finiment}, Ann. Sci. \'{E}cole Norm. Sup. (3) 39 (1922)
  29--130.
\newline\urlprefix\url{http://www.numdam.org/item?id=ASENS_1922_3_39__29_0}

\bibitem{moultonBook}
F.~Moulton, D.~Buchanan, T.~Buck, F.~Griffin, W.~Longley, W.~MacMillan,
  Periodic Orbits, no. Publication No. 161, Carnegie Institution of Washington,
  1920.

\bibitem{MR295648}
D.~G. Saari, \href{https://doi.org/10.2307/1995752}{Improbability of collisions
  in {N}ewtonian gravitational systems}, Trans. Amer. Math. Soc. 162 (1971)
  267--271; erratum, ibid. 168 (1972), 521.
\newblock \href {https://doi.org/10.2307/1995752} {\path{doi:10.2307/1995752}}.
\newline\urlprefix\url{https://doi.org/10.2307/1995752}

\bibitem{MR321386}
D.~G. Saari, \href{https://doi.org/10.2307/1996638}{Improbability of collisions
  in {N}ewtonian gravitational systems. {II}}, Trans. Amer. Math. Soc. 181
  (1973) 351--368.
\newblock \href {https://doi.org/10.2307/1996638} {\path{doi:10.2307/1996638}}.
\newline\urlprefix\url{https://doi.org/10.2307/1996638}

\bibitem{MR3951693}
M.~Guardia, V.~Kaloshin, J.~Zhang,
  \href{https://doi.org/10.1007/s00205-019-01368-7}{Asymptotic density of
  collision orbits in the restricted circular planar 3 body problem}, Arch.
  Ration. Mech. Anal. 233~(2) (2019) 799--836.
\newblock \href {https://doi.org/10.1007/s00205-019-01368-7}
  {\path{doi:10.1007/s00205-019-01368-7}}.
\newline\urlprefix\url{https://doi.org/10.1007/s00205-019-01368-7}

\bibitem{MR1555161}
T.~Levi-Civita, \href{https://doi.org/10.1007/BF02404404}{Sur la
  r\'{e}gularisation du probl\`eme des trois corps}, Acta Math. 42~(1) (1920)
  99--144.
\newblock \href {https://doi.org/10.1007/BF02404404}
  {\path{doi:10.1007/BF02404404}}.
\newline\urlprefix\url{https://doi.org/10.1007/BF02404404}

\bibitem{cellettiCollisions}
A.~Celletti, Basics of regularization theory, in: B.~Stevens, A.~Maciejewski,
  H.~M. (Eds.), Chaotic Worlds: From Order to Disorder in Gravational N-Body
  Dynamical Systems, Springer, Dordrecht, 2006.

\bibitem{MR633766}
R.~L. Devaney, Singularities in classical mechanical systems, in: Ergodic
  theory and dynamical systems, {I} ({C}ollege {P}ark, {M}d., 1979--80),
  Vol.~10 of Progr. Math., Birkh\"{a}user, Boston, Mass., 1981, pp. 211--333.

\bibitem{MR562695}
R.~McGehee, Singularities in classical celestial mechanics, in: Proceedings of
  the {I}nternational {C}ongress of {M}athematicians ({H}elsinki, 1978), Acad.
  Sci. Fennica, Helsinki, 1980, pp. 827--834.

\bibitem{MR359459}
R.~McGehee, \href{https://doi.org/10.1007/BF01390175}{Triple collision in the
  collinear three-body problem}, Invent. Math. 27 (1974) 191--227.
\newblock \href {https://doi.org/10.1007/BF01390175}
  {\path{doi:10.1007/BF01390175}}.
\newline\urlprefix\url{https://doi.org/10.1007/BF01390175}

\bibitem{MR3069058}
R.~Moeckel, R.~Montgomery,
  \href{https://doi.org/10.2140/pjm.2013.262.129}{Symmetric regularization,
  reduction and blow-up of the planar three-body problem}, Pacific J. Math.
  262~(1) (2013) 129--189.
\newblock \href {https://doi.org/10.2140/pjm.2013.262.129}
  {\path{doi:10.2140/pjm.2013.262.129}}.
\newline\urlprefix\url{https://doi.org/10.2140/pjm.2013.262.129}

\bibitem{MR638060}
E.~A. Belbruno, \href{https://doi.org/10.1007/BF01234179}{A new regularization
  of the restricted three-body problem and an application}, Celestial Mech.
  25~(4) (1981) 397--415.
\newblock \href {https://doi.org/10.1007/BF01234179}
  {\path{doi:10.1007/BF01234179}}.
\newline\urlprefix\url{https://doi.org/10.1007/BF01234179}

\bibitem{Devaney1980TripleCI}
R.~L. Devaney, Triple collision in the planar isosceles three body problem,
  Inventiones mathematicae 60 (1980) 249--267.

\bibitem{MR640127}
C.~Sim\'{o}, Analysis of triple collision in the isosceles problem, in:
  Classical mechanics and dynamical systems ({M}edford, {M}ass., 1979), Vol.~70
  of Lecture Notes in Pure and Appl. Math., Dekker, New York, 1981, pp.
  203--224.

\bibitem{ElBialy:1989td}
M.~S. ElBialy, \href{https://doi.org/10.1007/BF00945869}{Triple collisions in
  the isosceles three body problem with small mass ratio}, Zeitschrift f{\"u}r
  angewandte Mathematik und Physik ZAMP 40~(5) (1989) 645--664.
\newblock \href {https://doi.org/10.1007/BF00945869}
  {\path{doi:10.1007/BF00945869}}.
\newline\urlprefix\url{https://doi.org/10.1007/BF00945869}

\bibitem{MR571374}
E.~A. Lacomba, L.~Losco,
  \href{https://doi-org.libproxy.york.ac.uk/10.1090/S0273-0979-1980-14802-8}{Triple
  collisions in the isosceles {$3$}-body problem}, Bull. Amer. Math. Soc.
  (N.S.) 3~(1, part 1) (1980) 710--714.
\newblock \href {https://doi.org/10.1090/S0273-0979-1980-14802-8}
  {\path{doi:10.1090/S0273-0979-1980-14802-8}}.
\newline\urlprefix\url{https://doi-org.libproxy.york.ac.uk/10.1090/S0273-0979-1980-14802-8}

\bibitem{10.2307/24893242}
R.~MOECKEL, \href{http://www.jstor.org/stable/24893242}{Orbits near triple
  collision in the three-body problem}, Indiana University Mathematics Journal
  32~(2) (1983) 221--240.
\newline\urlprefix\url{http://www.jstor.org/stable/24893242}

\bibitem{MR3880194}
M.~Alvarez-Ram\'{\i}rez, E.~Barrab\'{e}s, M.~Medina, M.~Oll\'{e},
  \href{https://doi.org/10.1016/j.cnsns.2018.10.026}{Ejection-collision orbits
  in the symmetric collinear four-body problem}, Commun. Nonlinear Sci. Numer.
  Simul. 71 (2019) 82--100.
\newblock \href {https://doi.org/10.1016/j.cnsns.2018.10.026}
  {\path{doi:10.1016/j.cnsns.2018.10.026}}.
\newline\urlprefix\url{https://doi.org/10.1016/j.cnsns.2018.10.026}

\bibitem{MR949626}
E.~A. Lacomba, J.~Llibre,
  \href{https://doi.org/10.1016/0022-0396(88)90019-8}{Transversal
  ejection-collision orbits for the restricted problem and the {H}ill's problem
  with applications}, J. Differential Equations 74~(1) (1988) 69--85.
\newblock \href {https://doi.org/10.1016/0022-0396(88)90019-8}
  {\path{doi:10.1016/0022-0396(88)90019-8}}.
\newline\urlprefix\url{https://doi.org/10.1016/0022-0396(88)90019-8}

\bibitem{MR993819}
J.~Delgado~Fern\'{a}ndez, \href{https://doi.org/10.1007/BF01235542}{Transversal
  ejection-collision orbits in {H}ill's problem for {$C\gg 1$}}, Celestial
  Mech. 44~(3) (1988/89) 299--307.
\newblock \href {https://doi.org/10.1007/BF01235542}
  {\path{doi:10.1007/BF01235542}}.
\newline\urlprefix\url{https://doi.org/10.1007/BF01235542}

\bibitem{MR1342132}
C.~Pinyol,
  \href{https://doi-org.ezproxy.fau.edu/10.1007/BF00049513}{Ejection-collision
  orbits with the more massive primary in the planar elliptic restricted three
  body problem}, Celestial Mech. Dynam. Astronom. 61~(4) (1995) 315--331.
\newline\urlprefix\url{https://doi-org.ezproxy.fau.edu/10.1007/BF00049513}

\bibitem{MR967629}
A.~Chenciner, J.~Llibre, \href{https://doi.org/10.1017/S0143385700009330}{A
  note on the existence of invariant punctured tori in the planar circular
  restricted three-body problem}, Ergodic Theory Dynam. Systems 8$^*$~(Charles
  Conley Memorial Issue) (1988) 63--72.
\newblock \href {https://doi.org/10.1017/S0143385700009330}
  {\path{doi:10.1017/S0143385700009330}}.
\newline\urlprefix\url{https://doi.org/10.1017/S0143385700009330}

\bibitem{MR1849229}
J.~F\'{e}joz, \href{https://doi.org/10.1006/jdeq.2000.3972}{Averaging the
  planar three-body problem in the neighborhood of double inner collisions}, J.
  Differential Equations 175~(1) (2001) 175--187.
\newblock \href {https://doi.org/10.1006/jdeq.2000.3972}
  {\path{doi:10.1006/jdeq.2000.3972}}.
\newline\urlprefix\url{https://doi.org/10.1006/jdeq.2000.3972}

\bibitem{MR1919782}
J.~F\'{e}joz, \href{https://doi.org/10.1006/jdeq.2001.4117}{Quasiperiodic
  motions in the planar three-body problem}, J. Differential Equations 183~(2)
  (2002) 303--341.
\newblock \href {https://doi.org/10.1006/jdeq.2001.4117}
  {\path{doi:10.1006/jdeq.2001.4117}}.
\newline\urlprefix\url{https://doi.org/10.1006/jdeq.2001.4117}

\bibitem{MR3417880}
L.~Zhao, \href{https://doi.org/10.1002/cpa.21539}{Quasi-periodic
  almost-collision orbits in the spatial three-body problem}, Comm. Pure Appl.
  Math. 68~(12) (2015) 2144--2176.
\newblock \href {https://doi.org/10.1002/cpa.21539}
  {\path{doi:10.1002/cpa.21539}}.
\newline\urlprefix\url{https://doi.org/10.1002/cpa.21539}

\bibitem{MR1805879}
S.~V. Bolotin, R.~S. Mackay,
  \href{https://doi.org/10.1023/A:1008393706818}{Periodic and chaotic
  trajectories of the second species for the {$n$}-centre problem}, Celestial
  Mech. Dynam. Astronom. 77~(1) (2000) 49--75 (2001).
\newblock \href {https://doi.org/10.1023/A:1008393706818}
  {\path{doi:10.1023/A:1008393706818}}.
\newline\urlprefix\url{https://doi.org/10.1023/A:1008393706818}

\bibitem{MR2245344}
S.~Bolotin, R.~S. MacKay,
  \href{https://doi.org/10.1007/s10569-006-9006-0}{Nonplanar second species
  periodic and chaotic trajectories for the circular restricted three-body
  problem}, Celestial Mech. Dynam. Astronom. 94~(4) (2006) 433--449.
\newblock \href {https://doi.org/10.1007/s10569-006-9006-0}
  {\path{doi:10.1007/s10569-006-9006-0}}.
\newline\urlprefix\url{https://doi.org/10.1007/s10569-006-9006-0}

\bibitem{MR2331205}
S.~Bolotin, \href{https://doi.org/10.1142/9789812702142_0007}{Shadowing chains
  of collision orbits for the elliptic 3-body problem}, in: S{PT}
  2004---{S}ymmetry and perturbation theory, World Sci. Publ., Hackensack, NJ,
  2005, pp. 51--58.
\newblock \href {https://doi.org/10.1142/9789812702142\_0007}
  {\path{doi:10.1142/9789812702142\_0007}}.
\newline\urlprefix\url{https://doi.org/10.1142/9789812702142_0007}

\bibitem{MR1877971}
J.~Font, A.~Nunes, C.~Sim\'{o},
  \href{https://doi.org/10.1088/0951-7715/15/1/306}{Consecutive
  quasi-collisions in the planar circular {RTBP}}, Nonlinearity 15~(1) (2002)
  115--142.
\newblock \href {https://doi.org/10.1088/0951-7715/15/1/306}
  {\path{doi:10.1088/0951-7715/15/1/306}}.
\newline\urlprefix\url{https://doi.org/10.1088/0951-7715/15/1/306}

\bibitem{MR2475705}
J.~Font, A.~Nunes, C.~Sim\'{o},
  \href{https://doi.org/10.1007/s10569-008-9176-z}{A numerical study of the
  orbits of second species of the planar circular {RTBP}}, Celestial Mech.
  Dynam. Astronom. 103~(2) (2009) 143--162.
\newblock \href {https://doi.org/10.1007/s10569-008-9176-z}
  {\path{doi:10.1007/s10569-008-9176-z}}.
\newline\urlprefix\url{https://doi.org/10.1007/s10569-008-9176-z}

\bibitem{MR3693390}
M.~Oll\'{e}, O.~Rodr\'{\i}guez, J.~Soler,
  \href{https://doi.org/10.1016/j.cnsns.2017.07.013}{Ejection-collision orbits
  in the {RTBP}}, Commun. Nonlinear Sci. Numer. Simul. 55 (2018) 298--315.
\newblock \href {https://doi.org/10.1016/j.cnsns.2017.07.013}
  {\path{doi:10.1016/j.cnsns.2017.07.013}}.
\newline\urlprefix\url{https://doi.org/10.1016/j.cnsns.2017.07.013}

\bibitem{MR4110029}
M.~Oll\'{e}, O.~Rodr\'{\i}guez, J.~Soler,
  \href{https://doi.org/10.1016/j.cnsns.2020.105294}{Analytical and numerical
  results on families of {$n$}-ejection-collision orbits in the {RTBP}},
  Commun. Nonlinear Sci. Numer. Simul. 90 (2020) 105294, 27.
\newblock \href {https://doi.org/10.1016/j.cnsns.2020.105294}
  {\path{doi:10.1016/j.cnsns.2020.105294}}.
\newline\urlprefix\url{https://doi.org/10.1016/j.cnsns.2020.105294}

\bibitem{MR4162341}
M.~Oll\'{e}, O.~Rodr\'{\i}guez, J.~Soler,
  \href{https://doi.org/10.1016/j.cnsns.2020.105550}{Transit regions and
  ejection/collision orbits in the {RTBP}}, Commun. Nonlinear Sci. Numer.
  Simul. 94 (2021) 105550, 29.
\newblock \href {https://doi.org/10.1016/j.cnsns.2020.105550}
  {\path{doi:10.1016/j.cnsns.2020.105550}}.
\newline\urlprefix\url{https://doi.org/10.1016/j.cnsns.2020.105550}

\bibitem{tereOlleCollisions}
T.~M. Seara, M.~Oll\'{e}, O.~Rodr\'{\i}guez, J.~Soler, Generalised analytical
  results on $n$-ejection-collision orbits in the rtbp: analysis of
  bifurcations, (Submitted)) (2022).

\bibitem{MR1554890}
G.~H. Darwin,
  \href{https://doi-org.ezproxy.fau.edu/10.1007/BF02417978}{Periodic {O}rbits},
  Acta Math. 21~(1) (1897) 99--242.
\newline\urlprefix\url{https://doi-org.ezproxy.fau.edu/10.1007/BF02417978}

\bibitem{hidenHumanComputers}
S.~B. Edwards, Hidden {H}uman {C}omputers: the {B}lack {W}omen of {N}{A}{S}{A},
  no. ISBN 978-1680783872, Essential Library, 2017.

\bibitem{hiddenFigures}
M.~L. Sheltterly, Hidden Figures, no. ISBN 978-0-06-236360-2, William Morrow
  and Company, 2016.

\bibitem{dorthyVaughn}
A.~Selby, Who is the great Dorthy Vaughn, no. ISBN 1978491069, Ladies Image
  Publishing, 2017.

\bibitem{nasaComputers}
O.~Lutz, \href{https://www.nasa.gov/feature/jpl/when-computers-were-human}{When
  computers were human}, NASA JPL, 2016.
\newline\urlprefix\url{https://www.nasa.gov/feature/jpl/when-computers-were-human}

\bibitem{mcmastersPage}
J.~Walsh, {H}uman {C}omputers at {N}{A}{S}{A} project, Digital Commons $@$
  Macalester, 2016.

\bibitem{alesanderaBook}
A.~Celletti, E.~Perozzi, Celestial {M}echanics: {T}he {W}altz of the {P}lanets,
  Springer (in association with Praxis Publishing), 2007.

\bibitem{MR2391999}
E.~Belbruno, Fly me to the moon, Princeton University Press, Princeton, NJ,
  2007, an insider's guide to the new science of space travel, With a foreword
  by Neil deGrasse Tyson.

\bibitem{MR618636}
W.-J. Beyn, E.~Doedel, \href{http://dx.doi.org/10.1137/0902009}{Stability and
  multiplicity of solutions to discretizations of nonlinear ordinary
  differential equations}, SIAM J. Sci. Statist. Comput. 2~(1) (1981) 107--120.
\newblock \href {https://doi.org/10.1137/0902009} {\path{doi:10.1137/0902009}}.
\newline\urlprefix\url{http://dx.doi.org/10.1137/0902009}

\bibitem{MR1068199}
W.-J. Beyn,
  \href{http://dx.doi.org.proxy.libraries.rutgers.edu/10.1093/imanum/10.3.379}{The
  numerical computation of connecting orbits in dynamical systems}, IMA J.
  Numer. Anal. 10~(3) (1990) 379--405.
\newblock \href {https://doi.org/10.1093/imanum/10.3.379}
  {\path{doi:10.1093/imanum/10.3.379}}.
\newline\urlprefix\url{http://dx.doi.org.proxy.libraries.rutgers.edu/10.1093/imanum/10.3.379}

\bibitem{MR1007358}
E.~J. Doedel, M.~J. Friedman,
  \href{http://dx.doi.org.proxy.libraries.rutgers.edu/10.1016/0377-0427(89)90153-2}{Numerical
  computation of heteroclinic orbits}, J. Comput. Appl. Math. 26~(1-2) (1989)
  155--170, continuation techniques and bifurcation problems.
\newblock \href {https://doi.org/10.1016/0377-0427(89)90153-2}
  {\path{doi:10.1016/0377-0427(89)90153-2}}.
\newline\urlprefix\url{http://dx.doi.org.proxy.libraries.rutgers.edu/10.1016/0377-0427(89)90153-2}

\bibitem{MR1205453}
M.~J. Friedman, E.~J. Doedel,
  \href{http://dx.doi.org.proxy.libraries.rutgers.edu/10.1007/BF01063734}{Computational
  methods for global analysis of homoclinic and heteroclinic orbits: a case
  study}, J. Dynam. Differential Equations 5~(1) (1993) 37--57.
\newblock \href {https://doi.org/10.1007/BF01063734}
  {\path{doi:10.1007/BF01063734}}.
\newline\urlprefix\url{http://dx.doi.org.proxy.libraries.rutgers.edu/10.1007/BF01063734}

\bibitem{MR1456497}
E.~J. Doedel, M.~J. Friedman, B.~I. Kunin, Successive continuation for locating
  connecting orbits, Numer. Algorithms 14~(1-3) (1997) 103--124, dynamical
  numerical analysis (Atlanta, GA, 1995).

\bibitem{MR2003792}
F.~J. Mu\~{n}oz Almaraz, E.~Freire, J.~Gal\'{a}n, E.~Doedel, A.~Vanderbauwhede,
  \href{https://doi-org.ezproxy.fau.edu/10.1016/S0167-2789(03)00097-6}{Continuation
  of periodic orbits in conservative and {H}amiltonian systems}, Phys. D
  181~(1-2) (2003) 1--38.
\newblock \href {https://doi.org/10.1016/S0167-2789(03)00097-6}
  {\path{doi:10.1016/S0167-2789(03)00097-6}}.
\newline\urlprefix\url{https://doi-org.ezproxy.fau.edu/10.1016/S0167-2789(03)00097-6}

\bibitem{MR2454068}
E.~J. Doedel, B.~W. Kooi, G.~A.~K. van Voorn, Y.~A. Kuznetsov,
  \href{http://dx.doi.org.proxy.libraries.rutgers.edu/10.1142/S0218127408021439}{Continuation
  of connecting orbits in 3{D}-{ODE}s. {I}. {P}oint-to-cycle connections},
  Internat. J. Bifur. Chaos Appl. Sci. Engrg. 18~(7) (2008) 1889--1903.
\newblock \href {https://doi.org/10.1142/S0218127408021439}
  {\path{doi:10.1142/S0218127408021439}}.
\newline\urlprefix\url{http://dx.doi.org.proxy.libraries.rutgers.edu/10.1142/S0218127408021439}

\bibitem{MR2511084}
E.~J. Doedel, B.~W. Kooi, G.~A.~K. Van~Voorn, Y.~A. Kuznetsov,
  \href{http://dx.doi.org/10.1142/S0218127409022804}{Continuation of connecting
  orbits in 3{D}-{ODE}s. {II}. {C}ycle-to-cycle connections}, Internat. J.
  Bifur. Chaos Appl. Sci. Engrg. 19~(1) (2009) 159--169.
\newblock \href {https://doi.org/10.1142/S0218127409022804}
  {\path{doi:10.1142/S0218127409022804}}.
\newline\urlprefix\url{http://dx.doi.org/10.1142/S0218127409022804}

\bibitem{MR2989589}
R.~C. Calleja, E.~J. Doedel, A.~R. Humphries, A.~Lemus-Rodr\'{\i}guez, E.~B.
  Oldeman,
  \href{https://doi-org.ezproxy.fau.edu/10.1007/s10569-012-9434-y}{Boundary-value
  problem formulations for computing invariant manifolds and connecting orbits
  in the circular restricted three body problem}, Celestial Mech. Dynam.
  Astronom. 114~(1-2) (2012) 77--106.
\newblock \href {https://doi.org/10.1007/s10569-012-9434-y}
  {\path{doi:10.1007/s10569-012-9434-y}}.
\newline\urlprefix\url{https://doi-org.ezproxy.fau.edu/10.1007/s10569-012-9434-y}

\bibitem{MR1867240}
G.~G\'{o}mez, J.~Llibre, R.~Mart\'{\i}nez, C.~Sim\'{o},
  \href{https://doi-org.ezproxy.fau.edu/10.1142/9789812810632_bmatter}{Dynamics
  and mission design near libration points. {V}ol. {I}}, Vol.~2 of World
  Scientific Monograph Series in Mathematics, World Scientific Publishing Co.,
  Inc., River Edge, NJ, 2001, fundamentals: the case of collinear libration
  points, With a foreword by Walter Flury.
\newblock \href {https://doi.org/10.1142/9789812810632_bmatter}
  {\path{doi:10.1142/9789812810632_bmatter}}.
\newline\urlprefix\url{https://doi-org.ezproxy.fau.edu/10.1142/9789812810632_bmatter}

\bibitem{MR1881823}
G.~G\'{o}mez, C.~Sim\'{o}, J.~Llibre, R.~Mart\'{\i}nez, Dynamics and mission
  design near libration points. {V}ol. {II}, Vol.~3 of World Scientific
  Monograph Series in Mathematics, World Scientific Publishing Co., Inc., River
  Edge, NJ, 2001, fundamentals: the case of triangular libration points.

\bibitem{MR1878993}
G.~G\'{o}mez, A.~Jorba, C.~Sim\'{o}, J.~Masdemont, Dynamics and mission design
  near libration points. {V}ol. {III}, Vol.~4 of World Scientific Monograph
  Series in Mathematics, World Scientific Publishing Co., Inc., River Edge, NJ,
  2001, advanced methods for collinear points.

\bibitem{MR1875754}
G.~G\'{o}mez, A.~Jorba, C.~Sim\'{o}, J.~Masdemont, Dynamics and mission design
  near libration points. {V}ol. {IV}, Vol.~5 of World Scientific Monograph
  Series in Mathematics, World Scientific Publishing Co., Inc., River Edge, NJ,
  2001, advanced methods for triangular points.

\bibitem{MR1870302}
W.~S. Koon, M.~W. Lo, J.~E. Marsden, S.~D. Ross, Dynamical systems, the
  three-body problem and space mission design, in: International {C}onference
  on {D}ifferential {E}quations, {V}ol. 1, 2 ({B}erlin, 1999), World Sci.
  Publ., River Edge, NJ, 2000, pp. 1167--1181.

\bibitem{MR2086140}
G.~G\'omez, W.~S. Koon, M.~W. Lo, J.~E. Marsden, J.~Masdemont, S.~D. Ross,
  \href{https://doi-org.ezproxy.fau.edu/10.1088/0951-7715/17/5/002}{Connecting
  orbits and invariant manifolds in the spatial restricted three-body problem},
  Nonlinearity 17~(5) (2004) 1571--1606.
\newline\urlprefix\url{https://doi-org.ezproxy.fau.edu/10.1088/0951-7715/17/5/002}

\bibitem{MR1884895}
W.~S. Koon, J.~E. Marsden, S.~D. Ross, M.~W. Lo,
  \href{https://doi-org.ezproxy.fau.edu/10.1090/conm/292/04919}{Constructing a
  low energy transfer between {J}ovian moons}, in: Celestial mechanics
  ({E}vanston, {IL}, 1999), Vol. 292 of Contemp. Math., Amer. Math. Soc.,
  Providence, RI, 2002, pp. 129--145.
\newline\urlprefix\url{https://doi-org.ezproxy.fau.edu/10.1090/conm/292/04919}

\bibitem{Barrabes:2009ve}
E.~Barrab{\'e}s, J.~M. Mondelo, M.~Oll{\'e},
  \href{https://dx.doi.org/10.1088/0951-7715/22/12/006}{Numerical continuation
  of families of homoclinic connections of periodic orbits in the rtbp} 22~(12)
  (2009) 2901.
\newblock \href {https://doi.org/10.1088/0951-7715/22/12/006}
  {\path{doi:10.1088/0951-7715/22/12/006}}.
\newline\urlprefix\url{https://dx.doi.org/10.1088/0951-7715/22/12/006}

\bibitem{Barrabes:2013ws}
E.~Barrab{\'e}s, J.~M. Mondelo, M.~Oll{\'e},
  \href{https://dx.doi.org/10.1088/0951-7715/26/10/2747}{Numerical continuation
  of families of heteroclinic connections between periodic orbits in a
  hamiltonian system} 26~(10) (2013) 2747.
\newblock \href {https://doi.org/10.1088/0951-7715/26/10/2747}
  {\path{doi:10.1088/0951-7715/26/10/2747}}.
\newline\urlprefix\url{https://dx.doi.org/10.1088/0951-7715/26/10/2747}

\bibitem{Kumar:2021vc}
B.~Kumar, R.~L. Anderson, R.~de~la Llave,
  \href{https://www.sciencedirect.com/science/article/pii/S1007570421000022}{High-order
  resonant orbit manifold expansions for mission design in the planar circular
  restricted 3-body problem}, Communications in Nonlinear Science and Numerical
  Simulation 97 (2021) 105691.
\newblock \href {https://doi.org/https://doi.org/10.1016/j.cnsns.2021.105691}
  {\path{doi:https://doi.org/10.1016/j.cnsns.2021.105691}}.
\newline\urlprefix\url{https://www.sciencedirect.com/science/article/pii/S1007570421000022}

\bibitem{MR4361879}
B.~Kumar, R.~L. Anderson, R.~de~la Llave,
  \href{https://doi.org/10.1007/s10569-021-10057-1}{Rapid and accurate methods
  for computing whiskered tori and their manifolds in periodically perturbed
  planar circular restricted 3-body problems}, Celestial Mech. Dynam. Astronom.
  134~(1) (2022) Paper No. 3, 38.
\newblock \href {https://doi.org/10.1007/s10569-021-10057-1}
  {\path{doi:10.1007/s10569-021-10057-1}}.
\newline\urlprefix\url{https://doi.org/10.1007/s10569-021-10057-1}

\bibitem{https://doi.org/10.48550/arxiv.2301.08526}
M.~Barcelona, A.~Haro, J.-M. Mondelo,
  \href{https://arxiv.org/abs/2301.08526}{Semi-analytical computation of
  heteroclinic connections between center manifolds with the parameterization
  method} (2023).
\newblock \href {https://doi.org/10.48550/ARXIV.2301.08526}
  {\path{doi:10.48550/ARXIV.2301.08526}}.
\newline\urlprefix\url{https://arxiv.org/abs/2301.08526}

\bibitem{MR2136742}
M.~Dellnitz, O.~Junge, W.~S. Koon, F.~Lekien, M.~W. Lo, J.~E. Marsden,
  K.~Padberg, R.~Preis, S.~D. Ross, B.~Thiere,
  \href{https://doi-org.ezproxy.fau.edu/10.1142/S0218127405012545}{Transport in
  dynamical astronomy and multibody problems}, Internat. J. Bifur. Chaos Appl.
  Sci. Engrg. 15~(3) (2005) 699--727.
\newline\urlprefix\url{https://doi-org.ezproxy.fau.edu/10.1142/S0218127405012545}

\bibitem{MR2112702}
G.~Arioli,
  \href{http://dx.doi.org.ezproxy.fau.edu/10.3934/dcds.2004.11.745}{Branches of
  periodic orbits for the planar restricted 3-body problem}, Discrete Contin.
  Dyn. Syst. 11~(4) (2004) 745--755.
\newblock \href {https://doi.org/10.3934/dcds.2004.11.745}
  {\path{doi:10.3934/dcds.2004.11.745}}.
\newline\urlprefix\url{http://dx.doi.org.ezproxy.fau.edu/10.3934/dcds.2004.11.745}

\bibitem{MR2259202}
G.~Arioli, V.~Barutello, S.~Terracini,
  \href{http://dx.doi.org.ezproxy.fau.edu/10.1007/s00220-006-0111-4}{A new
  branch of {M}ountain {P}ass solutions for the choreographical 3-body
  problem}, Comm. Math. Phys. 268~(2) (2006) 439--463.
\newblock \href {https://doi.org/10.1007/s00220-006-0111-4}
  {\path{doi:10.1007/s00220-006-0111-4}}.
\newline\urlprefix\url{http://dx.doi.org.ezproxy.fau.edu/10.1007/s00220-006-0111-4}

\bibitem{MR2012847}
T.~Kapela, P.~Zgliczy\'nski,
  \href{http://dx.doi.org.ezproxy.fau.edu/10.1088/0951-7715/16/6/302}{The
  existence of simple choreographies for the {$N$}-body problem---a
  computer-assisted proof}, Nonlinearity 16~(6) (2003) 1899--1918.
\newblock \href {https://doi.org/10.1088/0951-7715/16/6/302}
  {\path{doi:10.1088/0951-7715/16/6/302}}.
\newline\urlprefix\url{http://dx.doi.org.ezproxy.fau.edu/10.1088/0951-7715/16/6/302}

\bibitem{MR2185163}
T.~Kapela,
  \href{http://dx.doi.org.ezproxy.fau.edu/10.1142/9789812702067_0165}{{$N$}-body
  choreographies with a reflectional symmetry---computer assisted existence
  proofs}, in: E{QUADIFF} 2003, World Sci. Publ., Hackensack, NJ, 2005, pp.
  999--1004.
\newblock \href {https://doi.org/10.1142/9789812702067_0165}
  {\path{doi:10.1142/9789812702067_0165}}.
\newline\urlprefix\url{http://dx.doi.org.ezproxy.fau.edu/10.1142/9789812702067_0165}

\bibitem{MR3622273}
T.~Kapela, C.~Sim\'o,
  \href{http://dx.doi.org.ezproxy.fau.edu/10.1088/1361-6544/aa4ff3}{Rigorous
  {KAM} results around arbitrary periodic orbits for {H}amiltonian systems},
  Nonlinearity 30~(3) (2017) 965--986.
\newblock \href {https://doi.org/10.1088/1361-6544/aa4ff3}
  {\path{doi:10.1088/1361-6544/aa4ff3}}.
\newline\urlprefix\url{http://dx.doi.org.ezproxy.fau.edu/10.1088/1361-6544/aa4ff3}

\bibitem{MR2312391}
T.~Kapela, C.~Sim{\'o}, Computer assisted proofs for nonsymmetric planar
  choreographies and for stability of the {E}ight, Nonlinearity 20~(5) (2007)
  1241--1255.
\newblock \href {https://doi.org/10.1088/0951-7715/20/5/010}
  {\path{doi:10.1088/0951-7715/20/5/010}}.

\bibitem{MR3896998}
J.~Burgos-Garc\'{\i}a, J.-P. Lessard, J.~D. Mireles~James,
  \href{https://doi-org.ezproxy.fau.edu/10.1007/s10569-018-9879-8}{Spatial
  periodic orbits in the equilateral circular restricted four-body problem:
  computer-assisted proofs of existence}, Celestial Mech. Dynam. Astronom.
  131~(1) (2019) Art. 2, 36.
\newblock \href {https://doi.org/10.1007/s10569-018-9879-8}
  {\path{doi:10.1007/s10569-018-9879-8}}.
\newline\urlprefix\url{https://doi-org.ezproxy.fau.edu/10.1007/s10569-018-9879-8}

\bibitem{MR4208440}
R.~Calleja, C.~Garc\'{\i}a-Azpeitia, J.-P. Lessard, J.~D. Mireles~James,
  \href{https://doi.org/10.1088/1361-6544/abcb08}{Torus knot choreographies in
  the {$n$}-body problem}, Nonlinearity 34~(1) (2021) 313--349.
\newblock \href {https://doi.org/10.1088/1361-6544/abcb08}
  {\path{doi:10.1088/1361-6544/abcb08}}.
\newline\urlprefix\url{https://doi.org/10.1088/1361-6544/abcb08}

\bibitem{MR3923486}
I.~Walawska, D.~Wilczak,
  \href{https://doi.org/10.1016/j.cnsns.2019.03.005}{Validated numerics for
  period-tupling and touch-and-go bifurcations of symmetric periodic orbits in
  reversible systems}, Commun. Nonlinear Sci. Numer. Simul. 74 (2019) 30--54.
\newblock \href {https://doi.org/10.1016/j.cnsns.2019.03.005}
  {\path{doi:10.1016/j.cnsns.2019.03.005}}.
\newline\urlprefix\url{https://doi.org/10.1016/j.cnsns.2019.03.005}

\bibitem{MR1947690}
G.~Arioli,
  \href{http://dx.doi.org.ezproxy.fau.edu/10.1007/s00220-002-0666-7}{Periodic
  orbits, symbolic dynamics and topological entropy for the restricted 3-body
  problem}, Comm. Math. Phys. 231~(1) (2002) 1--24.
\newblock \href {https://doi.org/10.1007/s00220-002-0666-7}
  {\path{doi:10.1007/s00220-002-0666-7}}.
\newline\urlprefix\url{http://dx.doi.org.ezproxy.fau.edu/10.1007/s00220-002-0666-7}

\bibitem{MR1961956}
D.~Wilczak, P.~Zgliczynski,
  \href{http://dx.doi.org.proxy.libraries.rutgers.edu/10.1007/s00220-002-0709-0}{Heteroclinic
  connections between periodic orbits in planar restricted circular three-body
  problem---a computer assisted proof}, Comm. Math. Phys. 234~(1) (2003)
  37--75.
\newblock \href {https://doi.org/10.1007/s00220-002-0709-0}
  {\path{doi:10.1007/s00220-002-0709-0}}.
\newline\urlprefix\url{http://dx.doi.org.proxy.libraries.rutgers.edu/10.1007/s00220-002-0709-0}

\bibitem{MR3032848}
M.~J. Capi{\'n}ski, \href{http://dx.doi.org/10.1137/110847366}{Computer
  assisted existence proofs of {L}yapunov orbits at {$L_2$} and transversal
  intersections of invariant manifolds in the {J}upiter-{S}un {PCR}3{BP}}, SIAM
  J. Appl. Dyn. Syst. 11~(4) (2012) 1723--1753.
\newblock \href {https://doi.org/10.1137/110847366}
  {\path{doi:10.1137/110847366}}.
\newline\urlprefix\url{http://dx.doi.org/10.1137/110847366}

\bibitem{MR3906230}
S.~Kepley, J.~D. Mireles~James,
  \href{https://doi-org.ezproxy.fau.edu/10.1016/j.jde.2018.08.007}{Chaotic
  motions in the restricted four body problem via {D}evaney's saddle-focus
  homoclinic tangle theorem}, J. Differential Equations 266~(4) (2019)
  1709--1755.
\newblock \href {https://doi.org/10.1016/j.jde.2018.08.007}
  {\path{doi:10.1016/j.jde.2018.08.007}}.
\newline\urlprefix\url{https://doi-org.ezproxy.fau.edu/10.1016/j.jde.2018.08.007}

\bibitem{MR2824484}
J.~Galante, V.~Kaloshin,
  \href{https://doi.org/10.1215/00127094-1415878}{Destruction of invariant
  curves in the restricted circular planar three-body problem by using
  comparison of action}, Duke Math. J. 159~(2) (2011) 275--327.
\newblock \href {https://doi.org/10.1215/00127094-1415878}
  {\path{doi:10.1215/00127094-1415878}}.
\newline\urlprefix\url{https://doi.org/10.1215/00127094-1415878}

\bibitem{oscillations}
M.~Capi\'{n}ski, M.~Guardia, P.~Mart\'{\i}n, T.~M. Seara, P.~Zgliczy\'nski,
  Oscillatroy motions and parabolic manifolds at infinitey in the planar
  circular restricted three body problem, Journal of Differential Equations
  320~(3) (2022) 316--370.

\bibitem{capinski_roldan}
M.~Capi{\'n}ski, P.~Rold{\'a}n, Existence of a center manifold in a practical
  domain around ${L}_1$ in the restricted three body problem, SIAM J. Appl.
  Dyn. Syst. 11~(1) (2011) 285--318.

\bibitem{diffusionCRTBP}
M.~Capi\'{n}ski, N.~Wodka, Computer assisted proof of drift orbits along
  normally hyperbolic manifolds ii: applications to the restricted three body
  problem, (Submitted)) (2021).

\bibitem{maciejMarianDiffusion}
M.~Capi\'{n}ski, M.~Gidea, Arnold diffusion, quantitative estimates and
  stochastic behavior in the three-body problem, Communications on Pure and
  Applied Mathematics (2018).

\bibitem{MR1101365}
A.~Celletti, L.~Chierchia,
  \href{https://doi-org.ezproxy.fau.edu/10.1007/978-1-4613-9092-3_6}{A
  computer-assisted approach to small-divisors problems arising in
  {H}amiltonian mechanics}, in: Computer aided proofs in analysis
  ({C}incinnati, {OH}, 1989), Vol.~28 of IMA Vol. Math. Appl., Springer, New
  York, 1991, pp. 43--51.
\newline\urlprefix\url{https://doi-org.ezproxy.fau.edu/10.1007/978-1-4613-9092-3_6}

\bibitem{MR1101369}
R.~de~la Llave, D.~Rana,
  \href{http://dx.doi.org/10.1007/978-1-4613-9092-3_12}{Accurate strategies for
  {K}.{A}.{M}. bounds and their implementation}, in: Computer aided proofs in
  analysis ({C}incinnati, {OH}, 1989), Vol.~28 of IMA Vol. Math. Appl.,
  Springer, New York, 1991, pp. 127--146.
\newblock \href {https://doi.org/10.1007/978-1-4613-9092-3_12}
  {\path{doi:10.1007/978-1-4613-9092-3_12}}.
\newline\urlprefix\url{http://dx.doi.org/10.1007/978-1-4613-9092-3_12}

\bibitem{alexCAPKAM}
J.-L. Figueras, A.~Haro, A.~Luque, Rigorous computer assisted application of
  kam theory: a modern approach, (Submitted) arXiv:1601.00084 [math.DS] (2016).

\bibitem{MR2150352}
F.~Gabern, A.~Jorba, U.~Locatelli,
  \href{https://doi.org/10.1088/0951-7715/18/4/017}{On the construction of the
  {K}olmogorov normal form for the {T}rojan asteroids}, Nonlinearity 18~(4)
  (2005) 1705--1734.
\newblock \href {https://doi.org/10.1088/0951-7715/18/4/017}
  {\path{doi:10.1088/0951-7715/18/4/017}}.
\newline\urlprefix\url{https://doi.org/10.1088/0951-7715/18/4/017}

\bibitem{MR4128817}
C.~Caracciolo, U.~Locatelli,
  \href{https://doi.org/10.3934/jcd.2020017}{Computer-assisted estimates for
  {B}irkhoff normal forms}, J. Comput. Dyn. 7~(2) (2020) 425--460.
\newblock \href {https://doi.org/10.3934/jcd.2020017}
  {\path{doi:10.3934/jcd.2020017}}.
\newline\urlprefix\url{https://doi.org/10.3934/jcd.2020017}

\bibitem{MR759197}
O.~E. Lanford, III,
  \href{http://dx.doi.org/10.1016/0378-4371(84)90262-0}{Computer-assisted
  proofs in analysis}, Phys. A 124~(1-3) (1984) 465--470, mathematical physics,
  VII (Boulder, Colo., 1983).
\newblock \href {https://doi.org/10.1016/0378-4371(84)90262-0}
  {\path{doi:10.1016/0378-4371(84)90262-0}}.
\newline\urlprefix\url{http://dx.doi.org/10.1016/0378-4371(84)90262-0}

\bibitem{MR1420838}
H.~Koch, A.~Schenkel, P.~Wittwer,
  \href{http://dx.doi.org/10.1137/S0036144595284180}{Computer-assisted proofs
  in analysis and programming in logic: a case study}, SIAM Rev. 38~(4) (1996)
  565--604.
\newblock \href {https://doi.org/10.1137/S0036144595284180}
  {\path{doi:10.1137/S0036144595284180}}.
\newline\urlprefix\url{http://dx.doi.org/10.1137/S0036144595284180}

\bibitem{jpjbReview}
J.~B. van~den Berg, J.~Lessard, Rigorous numerics in dynamics, Notices Amer.
  Math. Soc. 62~(9) (2015) 1057--1061.

\bibitem{MR3990999}
J.~G\'{o}mez-Serrano,
  \href{https://doi.org/10.1007/s40324-019-00186-x}{Computer-assisted proofs in
  {PDE}: a survey}, SeMA J. 76~(3) (2019) 459--484.
\newblock \href {https://doi.org/10.1007/s40324-019-00186-x}
  {\path{doi:10.1007/s40324-019-00186-x}}.
\newline\urlprefix\url{https://doi.org/10.1007/s40324-019-00186-x}

\bibitem{MR2807595}
W.~Tucker, Validated numerics, Princeton University Press, Princeton, NJ, 2011,
  a short introduction to rigorous computations.

\bibitem{MR3971222}
M.~T. Nakao, M.~Plum, Y.~Watanabe,
  \href{https://doi.org/10.1007/978-981-13-7669-6}{Numerical verification
  methods and computer-assisted proofs for partial differential equations},
  Vol.~53 of Springer Series in Computational Mathematics, Springer, Singapore,
  [2019] \copyright 2019.
\newblock \href {https://doi.org/10.1007/978-981-13-7669-6}
  {\path{doi:10.1007/978-981-13-7669-6}}.
\newline\urlprefix\url{https://doi.org/10.1007/978-981-13-7669-6}

\bibitem{CAPD_paper}
T.~Kapela, M.~Mrozek, D.~Wilczak, P.~Zgliczy{\'n}ski,
  \href{https://www.sciencedirect.com/science/article/pii/S1007570420304081}{C{A}{P}{D}::{D}yn{S}ys:
  A flexible c++ toolbox for rigorous numerical analysis of dynamical systems},
  Communications in Nonlinear Science and Numerical Simulation 101 (2021)
  105578.
\newblock \href {https://doi.org/https://doi.org/10.1016/j.cnsns.2020.105578}
  {\path{doi:https://doi.org/10.1016/j.cnsns.2020.105578}}.
\newline\urlprefix\url{https://www.sciencedirect.com/science/article/pii/S1007570420304081}

\bibitem{MR1930946}
P.~Zgliczynski, {$C^1$} {L}ohner algorithm, Found. Comput. Math. 2~(4) (2002)
  429--465.
\newblock \href {https://doi.org/10.1007/s102080010025}
  {\path{doi:10.1007/s102080010025}}.

\bibitem{cnLohner}
D.~Wilczak, P.~Zgliczynski, $c^n$-lohner algorithm, Scheade Informaticae 20
  (2011) 9--46.

\bibitem{MR3792792}
J.~D. Mireles~James, Validated numerics for equilibria of analytic vector
  fields: invariant manifolds and connecting orbits, in: Rigorous numerics in
  dynamics, Vol.~74 of Proc. Sympos. Appl. Math., Amer. Math. Soc., Providence,
  RI, 2018, pp. 27--80.

\bibitem{jayAndMaxime}
J.~D.~M. James, M.~Murray, \href{https://arxiv.org/abs/2212.00930}{Computer
  assisted proof of homoclinic chaos in the spatial equilateral restricted four
  body problem} (2022).
\newblock \href {https://doi.org/10.48550/ARXIV.2212.00930}
  {\path{doi:10.48550/ARXIV.2212.00930}}.
\newline\urlprefix\url{https://arxiv.org/abs/2212.00930}

\bibitem{MR1870260}
F.~J. Mu\~{n}oz Almaraz, J.~Gal\'{a}n, E.~Freire, Numerical continuation of
  periodic orbits in symmetric {H}amiltonian systems, in: International
  {C}onference on {D}ifferential {E}quations, {V}ol. 1, 2 ({B}erlin, 1999),
  World Sci. Publ., River Edge, NJ, 2000, pp. 919--921.

\bibitem{MR1992054}
E.~J. Doedel, R.~C. Paffenroth, H.~B. Keller, D.~J. Dichmann,
  J.~Gal\'{a}n-Vioque, A.~Vanderbauwhede,
  \href{https://doi-org.ezproxy.fau.edu/10.1142/S0218127403007291}{Computation
  of periodic solutions of conservative systems with application to the 3-body
  problem}, Internat. J. Bifur. Chaos Appl. Sci. Engrg. 13~(6) (2003)
  1353--1381.
\newblock \href {https://doi.org/10.1142/S0218127403007291}
  {\path{doi:10.1142/S0218127403007291}}.
\newline\urlprefix\url{https://doi-org.ezproxy.fau.edu/10.1142/S0218127403007291}

\bibitem{BDLM}
J.~B. van~den Berg, A.~Desch\^enes, J.-P. Lessard, J.~D. Mireles~James,
  Stationary coexistence of hexagons and rolls via rigorous computations, SIAM
  Journal on Applied Dynamical Systems 14~(2) (2015) 942--979.

\bibitem{MR3919451}
S.~Kepley, J.~D. Mireles~James,
  \href{https://doi-org.ezproxy.fau.edu/10.1007/s10569-019-9890-8}{Homoclinic
  dynamics in a restricted four-body problem: transverse connections for the
  saddle-focus equilibrium solution set}, Celestial Mech. Dynam. Astronom.
  131~(3) (2019) Paper No. 13, 55.
\newblock \href {https://doi.org/10.1007/s10569-019-9890-8}
  {\path{doi:10.1007/s10569-019-9890-8}}.
\newline\urlprefix\url{https://doi-org.ezproxy.fau.edu/10.1007/s10569-019-9890-8}

\bibitem{MR0231516}
R.~E. Moore, Interval analysis, Prentice-Hall Inc., Englewood Cliffs, N.J.,
  1966.

\bibitem{MR1100928}
A.~Neumaier, Interval methods for systems of equations, Vol.~37 of Encyclopedia
  of Mathematics and its Applications, Cambridge University Press, Cambridge,
  1990.

\bibitem{MR1057685}
A.~Neumaier, S.~Zu~He, \href{https://doi.org/10.1016/0022-247X(90)90053-I}{The
  {K}rawczyk operator and {K}antorovich's theorem}, J. Math. Anal. Appl.
  149~(2) (1990) 437--443.
\newblock \href {https://doi.org/10.1016/0022-247X(90)90053-I}
  {\path{doi:10.1016/0022-247X(90)90053-I}}.
\newline\urlprefix\url{https://doi.org/10.1016/0022-247X(90)90053-I}

\bibitem{MR2652784}
S.~M. Rump, \href{http://dx.doi.org/10.1017/S096249291000005X}{Verification
  methods: rigorous results using floating-point arithmetic}, Acta Numer. 19
  (2010) 287--449.
\newblock \href {https://doi.org/10.1017/S096249291000005X}
  {\path{doi:10.1017/S096249291000005X}}.
\newline\urlprefix\url{http://dx.doi.org/10.1017/S096249291000005X}

\bibitem{MR3822720}
J.~B. van~den Berg, J.-P. Lessard (Eds.),
  \href{https://doi.org/10.1090/psapm/074}{Rigorous numerics in dynamics},
  Vol.~74 of Proceedings of Symposia in Applied Mathematics, American
  Mathematical Society, Providence, RI, 2018, aMS Short Course: Rigorous
  Numerics in Dynamics, January 4--5, 2016, Seattle, Washington.
\newblock \href {https://doi.org/10.1090/psapm/074}
  {\path{doi:10.1090/psapm/074}}.
\newline\urlprefix\url{https://doi.org/10.1090/psapm/074}

\bibitem{c1Lohner}
P.~Zgliczynski, \href{https://doi.org/10.1007/s102080010025}{{$C^1$} {L}ohner
  algorithm}, Found. Comput. Math. 2~(4) (2002) 429--465.
\newblock \href {https://doi.org/10.1007/s102080010025}
  {\path{doi:10.1007/s102080010025}}.
\newline\urlprefix\url{https://doi.org/10.1007/s102080010025}

\bibitem{myNotes}
J.~D. Mireles~James, Validated numerics for equilibriua of analytic vector
  fields: invariant manifolds and connecting orbits, (To appear in AMS lecture
  notes - winter short course series) (2017) 1--55.

\bibitem{Ru99a}
S.~Rump, {INTLAB - INTerval LABoratory}, in: T.~Csendes (Ed.),
  {Developments~in~Reliable Computing}, Kluwer Academic Publishers, Dordrecht,
  1999, pp. 77--104, http://www.ti3.tu-harburg.de/rump/.

\bibitem{MR2177465}
X.~Cabr{\'e}, E.~Fontich, R.~de~la Llave, The parameterization method for
  invariant manifolds. {III}. {O}verview and applications, J. Differential
  Equations 218~(2) (2005) 444--515.

\bibitem{mamotreto}
A.~Haro, M.~Canadell, J.-L.~s. Figueras, A.~Luque, J.-M. Mondelo,
  \href{http://dx.doi.org.ezproxy.fau.edu/10.1007/978-3-319-29662-3}{The
  parameterization method for invariant manifolds}, Vol. 195 of Applied
  Mathematical Sciences, Springer, [Cham], 2016, from rigorous results to
  effective computations.
\newblock \href {https://doi.org/10.1007/978-3-319-29662-3}
  {\path{doi:10.1007/978-3-319-29662-3}}.
\newline\urlprefix\url{http://dx.doi.org.ezproxy.fau.edu/10.1007/978-3-319-29662-3}

\end{thebibliography}
\end{document}